\documentclass[11pt]{amsart}
\usepackage{a4wide}
\usepackage{verbatim}
\usepackage[T1]{fontenc}
\usepackage[utf8]{inputenc}
\usepackage{amsfonts}
\usepackage{amsthm}
\usepackage{amsmath}
\usepackage{bm}
\usepackage{amssymb,amsbsy,amsopn}
\usepackage{mathrsfs}
\usepackage{graphicx,color,float,subcaption,caption,algorithm}
\usepackage{mathptmx,bbm}
\usepackage[bookmarks=true,bookmarksnumbered=true,bookmarkstype=toc,colorlinks, linkcolor={blue}, citecolor={blue}, urlcolor={blue}]{hyperref}
\usepackage{dsfont}
\usepackage{leftidx}
\usepackage{empheq}
\usepackage[shortlabels]{enumitem}
\usepackage[parfill]{parskip}
\usepackage{setspace}
\usepackage{xfrac}
\usepackage{esvect}
\usepackage{multirow}
\usepackage{hhline}
\usepackage{enumitem}  
\allowdisplaybreaks

\usepackage{tikz}
\usetikzlibrary{arrows,automata,shapes}
\usetikzlibrary{positioning}
\newtheorem{definition}{Definition}[section] 
\newtheorem{theorem}{Theorem}[section]

\newtheorem{lemma}[theorem]{Lemma}
\newtheorem{remark}[theorem]{Remark}

\newtheorem{assumption}[theorem]{Assumption}

\newcommand{\ie}{\setlength{\parskip}{0cm} \setlength{\itemsep}{0cm}}
\DeclareMathAlphabet{\mathcal}{OMS}{cmsy}{m}{n}

\usepackage{amscd}

\newcommand{\la}{\left(}
\newcommand{\ra}{\right)}
\newcommand{\lb}{\left\langle}
\newcommand{\rb}{\right\rangle}
\def\wh{\widehat}
\def\ol{\overline}
\def\E{\mathbb E}
\def\P{\mathbb P}
\def\b{\mathbb}
\def\bal{\begin{equation*}\begin{aligned}} 
		\def\eal{\end{aligned}\end{equation*}}
\def\be{\begin{equation}\label}
	\def\ee{\end{equation}}
\def\bd{\begin{definition}\label}
	\def\ed{\end{definition}}
\def\bt{\begin{theorem}\label}
	\def\et{\end{theorem}}
\def\bl{\begin{lemma}\label}
	\def\el{\end{lemma}}

\def\R{\mathbb R}
\def\Z{\mathbb Z}
\def\L{\mathbb L}
\def\H{\mathbb H}
\DeclareMathOperator*{\essup}{ess\,sup}

\date{}
\begin{document}
\title[Landau-Lifshitz-Slonczewski equation]
{Existence, uniqueness and regularity of solutions to the stochastic Landau-Lifschitz-Slonczewski equation}
\author{Beniamin Goldys}
\address{School of Mathematics and Statistics, The University of Sydney, Sydney 2006, Australia}
\email{Beniamin.Goldys@sydney.edu.au}
\author{Chunxi Jiao}
\address{School of Mathematics and Statistics, The University of Sydney, Sydney 2006, Australia}
\email{chunxi.jiao@sydney.edu.au}
\author{Kim Ngan Le}
\address{School of MAthematics, Monash University, VIC 3800, Australia}
\email{ngan.le@monash.edu}
\thanks{This  work was 
supported by  the Australian Research Council Project DP200101866. 
} 
\keywords{nanowire, spin transfer torque, stochastic Landau-Lifshitz-Slonczewski equation, transport noise, pathwise solutions}
\subjclass{35D30, 35K55, 35K59, 35Q49, 35Q60, 60H15}
\begin{abstract}
In this paper we are concerned with the stochastic Landau-Lifshitz-Slonczewski equation (LLS) that describes magnetisation of an inifnite nanowire evolving under current driven spin torque. The current brings into the system a multiplicative gradient noise that appears as a transport term in the equation. We prove the existence, uniqueness and regularity of pathwise solutions to the equation. 
\end{abstract}
\maketitle
\tableofcontents



\section{Introduction}
	In this paper we are concerned with the existence, uniqueness and regularity of solutions to the stochastic Landau-Lifshitz-Slonczewski equation (LLS) equation considered on real line, see \eqref{eq_spde}.  To the best of our knowledge , this is the first  result on a system of stochastic PDEs that combines variational structure with the transport noise in the presence of geometric constraints, more precisely, with solutions taking values in a sphere. 
	\par
	Let us recall briefly the physical motivation for LLS equation, see \cite{bertotti,mp} for more details. 
	A deterministic version of equation \eqref{eq_spde} was introduced in \cite{slon} in order to study the magnetisation dynamics of ferromagnetic elements in presence of electric current. If the ferromagnetic element is small enough (100 nanometers) then the interaction between the electric current and the magnetisation results in the current-induced magnetisation switching and spin wave emission. It is expected that good understanding of those effects will allow us to develop new types of current-controlled magnetic memories and current controlled magnetic oscillators. 
	\par
	Mathematical theory of the deterministic LLS equation is at an early stage. The case, when the ferromagnetic material fills in a 3-dimensional domain is studied in an important paper \cite{mp}, where the existence and uniqueness of solutions is proved and their regularity is studied. A physically important case of a nanowire is a subject of ongoing intense research in physics. Mathematical analysis of dynamics of travelling domain walls and their stability was only recently initiated in \cite{mel-rad,rad,siemer}. 
	\par
	The necessity to include random fluctuations (such as thermal noise) into the dynamics of magnetisation has been conjectured by physicists for many years (see for example \cite{bertotti,brown, neel}. The existence and uniqueness of solutions for the Landau-Lifshitz equation without the Slonczewski term but including random fluctuations was intensely studied in recent years  \cite{Brzezniak2013_sLLG,ngan1,ngan2}. 
	\par
	Let us first recall briefly the formulation of the deterministic LLS equation. We will identify an infinite nanowire made of ferromagnetic material with a real line $\R$ and will denote by $m(t,x)\in\R^3$ the magnetisation vector at a time $t\ge 0$ and at a point $x\in\R$. For temperatures below the Curie point the length $|m(t,x)|$ of this vector is constant in $(t,x)$ \cite{brown}, hence can be assumed equal to 1: 
	\begin{align*}
		m: [0,\infty) \times \mathbb{R} \to \mathbb{S}^2.
	\end{align*}
	The LLS equation proposed in \cite{slon} describes the dynamics of the magnetisation vector subject to the spin-velocity field (electric current):
	\begin{align*}
		v: (0,\infty) \times \mathbb{R} \to \mathbb{R}\,.
	\end{align*}
	It takes the form 
	\begin{equation}\label{LLS_Gilbert_full}
		\partial_t m = -m\times \partial_{xx}m -\alpha m\times\la m\times \partial_{xx}m \ra -v \partial_x m + \gamma m\times\la v \partial_x m \ra \,,
	\end{equation}
	with $\alpha>0$ and $\gamma\in\R$.  
	The term $ v \partial_x m $ is known as the adiabatic term and the non-adiabatic term is given by $ \gamma m \times (v \partial_x m) $. 
	For more details on the form of this equation see \cite{mp}. 
	\vskip0.05cm\noindent
	We will consider a version of equation \eqref{LLS_Gilbert_full} with the spin-velocity field perturbed by noise: 
	\begin{equation}\label{eq_spde}
		\partial_t m= -m\times\partial_{xx}m-\alpha m\times\la m\times\partial_{xx}m\ra-\partial_x m \circ(v\,dt+dW)+ \la m\times \gamma \partial_x m\ra\circ(v\,dt+dW)\,,
	\end{equation}
	where $W$ is an infinite-dimensional Wiener process taking values in an appropriate function space. We emphasise, that noise arises in equation \eqref{eq_spde} in a way very different way from the way it appears in stochastic Landau-Lifshitz equations studied in  \cite{Brzezniak2013_sLLG,bgj,ngan1,ngan2}. While in the aforementioned papers it is a thermal noise arising inside the magnetic domain and has bounded diffusion coefficient, in \eqref{eq_spde} it is a transport noise brought into the system by the electric current, and has the gradient of the solution as a diffusion coefficient. Therefore, analysis of this equation requires  more delicate arguments. We will show that for every initial condition $m_0$ with 
	\begin{equation}\label{eq: m0 condition}
		\left|m_0(x)\right|=1,\quad \int_\R\left|\partial_x m_0 \right|^2\,dx<\infty\,,
	\end{equation}
	there exists a unique pathwise and strong in PDEs sense solution to \eqref{eq_spde}. We will use the observation made in \cite{mp} that 
	under the constraint $ |m(t,x)|=1 $ for all $ (t,x) \in (0,\infty) \times \mathbb{R} $, we have 
	\[\partial_x m =-m\times\la m\times\partial_x m\ra\,,\]
	hence 
	equation \eqref{eq_spde} can be written in the form
	\begin{equation}\label{sLLG}
		\begin{aligned}
			dm
			&= -m \times \la \partial_{xx}m + \alpha m \times \partial_{xx}m \ra dt 
			+ m \times \la m \times \partial_x m + \gamma \partial_x m \ra \circ\la v \,dt+\,dW\ra,
		\end{aligned}
	\end{equation} 
	with $ m(0) = m_0 $ satisfying \eqref{eq: m0 condition}. 	
	We will assume that $W$ is a Wiener process taking values in $H^2(\R)$ and will prove the existence and uniqueness of pathwise solutions to this equation, see Theorem \ref{Theorem: existence} for details. Due to the presence of gradient noise of multiplicative type, we can prove this theorem only for the Wiener process with the covariance small enough, see Theorem \ref{Theorem: existence} for precise formulation. 
	Let us comment on the proof of this theorem. We start with the formulation of a semidiscrete approximation scheme that allows to construction approximating solutions that satisfy the sphere constraint. The same approach was used in \cite{mp} to study the deterministic equation \eqref{LLS_Gilbert_full}. Then we obtain a set of uniform estimates for the approximate solutions. This step requires using quadratic interpolations and is technically much more complicated than in the case of the stochastic Landau-Lifshitz equation without transport. Next, we follow compactness type argument to prove the existence of a limiting point that is a strong in PDEs sense solution to stochastic LLS equation. Then we show uniqueness of pathwise solutions and use the Yamada-Watanabe theorem in the same way as in \cite{Brzezniak2013_sLLG}.

\section{Semi-discrete scheme and the main result}\label{Section: stochastic formulation}
\subsection{Notation}
\subsubsection{Function spaces}
	Let $ \rho_w(x) = (1+x^2)^{-w} $ for $ w \geq 0 $. Clearly, 
	\begin{equation}\label{rhow properties}
		\rho_w(x) \in (0,1], \quad 
		|\rho_w'(x)| \leq w \rho_w(x),
	\end{equation}
	and for $ w > \frac{1}{2} $, $ \int_\R \rho_w(x) \ dx<\infty $. 
	For $p\in[1,\infty)$, define the weighted Lebesgue space $ \L^p_w $ by 
	\begin{align*}
		\L^p_w &= \left\{ f: \R \to \R^3;\, \int_{\R} |f(x)|^p \rho_w(x)\, dx < \infty \right\}.
	\end{align*}
	If $w=0$, then we will write $ \L^p$ instead of $ \L^p_0$. 
	We will denote by $ \H_w^1$ the Hilbert space
	\[\b H_w^1=\left\{f\in\b L_w^2;\, Df\in\b L_w^2\right\}\,.\]
	Let $0<w_1 < w_2 $. 
	Then $ \rho_{w_1} > \rho_{w_2} $ and 
	the embeddings $ \L^2 \hookrightarrow \L^2_{w_1} \hookrightarrow \L^2_{w_2} $ are continuous with
	\begin{align*}
		|f|_{\L^2_{w_2}} \leq |f|_{\L^2_{w_1}} \leq |f|_{\L^2}, \quad \forall f \in \L^2.
	\end{align*}

	Moreover, the embeddings 
	\[ \H_{w_1}^1 \hookrightarrow \L^2_{w_2}\quad\text{and}\quad \L^2_{w_1} \cap \mathring{\H}^2 \hookrightarrow \H_{w_2}^1\]
	are compact, where $\mathring{\b H}^2$ stands for a standard homogeneous Sobolev space of functions $f:\R\to\R^3$ with weak derivatives $ Df, D^2 f \in \L^2 $.
	The Laplace operator $\Delta$ considered in $\b L^2_w$ with the domain $\b H^2_w$ is variational and the operator 
	$ A_1 = I-\Delta $ is invertible. 
	For $ \beta> 0 $, let $ \b H^\beta_w $ denote the domain of $ A_1^{\beta/2} $ endowed with the  norm $ |\cdot|_{\b H^{\beta}_w} := |A_1^{\beta/2} \cdot|_{\L^2_w} $. 
	and with dual space $ \b H^{-\beta}_w $. 

\subsubsection{Assumption and Notation}
	Let $W$ be an $H^2(\R)$-valued Wiener process with the covariance operator $Q$. Then there exists a complete orthonormal sequence $\left\{f_j;\,j\ge 1\right\}$ of $H^2(\R)$ made of eigenvectors of $Q$, that is 
	\[Qf_j=q^2_jf_j,\quad q^2:=\sum_jq^2_j<\infty\,,\]
	and then we have 
	\[W(t)=\sum_{j=1}^\infty q_jW_j(t)f_j\,.\]
	The following is a standing assumption for the rest of the paper and it will not be enunciated again. 
	\begin{assumption}\label{hyp1}
		\begin{equation}\label{Ckappa bound}
			C_\kappa^2:=\left|\sum_{j=1}^\infty q_j^2 \la f_j^2 + (f_j^\prime)^2 + (f_j^{\prime \prime})^2 \ra \right|_{\L^\infty}<\infty,
		\end{equation}
	and  
		$ v $ is in $ \mathcal{C}([0,T]; H^1(\R)) $ with
		\begin{equation}\label{Cv bound}
			C_v:=\essup_{t \in \R_+} |v(t)|_{\L^\infty}<\infty\,.
		\end{equation}
	\end{assumption}
	Define a function 
	\[\kappa^2(x)=\sum_{j=1}^\infty q_j^2f_j^2(x),\quad x\in\R\,.\]
	\begin{remark}\label{rem1}
	(a) Assumption \ref{hyp1} yields
		\[|\kappa|_{\L^\infty} \leq C_\kappa,\quad\mathrm{and}\quad |\kappa \kappa^\prime|_{\L^\infty} \leq C_\kappa^2\,.\] 
		
	(b) 
		Every $\b H^2$-valued finite-dimensional Wiener process satisfies \eqref{Ckappa bound} provided $f_j^{\prime \prime} \in\b L^\infty$ for $j\ge 1$. 
	\end{remark}
	The following notations will be used throughout the paper. Let $ G: \L^\infty\cap\mathring{\H}^1\to \L^2$ be defined as 
	\[G(m)=  m \times \la m \times \partial_x m \ra+ m\times \gamma \partial_x m\,.\]
	Let $ \mathcal{G}(m) := G^\prime(m) (G(m)) $,
	which can be expressed as
	\begin{align*}
		\mathcal{G}(m) 
		&= (\gamma^2 - |m|^2) m \times \la m \times \partial_{xx}m \ra -2 \gamma |m|^2 m \times \partial_{xx}m - \gamma^2 \partial_x m \times \la m \times \partial_x m \ra \\
		&\quad - \left|m \times \partial_x m\right|^2 m + \gamma \langle m, \partial_x m \rangle m \times \partial_x m. 
	\end{align*}
	
	In the rest of the paper we consider equation \eqref{sLLG} in its It{\^o} form: 
	\begin{equation}\label{eq_ito}
			dm=\la F(m)+ \frac{1}{2}S(m)\ra\, dt + G(m)\, dW, \quad m(0) = m_0. 
	\end{equation}	
	Here, 
	\begin{align*}
		F(m)&:= m \times \la m \times v \partial_x m \ra + \gamma m\times v \partial_x m - m \times \la \partial_{xx} m + \alpha m \times \partial_{xx} m \ra  \,
	\end{align*}
	and the Stratonovich correction term $S(m)$ takes the form
	\begin{align*}
		S(m) 
		&:=\kappa^2\mathcal{G}(m)+ \kappa \kappa^\prime \left[ m\times(m\times G(m))+ \gamma m\times G(m)\right] \\
		&= \kappa^2 \mathcal{G}(m) + \kappa \kappa^\prime \left[ (\gamma^2-|m|^2) m \times \la m \times \partial_x m \ra -2 \gamma |m|^2 m \times \partial_x m \right].
	\end{align*}

\subsection{Semi-discrete scheme}
\subsubsection{Discrete operators and discrete spaces}	
	Let $ \Z_h = \{ x = kh : k \in \Z \} $ denote a discretization of the real line of mesh size $ h > 0 $. 	
	For $ u : \Z_h \to \mathbb{R}^3 $, we write $ u^{\pm}(x) $ for $ u(x \pm h) $, and we introduce discrete gradient and discrete Laplace operators:
	\begin{equation}\label{discrete grad Laplacian}
		\begin{aligned}
			\partial^{h} u = \frac{1}{h} \left( u^+ - u \right) 
			\quad\text{and}\quad
			\Delta^h u = \frac{1}{h} (\partial^{h}u- \partial^{h} u^-).
		\end{aligned}
	\end{equation}	
	Let $ \L_h^\infty $, $ \L_h^p $, $ \H_h^1$ and $ \mathbb{E}_h := \L_h^\infty \cap \mathring\H_h^1 $ be discrete spaces equipped with respective norms:
	\begin{alignat*}{3}
		|u|_{\L_h^\infty} &= \sup_{x \in \Z_h} |u(x)|\,, &&\quad
		|u|_{\L_h^p}^p = h \sum_{x \in \Z_h} |u(x)|^p\,, \\
		|u |_{\E_h}^2 &= | u |_{\L_h^\infty}^2 + |\partial^{h} u|_{\L_h^2}^2\,, &&\quad
		|u|_{\b H_h^1}^2 = |u|_{\L_h^2}^2+ |\partial^{h} u|^2_{\L_h^2}\,, 
	\end{alignat*}
	where $p\in[1,\infty)$. 
	
	We will say that $u:[0,T]\times\Omega\times\b Z_h\to\R^3$ is an $\b E_h$-valued progressively measurable process if for every $x\in\b Z_h$ the process $u(\cdot,x)$ is progressively measurable and for every $t \in [0,T]$, 
	\[|u(t)|_{\b E_h}<\infty,\quad\P\text{-a.s.}\]
	In particular, the process $\left\{|u(t)|_{\b E_h};\,t\ge 0\right\}$ is progressively measurable. 
	
	Let $ \mathcal{E}_h $ denote the space of $ \E_h $-valued progressively measurable processes, with norm
	\begin{align*}
		| u |_{\mathcal{E}_h} = \sup_{t \in [0,T]} \la \E \left[ | u(t) |_{\E_h}^2 \right] \ra^{\frac{1}{2}}.
	\end{align*}
	
\subsubsection{Discrete equation} 
	For $ u \in \mathbb{E}_h $, we define
	\begin{equation}\label{discrete coef}
		\begin{aligned}
			F^h(u) &= u \times \left( u \times v \partial^h u \right) + \gamma u \times v \partial^h u  - u \times \left( \Delta^h u + \alpha u \times \Delta^h u \right) \\	
			G^h(u) &= u \times \la u \times \partial^h u \ra + \gamma u \times \partial^h u\\
			S^h(u) 
			&= \mathcal{G}^h_\kappa(u) + \kappa \kappa' \left[ u\times(u\times G^h(u))+ \gamma u\times G^h(u)\right],
		\end{aligned}
	\end{equation}
	where
	\begin{align*}
		\mathcal{G}^h_\kappa(u) &= \frac{1}{2}\la (\kappa^2)^- + \kappa^2 \ra G^h_1(u) + \kappa^2 G^h_2(u) + (\kappa^2)^- G^h_3(u)\\
		G^h_1(u) &= (\gamma^2 - |u|^2) u \times \la u \times \Delta^h u \ra -2 \gamma |u|^2 u \times \Delta^h u \\
		G^h_2(u) &= - \gamma^2 \partial^h u \times \la u \times \partial^h u \ra - |u \times \partial^h u|^2 u \\
		G^h_3(u) &= 2\gamma \langle u, \partial^h u^- \rangle u \times \partial^h u\,.
	\end{align*}
	Fix a terminal time $ T \in (0,\infty) $, 
	we describe the semi-discrete scheme for~\eqref{eq_ito} as a stochastic differential equation in the space $\E_h$:
	\begin{equation}\label{SDE: discrete}
		\begin{aligned}
			dm^h &= \la F^h(m^h) + \frac{1}{2} S^h(m^h) \ra dt + G^h(m^h) \,dW, \quad 
			m^h(0) = m_0\in\E_h,
		\end{aligned}
	\end{equation}
	In \eqref{discrete coef} and \eqref{SDE: discrete}, $ \kappa^2, v(t) $ and $ W(t) $ are the restrictions of the corresponding functions to $ \Z_h $ for every $ t \in [0,T] $.  
	The term $ S^h(m^h) $ is a discretised and a modified version of the Stratonovich correction $ S(m) $. It is chosen to simplify the proof in Section \ref{Section: uniform estimates} without affecting the limit. For example, $ G_3^h(m^h) $ with the constant $ 2 $ does not match with the term $ \gamma \langle m, \partial_x m \rangle m \times \partial_x m $ in $ \mathcal{G}(u) $, but if $ \langle m^h, \partial^h m^{h-} \rangle $ converges to $ 0 $ in a suitable space, then the constant will be irrelevant to the final equation.

\subsection{Main result}
	\bd{def1}\label{def: sol}
	We say that a progressively measurable process $m$ defined on $\la\Omega,\mathcal F,\la\mathcal F_t\ra,\P, W\ra$ where $ W $ is a Wiener process, 
	is a solution to equation \eqref{eq_ito} if
	\begin{enumerate}[(a)]
		\ie
		\item 
		$|m(t,x)|=1$ $(t,x)$-a.e.
		
		
		\item 
		for every $T \in (0,\infty)$, 
		\[\E\left[\sup_{t\in [0,T]} \left|\partial_x m\right|^2_{\b L^2}+\int_0^T\left|\partial_{xx}m\right|^2_{\b L^2}\,dt \right]<\infty\,,\]
		
		\item 
		for every $t \in [0,T]$ the following equality holds in $\b L^2$\,:
		\be{sLLS_mainthm}
		m(t)-m_0 = \int_0^t\la F(m(s))+\frac{1}{2}S(m(s))\ra\,ds+\int_0^tG(m(s))\,dW(s),\quad \P\text{-a.s.}
		\ee
	\end{enumerate}	
	\ed
	Note that (a)--(d) above and Assumption \ref{hyp1} yield
	\[\int_0^t\la |F(m(s))|_{\b L_2}^2+S(m(s))|_{\b L^2}^2+|\kappa G(m(s))|^2_{\b L^2}\ra\,ds<\infty\,,\]
	hence the Bochner integral and the It\^o integral in \eqref{sLLS_mainthm} are well defined in $\b L^2$. 
	
	\begin{lemma}\label{Lemma: solution of disSDE}
		For every $ h > 0 $, let $ |m_0|_{\E_h} \leq K_0 $. Then there exists a unique solution $ m^h $ of the semi-discrete scheme \eqref{SDE: discrete} in $ \mathcal{E}_h $ satisfying $ |m^h(t,x)|=1 $, $ \P $-a.s. for all $ t \in [0,T] $ and $ x \in \Z_h $. 
	\end{lemma}

	\bt{th1}\label{Theorem: existence}
	There exists a solution $ (\Omega, \mathcal{F}, \la\mathcal F_t\ra, \P, W, m) $ of \eqref{eq_ito} in the sense of Definition \ref{def: sol}, such that for $ p \in [1,\infty) $,
	\begin{align*}
		\E\left[\sup_{t\in [0,T]} \left|\partial_x m(t)\right|^p_{\b L^2}+\la\int_0^T\left|\partial_{xx}m(t)\right|^2_{\b L^2}\,dt\ra^p \right]<\infty\,,
	\end{align*}
	and for every $T> 0$ and $\alpha\in\la 0,\frac{1}{2}\ra$, 
	\begin{align*}
		m-m_0\in C^\alpha\la [0,T];\,\b L^2\ra, \quad \P\text{-a.s.}
	\end{align*}
	Moreover, there exists a convergent subsequence $ \{m_h\} $ defined on $ (\Omega, \mathcal{F}, \P) $ such that $ m_h $ has the same law as a quadratic interpolation of $ m^h $ for every $ h > 0 $, and $ m $ is the $ \P $-a.s. limit of $ \{m_h\} $ in $ \mathcal{C}([0,T]; \H^{-1}_{w}) $ for some $ w \geq 1 $. 
	\et
	
	\begin{theorem}\label{Thm: pathwise uniqueness}
		The solution $ m $ of \eqref{eq_ito} is pathwise unique and therefore unique in law. 
	\end{theorem}

\section{Discretization}
	From the definition of discrete operators and the discrete $\L_h^p$ norm, we deduce following results.
	\begin{remark}\label{rem2}
	For $ u : \Z_h \to \mathbb{R}^3 $, 
	\begin{enumerate}[(a)]
	\ie
		\item  
		$\partial^{h} (\partial^{h} u) = \frac{1}{h} (\partial^{h}u^+- \partial^{h} u) = \Delta^h u^+$,
		
		\item 
		for any $p\in[1,\infty]$,
		$|u|_{\L_h^p} = |u^+|_{\L_h^p} = |u^-|_{\L_h^p}$,
		which implies 
		$ |\partial^{h}u|_{\L_h^p} = |\partial^{h} u^+|_{\L_h^p} = |\partial^{h} u^-|_{\L_h^p}$,
		and hence, 
		\begin{equation*}
			|\partial^h u|_{\b L_h^p}^2 \leq \frac{4}{h^2} | u|_{\b L_h^p}^2, \quad |\Delta^h u|_{\b L_h^p}^2 \leq \frac{4}{h^2} |\partial^hu|_{\b L_h^p}^2,
		\end{equation*}
	
		\item 
		Lemma \ref{Lemma: discrete interpolation ineq} indicates $|u|_{\b L_h^\infty}\le C |u|_{\b H^1_h}$ for any $u\in\b H_h^1 \,$. 
	\end{enumerate}
	\end{remark}

\subsection{Existence of a unique solution of the semi-discrete scheme}
	
	\begin{lemma}\label{Lemma: product of local Lip}
		For every $ h >0 $, if $ f,g: \E_h \to \E_h $ are locally Lipschitz and satisfy $ f(0) = g(0) = 0 $, then $ f \times g $, $ \langle f, g \rangle $ and $ \partial^h f $ are also locally Lipschitz on $ \E_h $. 
	\end{lemma}
	The result in Lemma \ref{Lemma: product of local Lip} is clear and we omit the proof here. 
	Then we check that the coefficients in \eqref{SDE: discrete} are locally Lipschitz on $ \E_h $. 
	
	\begin{lemma}\label{Lemma: F,G,Gtilde local Lip}
		For every $ h > 0 $, $ F^h, G^h, S^h: \E_h \to \E_h $ are locally Lipschitz on $ \E_h $. 
	\end{lemma}
	\begin{proof} 
		Let $ u,w \in \mathbb{E}_h $. It follows from Remark \ref{rem2}(b) that 
		\begin{align*}
			|\partial^h u - \partial^h w|_{\E_h}^2 
			&=|\partial^h (u-w) |_{\L_h^\infty}^2 + |\partial^{h} \partial^h (u-w)|_{\L_h^2}^2\\
			&\leq \frac{4}{h^2} \left( | u-w|_{\L_h^\infty}^2 + |\partial^h (u-w)|_{\L_h^2}^2 \right)
			=\frac{4}{h^2} |u-w|_{\E_h}^2,
		\end{align*}
		and
		\begin{align*}
			|\Delta^h u - \Delta^h w|_{\E_h}^2 
			&=|\Delta^h (u-w) |_{\L_h^\infty}^2 + |\partial^{h} \Delta^h (u-w)|_{\L_h^2}^2\\
			&\leq \frac{4}{h^2} \left( |\partial^h (u-w)|_{\L_h^\infty}^2 + |\Delta^h (u-w)|_{\L_h^2}^2 \right)\\
			&\leq \frac{16}{h^4} \left( | u-w|_{\L_h^\infty}^2 + |\partial^h (u-w)|_{\L_h^2}^2 \right) 
			=  \frac{16}{h^4} |u - w|_{\E_h}^2.
		\end{align*}
		By Lemma \ref{Lemma: product of local Lip} and \eqref{Cv bound}, $ F^h, G^h$ and $S^h $ are locally Lipschitz.
	\end{proof}
	
	\begin{proof}[\textbf{Proof of Lemma \ref{Lemma: solution of disSDE}}]
		For each $ n \in \mathbb{N} $ and $ r^h =F^h, S^h $ and $ G^h $, define 
		\begin{align*}
			r_n^h(u) 
			=  \begin{cases}
				r^h(u) &\text{if } |u|_{\mathbb{E}_h} \leq n \\
				r^h\la \frac{n u}{|u|_{\mathbb{E}_h}} \ra &\text{if } |u|_{\mathbb{E}_h} > n. 
			\end{cases}
		\end{align*}
		Then $ F_{n}^h $, $ S_n^h $ and $ G_n^h $ are Lipschitz on $ \mathbb{E}_h $. 
		
		Fix $ n \in \mathbb{N} $. 
		Let $ A_n: \mathcal{E}_h \to \mathcal{E}_h $ be given by
		\begin{equation}\label{pf: An defn}
		\begin{aligned}
			A_n(u)(t) 
			&= m_0 + \int_0^t \la F_{n}^h(u(s)) + \frac{1}{2} S_n^h(u(s)) \ra ds + \int_0^t G_n^h (u(s)) \ dW(s) \\
			&= m_0 + I_n(t) + J_n(t). 
		\end{aligned}
		\end{equation}
		We first verify that $ A_n(u) \in \mathcal{E}_h $ for $ u \in \mathcal{E}_h $. 
		Note that $ F_n^h $ and $ S_n^h $ are bounded on $ \mathbb{E}_h $, with 
		\begin{align*}
			\E \left[ |I_n(t)|_{\E_h}^2 \right]
			&\leq T \E \left[ \int_0^t \left| F_{n}^h(u(s)) + \frac{1}{2} S_n^h(u(s)) \right|_{\E_h}^2 ds \right] \\
			&\leq C_1(h,n) T \E \left[ \int_0^t |u(s)|_{\E_h}^2 ds \right] \\
			&\leq C_1(h,n) T^2 |u|_{\mathcal{E}_h}^2,
		\end{align*}
		for some constant $ C_1 $ that depends on $ h $ and $ n $. 
		For $ J_n(t) $, we have
		\begin{align*}
			\sum_j q_j^2 |f_j G^h(u(s))|_{\L_h^2}^2 
			&= \sum_j q_j^2 \left| f_j u(s) \times \la u(s) \times \partial^h u(s) \ra + \gamma f_j u(s) \times \partial^h u(s) \right|_{\L_h^2}^2 \\
			&\leq 2 |\kappa^2|_{\L_h^\infty} \la |u(s)|_{\L_h^\infty}^4 + \gamma^2 |u(s)|_{\L_h^\infty}^2 \ra |\partial^h u(s)|^2_{\L_h^2}.
		\end{align*}
		where the last inequality holds by Tonelli's theorem. 
		This together with Remark \ref{rem2}(b) implies
		\begin{align*}
			\sum_j q_j^2 |\partial^{h} (f_j G^h(u(s))) |_{\L_h^2}^2 
			&\leq \frac{4}{h^2} \sum_j q_j^2 |f_j G^h(u(s))|_{\L_h^2}^2 \\
			&\leq \frac{8}{h^2} |\kappa^2|_{\L_h^\infty} \la |u(s)|_{\L_h^\infty}^4 + \gamma^2 |u(s)|_{\L_h^\infty}^2 \ra \ |\partial^{h} u(s)|^2_{\L_h^2}.
		\end{align*}
		Then by the definition of $ G_n^h $, the assumption \eqref{Ckappa bound} and Fubini's theorem, 
		\begin{align*}
			\E \left[ \int_0^t \sum_j q_j^2 \left| f_j G_n^h (u(s)) \right|^2_{\H_h^1}  \ ds \right] 
			&\leq 2 C_\kappa^2 (n^4 + \gamma^2 n^2) \la 1+ \frac{4}{h^2} \ra T \sup_{s \in [0,t]} \mathbb{E}\left[ |u(s)|_{\E_h}^2 \right] \\
			&= C_2(h,n,\kappa) T |u|_{\mathcal{E}_h}^2,
		\end{align*}
		for $ C_2(h,n,\kappa) = 2C_\kappa^2 (n^4 + \gamma^2 n^2) (1+\frac{4}{h^2}) T$.
		Thus, $ J_n $ is a $ \H_h^1 $-valued continuous martingale. 
		By Lemma \ref{Lemma: discrete interpolation ineq} (or Remark \ref{rem2}(c)), there exists a constant $ C>0 $ such that 
		\begin{equation}\label{pf: J1}
		\begin{aligned}
			| J_n(t) |_{\E_h}^2 
			&= | J_n(t) |_{\L_h^\infty}^2 + | \partial^{h} J_n(t) |_{\L_h^2}^2 
			\leq (C^2+1) |J_n(t) |_{\H_h^1}^2.
		\end{aligned}
		\end{equation}
		From \cite[Corollary 4.29]{Red_book}, 
		\begin{equation}\label{pf: J2}
		\begin{aligned}
			\E \left[ \sup_{t \in [0,T]} \left| J_n(t) \right|_{\H_h^1}^2 \right]
			&= \E \left[ \sup_{t \in [0,T]} \left| \int_0^t G_n^h(u(s)) \ dW(s) \right|_{\H_h^1}^2 \right] \\
			&\leq \E \left[ \int_0^T \sum_j q_j^2 \left| f_j G_n^h(u(s)) \right|_{\H_h^1}^2 \ ds  \right] \\
			&\leq C_2(h,n,\kappa) T |u|_{\mathcal{E}_h}^2.
		\end{aligned}
		\end{equation}
		It follows from~\eqref{pf: J1} and~\eqref{pf: J2} that 
		\begin{equation}\label{pf: Gnh isometry Eh vs Hh1}
			\mathbb{E} \left[ \sup_{t \in [0,T]} \left| J_n(t) \right|_{\mathbb{E}_h}^2 \right]
			\leq \la C^2+ 1\ra C_2(h,n,\kappa) T |u|_{\mathcal{E}_h}^2.
		\end{equation}
		Thus, $ A_n(u) \in \mathcal{E}_h $ for $ u \in \mathcal{E}_h $. 	
	
		For $ u, \nu \in \mathcal{E}_h $, there exists a constant $ C >0 $ such that
		\begin{align*}
			|A_n(\nu) - A_n(u)|_{\mathcal{E}_h}^2
			&\leq C \sup_{t \in [0,T]} \mathbb{E} \left[ \left| \int_0^t F_{n}^h(\nu(s)) - F_{n}^h(u(s)) \ ds \right|_{\mathbb{E}_h}^2 \right] \\
			&\quad + C \sup_{t \in [0,T]} \mathbb{E} \left[ \left| \int_0^t \frac{1}{2} \la S_n^h(\nu(s)) - S_n^h(u(s)) \ra \ ds \right|_{\mathbb{E}_h}^2 \right] \\
			&\quad + C \sup_{t \in [0,T]} \mathbb{E} \left[ \left| \int_0^t \la G_n^h(\nu(s)) - G_n^h(u(s)) \ra dW(s) \right|_{\mathbb{E}_h}^2 \right].
		\end{align*}
		Similarly, $ M_n(t) := \int_0^t  \la G_n^h(\nu(s)) - G_n^h(u(s)) \ra dW(s) $ is a $ \H_h^1 $-valued continuous martingale, and replacing $ J_n $ by $ M_n $ in \eqref{pf: J1} and \eqref{pf: J2}, we have
		\begin{equation}\label{pf: Gnh v-u isometry Eh vs Hh1}
			\mathbb{E} \left[ \sup_{t \in [0,T]} \left| M_n(t) \right|_{\mathbb{E}_h}^2 \right]
			\leq \mathbb{E} \left[ \int_0^T \sum_j q_j^2 \left| f_j \la G_n^h(\nu(s)) - G_n^h(u(s))\ra\right|_{\H_h^1}^2 \ ds  \right].
		\end{equation}
		By construction, if $ |\nu(s)|_{\E_h}, |u(s)|_{\E_h} \leq n $, then
		\begin{align*}
			|G_n^h(\nu(s)) - G_n^h(u(s)) |_{\L_h^2} 
			& \leq |\nu(s) - u(s)|_{\L_h^\infty}  \la |\nu(s) |_{\L_h^\infty} + |u(s)|_{\L_h^\infty} +|\gamma| \ra |\partial^h \nu(s)|_{\L_h^2}  \\
			&\quad + \la |u(s)|_{\L_h^\infty}^2 + |\gamma| |u(s)|_{\L_h^\infty} \ra  | \partial^h \nu(s) - \partial^h u(s) |_{\L_h^2} \\
			&\leq (3 n^2 + 2|\gamma|n) \ |\nu(s) - u(s) |_{\E_h}.
		\end{align*}
		If $ |\nu(s)|_{\E_h}, |u(s)|_{\E_h} > n $, then let $ n_s^\nu = n |\nu(s)|_{\E_h}^{-1} $ and $ n_s^u = n |u(s)|_{\E_h}^{-1} $, we have 
		\begin{align*}
			|n_s^\nu - n_s^u| \ |u|_{\E_h} \leq |\nu(s)- u(s)|_{\E_h},
		\end{align*}
		and
		\begin{align*}
			&|G_n^h(\nu(s)) - G_n^h(u(s)) |_{\L_h^2} \\
			&= |G^h(n_s^\nu \nu(s)) - G^h(n_s^u u(s)) |_{\L_h^2} \\
			&\leq \la n_s^\nu  |\nu(s) - u(s)|_{\L_h^\infty} + | n_s^\nu - n_s^u| |u(s)|_{\L_h^\infty} \ra  \la n_s^\nu |\nu(s) |_{\L_h^\infty} + n_s^u |u(s)|_{\L_h^\infty} + |\gamma| \ra n_s^\nu |\partial^h \nu(s)|_{\L_h^2}  \\
			&\quad + \la n_s^\nu | \partial^h \nu(s) - \partial^h u(s) |_{\L_h^2} + | n_s^\nu - n_s^u| |\partial^h u(s)|_{\L_h^2} \ra \la (n_s^u)^2 |u(s)|_{\L_h^\infty}^2 + |\gamma| n_s^u |u(s)|_{\L_h^\infty} \ra \\
			&\leq 2 (3n^2 + 2|\gamma|n) \ |\nu(s) - u(s) |_{\E_h}.
		\end{align*}
		If $ |\nu(s)|_{\E_h} > n $ and  $ |u(s)|_{\E_h} \leq n $, then 
		\begin{align*}
			|n_s^\nu -1 | 
			\leq n^{-1} |n-|\nu(s)|_{\E_h}| 
			\leq n^{-1} | |u(s)|_{\E_h} - |\nu(s)|_{\E_h} | 
			\leq n^{-1} |u(s)-\nu(s)|_{\E_h},
		\end{align*}
		implying that $ |n_s^\nu - 1| \ |u|_{\E_h} \leq |u(s)- \nu(s)|_{\E_h} $ and
		\begin{align*}
			|G_n^h(\nu(s)) - G_n^h(u(s)) |_{\L_h^2} 
			&\leq 2 (3n^2 + 2|\gamma|n) \ |\nu(s) - u(s) |_{\E_h}.
		\end{align*}
		Similar result follow for $ \partial^h (G_n^h(\nu)-G_n^h(u)) $ using Remark \ref{rem2}(b). 
		Then, by \eqref{pf: Gnh v-u isometry Eh vs Hh1}, Lemma \ref{Lemma: F,G,Gtilde local Lip} and H{\"o}lder's inequality, there exist constants $ L_1(h,n,T) $ and $ L_2(h,n,T) $ such that
		\begin{align*}
			|A_n(\nu) - A_n(u)|_{\mathcal{E}_h}^2
			&\leq L_1(h,n,T) \sup_{t \in [0,T]} \mathbb{E} \left[ \int_0^t  |\nu(s) - u(s)|_{\mathbb{E}_h}^2 \ ds\right] \\
			&\leq L_2(h,n,T) |\nu-u|_{\mathcal{E}_h}^2. 
		\end{align*}
		Consider the discrete equation 
		\begin{equation}\label{SDE: discrete Fn Gn}
			dm^h_n(t) = \la F_{n}^h(m^h_n(t)) + \frac{1}{2} S_n^h(m^h_n(t)) \ra dt + G_n^h(m^h_n(t)) \ dW(t), 
		\end{equation}
		with $ m^h_n(0) = m_0 \in \mathbb{E}_h $ on intervals $ [(k-1) \tilde{T},k\tilde{T}] $ for $ k \geq 1 $, where $ \tilde{T} $ satisfies $ L_2(h,n,\tilde{T}) < 1 $. 
		By the Banach fixed point theorem, there exists a unique solution $ m_n^h \in \mathcal{E}_h $ of \eqref{SDE: discrete Fn Gn} on $ [0,T] $. 
		
		Define the stopping times
		\begin{align*}
			\tau_n := \inf \{ t \geq 0 : |m_n^h(t)|_{\E_h} > n \}, \quad 
			\tau_n' := \inf \{ t\geq 0 : |m_{n+1}^h(t)|_{\E_h} > n \}.
		\end{align*}
		Let $ \tau = \tau_n \wedge \tau_n' $.
		Then $ A_n(m_{n+1}^h) = A_{n+1}(m_{n+1}^h) $ on $ [0,\tau) $, and by \eqref{pf: Gnh v-u isometry Eh vs Hh1}, 
		\begin{align*}
			\mathbb{E} \left[ \sup_{t \in [0,T]} \left| m_{n+1}^h(t \wedge \tau) - m_n^h(t \wedge \tau) \right|_{\E_h}^2 \right] 
			&= \mathbb{E} \left[ \sup_{t \in [0,T]} \left| A_n(m_{n+1}^h)(t \wedge \tau) - A_n(m_n^h)(t \wedge \tau) \right|_{\E_h}^2  \right] \\
			&\leq L_1(h,n,T) \mathbb{E} \left[ \int_0^T \left| m_{n+1}^h(s) - m_n^h(s) \right|_{\E_h}^2 \ ds \right],
		\end{align*}
		which implies
		\begin{align*}
			\mathbb{E} \left[ \sup_{t \in [0,T]} \left| m_{n+1}^h(t \wedge \tau) - m_n^h(t \wedge \tau) \right|_{\E_h}^2 \right]  = 0,
		\end{align*}
		by Gr{\"o}nwall's lemma. 
		Thus, $ m_{n+1}^h(\cdot \wedge \tau) = m_n^h(\cdot \wedge \tau) $ and $ \tau = \tau_n $, $ \mathbb{P} $-a.s. and the discrete equation \eqref{SDE: discrete} admits a local solution $ m^h(t) = m_n^h(t) $ for $ t \in [0,\tau_n] $. 
		
		Recall that $ m_0 \in \mathbb{E}_h $ and $ |m_0(x)| =1 $ for all $ x \in \Z_h $. 
		Applying It{\^o}'s lemma to $ \frac{1}{2}|m^h(t,x)|^2 $, 
		\begin{align*}
			\frac{1}{2} d|m^h(t,x)|^2
			&= \lb F^h(m^h(t))(x) + \frac{1}{2} S^h(m^h(t))(x), m^h(t,x) \rb dt + \frac{1}{2} \kappa^2(x) |G^h(m^h(t))(x) |^2 dt \\
			&\quad + \lb G^h(m^h(t))(x), m^h(t,x) \rb dW(t).
		\end{align*}
		By $ \langle a, a \times b \rangle =0 $, for any $ t \in [0,\tau_n] $ and $ x \in \Z_h $, 
		\begin{align*}
			&\lb F^h(m^h), m^h \rb(t,x) = 0, \\
			&\lb S^h(m^h), m^h \rb(t,x) 
			= \lb \kappa^2 G^h_2(m^h), m^h \rb (t,x)
			= - \kappa^2(x) |G^h(m^h)|^2(t,x), \\
			&\lb G^h(m^h), m^h \rb(t,x) = 0.
		\end{align*}
		Therefore, $ |m^h(t,x)| = |m_n^h(t,x)| =  |m_0(x)| = 1 $ for any $ t \in [0,\tau_n] $ and $ x \in \Z_h $. 
		
		For any fixed $ h $ and $ n $, the unique solution $ m_n^h $ of \eqref{SDE: discrete Fn Gn} satisfies
		\begin{align*}
			\mathbb{E} \left[ |m^h_n(t \wedge \tau_n)|^2_{\mathbb{E}_h} \right]
			&= \mathbb{E} \left[ |m^h_n(t \wedge \tau_n)|_{\L_h^\infty}^2 + |\partial^{h} m^h_n(t \wedge \tau_n) |_{\L_h^2}^2 \right] \\
			&= 1 + \mathbb{E} \left[ |\partial^{h} m^h_n(t \wedge \tau_n) |_{\L_h^2}^2 \right].
		\end{align*}
		We apply It{\^o}'s lemma to $ \frac{1}{2}|\partial^{h} m_n^h(t \wedge \tau_n)|^2_{\L_h^2} $,
		\begin{align*}
			&\frac{1}{2}|\partial^{h} m_n^h(t \wedge \tau_n)|^2_{\L_h^2} 
			- \frac{1}{2}|\partial^{h} m_0|^2_{\L_h^2} \\
			&= \int_0^{t \wedge \tau_n} \lb -\Delta^h m_n^h(s), F^h(m_n^h(s)) + \frac{1}{2} S^h(m_n^h(s)) \rb_{\L_h^2} ds \\
			&\quad + \int_0^{t \wedge \tau_n} \frac{1}{2} \sum_j q_j^2 \left| \partial^{h}\left( f_j G^h(m_n^h(s))\right) \right|_{\L_h^2}^2 ds 
			+ \int_0^{t \wedge \tau_n} \lb -\Delta^h m_n^h(s), G^h(m_n^h(s)) \ dW(s) \rb_{\L_h^2} .
		\end{align*}
		Since $ |m_n^h(t,x)|=1 $ for $ (t,x) \in [0,\tau_n] \times \Z_h $, and
		\begin{align*}
			|\Delta^h m_n^h(t)|_{\L_h^2}^2 + |\partial^{h} m_n^h(t)|_{\L_h^2}^4 \leq \frac{8}{h^2} |\partial^{h} m_n^h(t)|_{\L_h^2}^2, 
		\end{align*}
		there exist constants $ \beta_1 $ and $ \beta_2 $ that depend on $ h $ (not $ n $) such that
		\begin{align*}
			\lb -\Delta^h m_n^h(s), F^h(m_n^h(s)) + \frac{1}{2} S^h(m_n^h(s)) \rb_{\L_h^2} 
			+ \frac{1}{2} \sum_j q_j^2 \left|\partial^{h}\la f_j G^h(m_n^h(s)) \ra\right|_{\L_h^2}^2
			\leq \beta_1(h)  |\partial^{h} m_n^h(s)|_{\L_h^2}^2,
		\end{align*}
		and
		\begin{align*}
			\lb \Delta^h m_n^h(s), G^h(m_n^h(s)) \rb_{\L_h^2}^2
			&\leq \frac{1}{4}\la |m_n^h(s) \times \Delta^h m_n^h(s)|^2_{\L_h^2} + |\partial^h m_n^h(s)|^2_{\L_h^2} \ra^2 \\
			&\leq \beta_2(h) |\partial^{h} m_n^h(s)|_{\L_h^2}^2. 
		\end{align*}
		Then, by the boundedness of $ \kappa^2 $, the stochastic integral $ \int_0^{t \wedge \tau_n} \langle \Delta^h m_n^h(s), G^h(m_n^h(s))\ dW(s) \rangle_{\L_h^2} $ is a square integrable continuous martingale for $ m_n^h \in \mathcal{E}_h $, for every $ h>0 $. 
		Now, we have
		\begin{align*}
			|\partial^{h} m_n^h(t \wedge \tau_n)|^2_{\L_h^2} 
			&\leq  |\partial^{h} m_0|^2_{\L_h^2} + 2\int_0^{t \wedge \tau_n} \beta_1(h) |\partial^{h} m_n^h(s)|^2_{\L_h^2} \ ds \\ 
			&\quad - 2\int_0^{t \wedge \tau_n} \lb \Delta^h m_n^h(s), G^h(m^h_n)(s) \ dW(s) \rb_{\L_h^2} ,
		\end{align*}
		where the stochastic integral part vanishes after taking expectation. 
		We obtain
		\begin{align*}
			\mathbb{E} \left[ |m^h_n(t \wedge \tau_n)|^2_{\mathbb{E}_h} \right]
			&\leq \mathbb{E} \left[ |m_0|^2_{\mathbb{E}_h} + 2\beta_1(h) \int_0^{t } |m^h_n(s\wedge \tau_n)|^2_{\mathbb{E}_h} \ ds \right].
		\end{align*}		
		By Gr{\"o}nwall's lemma,
		\begin{equation}\label{pf: E[mnh] Gronwall} 
			\mathbb{E} \left[ |m^h_n(t \wedge \tau_n)|^2_{\mathbb{E}_h} \right] \leq \mathbb{E} \left[ |m_0|^2_{\mathbb{E}_h} \right] \exp\la \int_0^{t} 2\beta_1(h) ds \ra \leq K(h,t).
		\end{equation}
		By the definition of $ \tau_n $, the left-hand-side of \eqref{pf: E[mnh] Gronwall} is greater than $ n^2 \mathbb{P}(\tau_n \in [0,t]) $, thus 
		\begin{align*}
			\lim_{n \to \infty} \mathbb{P}(\tau_n \in [0,t]) \leq \lim_{n \to \infty} K(h,t) n^{-2} = 0.
		\end{align*}
		In other words, $ \tau_n \to \infty $, $ \mathbb{P} $-a.s, as $ n \to \infty $. 
		Thus, the process 
		$ m^h(t) = \lim_{n \to \infty} m_n^h(t \wedge \tau_n) $ is the unique solution of the semi-discrete scheme \eqref{SDE: discrete}. 
	\end{proof}

\subsection{Uniform estimates for the solution $ m^h $ of the discrete SDE}\label{Section: uniform estimates}
	For every $ h>0 $, let 
	\begin{align*}
		M^h(t) := \int_0^t \langle \Delta^h m^h(s), G^h(m^h(s))\ dW(s) \rangle_{\L_h^2}.
	\end{align*}
	In the following lemma, we deduce an upper bound of the stochastic integral $ M^h(t) $ which is used in the proof of Lemma~\ref{Lemma: D+, mh x D+ est} to obtain uniform estimates for $ m^h $.
	\begin{lemma}\label{Lemma: Gh est}
		For any $ p \in (0,\infty) $, there exists a constant $ b_p $ independent of $ h $ such that 
		\begin{align*}
			\mathbb{E} \left[ \sup_{t \in [0,T]} \left|M^h(t) \right|^p \right] 
			&\leq \frac{1}{2} b_p (1+|\gamma|)^p C_{\kappa}^p \ \E \left[ \sup_{t \in [0,T]} |\partial^{h} m^h(t)|_{\L_h^2}^{2p} + \la \int_0^T |m^h(t) \times \Delta^h m^h(t)|_{\L_h^2}^2 \ dt \ra^p \right].
		\end{align*}
	\end{lemma}
	\begin{proof}
		We observe that for every $ j \geq 1 $, 
		\begin{align*}
			&\langle \Delta^h m^h(t), q_j f_j G^h(m^h(t)) \rangle_{\L_h^2}  \\
			&= h \sum_x q_j f_j \langle m^h \times \Delta^h m^h, m^h \times \partial^h m^h + \gamma \partial^h m^h \rangle (t,x)\\
			&\leq (1+|\gamma|) h \sum_x \left|q_j f_j  \partial^h m^h \right| \left| m^h \times \Delta^h m^h \right| (t,x) \\
			&\leq (1+|\gamma|) \la h \sum_x q_j^2 f_j^2 |\partial^h m^h|^2(t,x) \ra^{\frac{1}{2}} \la h \sum_x |m^h \times \Delta^h m^h|^2(t,x) \ra^\frac{1}{2},
		\end{align*}
		which implies
		\begin{align*}
			\sum_j \langle \Delta^h m^h(t), q_j f_j G^h(m^h(t)) \rangle_{\L_h^2}^2 
			&\leq (1+|\gamma|)^2 \ |m^h(t) \times \Delta^h m^h(t)|_{\L_h^2}^2 \ \la h \sum_x \sum_j q_j^2 f_j^2 |\partial^h m^h|^2(t,x) \ra \\
			&\leq (1+|\gamma|)^2 \ C_\kappa^2 \ |m^h(t) \times \Delta^h m^h(t)|_{\L_h^2}^2 \ |\partial^h m^h(t)|_{\L_h^2}^2.
		\end{align*}
		Then as in the proof of Lemma \ref{Lemma: solution of disSDE}, for every fixed $ h $,
		\begin{align*}
			\sum_j \langle \Delta^h m^h(t), q_j f_j G^h(m^h(t)) \rangle_{\L_h^2}^2 
			< \infty,\quad \mathbb{P} \text{-a.s.}
		\end{align*}
		implying that $ M^h(t) $ is a continuous martingale. 	
		By the Burkholder-Davis-Gundy inequality, for $ p \in(0,\infty) $, there exists a constant $ b_p $ such that
		\begin{align*}
			\E \left[ \sup_{t \in [0,T]} | M^h(t)  |^{p} \right] 
			&\leq b_{p} \E\left[ \la \int_0^T \sum_j \lb \Delta^h m^h(t), q_j f_j G^h(m^h(t)) \rb_{\L_h^2}^2 \ dt \ra^{\frac{p}{2}} \right] \\
			&\leq b_{p} (1+|\gamma|)^p C_{\kappa}^{p} \ \E \left[ \la \int_0^T \ |\partial^{h} m^h(t)|_{\L_h^2}^2 \ |m^h(t) \times \Delta^h m^h(t)|_{\L_h^2}^2 \ dt \ra^{\frac{p}{2}} \right]. 
		\end{align*}
		Taking the supremum over $ t $ for $ |\partial^{h} m^h(t)|_{\L_h^2}^2 $, 
		\begin{align*}
			&\E \left[ \sup_{t \in [0,T]} | M^h(t)  |^{p} \right] \\\
			&\leq b_p (1+|\gamma|)^p C_{\kappa}^{p} \ \E \left[  \ \sup_{t \in [0,T]} |\partial^{h} m^h(t)|_{\L_h^2}^p \ \la \int_0^T |m^h(t) \times \Delta^h m^h(t)|_{\L_h^2}^2 \ dt \ra^{\frac{p}{2}}  \right] \\				
			&\leq \frac{1}{2} b_p (1+|\gamma|)^p C_{\kappa}^{p} \ \E \left[ \sup_{t \in [0,T]} |\partial^{h} m^h(t)|_{\L_h^2}^{2p} + \la \int_0^T |m^h(t) \times \Delta^h m^h(t)|_{\L_h^2}^2 \ dt \ra^p \right].
		\end{align*}
	\end{proof}
	
	\begin{lemma}\label{Lemma: D+, mh x D+ est}
		For any $ p \in [1,\infty) $, assume that 
		$\{(q_j,f_j)\}_{j\geq 1}$
		satisfies
		\begin{equation}\label{defn: N1p, N2p>0}
		\begin{aligned}
			N_{1,p} &:= 1 - 4^{p-1} b_p (1+|\gamma|)^p C_{\kappa}^{p} > 0,\\
			N_{2,p} &:= 2^p \la \alpha - (1+2\gamma^2) C_\kappa^2 - \delta \ra^p - 4^{p-1} b_p (1+|\gamma|)^p C_{\kappa}^p > 0,
		\end{aligned}
	\end{equation}
		for some small $ \delta > 0 $, where $ b_p $ is the constant in Lemma \ref{Lemma: Gh est}. 
		Let $ |m_0|_{\E_h} \leq K_0 $. 
		Then, there exist constants $ K_{1,p} $ and $ K_{2,p} $ that are independent of $ h $, such that
		\begin{equation}
			\E\left[ \sup_{t \in [0,T]} |\partial^{h} m^h(t)|_{\L_h^2}^{2p} \right] \leq K_{1,p}, \label{estimate: E|D+mh|2p} 
		\end{equation}
		\begin{equation}
			\E \left[ \la \int_0^T |m^h(s) \times \Delta^h m^h(s) |_{\L_h^2}^2 \ ds \ra^p \right] \leq K_{2,p}, \label{estimate: E(int|mh x D2mh|)^p}
		\end{equation}
		for all $ h> 0 $. 
	\end{lemma}
	\begin{proof}
	As in Lemma \ref{Lemma: solution of disSDE}, let $ \phi(u) = \frac{1}{2} | \partial^{h} u|_{\L_h^2}^2 $ for $ u \in \mathbb{E}_h $. Then, for $ \nu,w \in \mathring{\H}_h^1 $,
	\begin{align*}
		\phi'(u) \nu &= \langle \partial^{h} u, \partial^{h} \nu \rangle_{\L_h^2} = - \langle \Delta^h u, \nu \rangle_{\L_h^2}, \quad
		\phi''(u)(\nu,w) = \langle \partial^{h}\nu, \partial^{h} w \rangle_{\L_h^2}.
	\end{align*}
	By It{\^o}'s lemma,
	\begin{equation}\label{pf: phi Ito}
		\begin{aligned}
			\frac{1}{2} | \partial^{h} m^h(t)|_{\L_h^2}^2- \frac{1}{2} | \partial^{h} m^h(0)|_{\L_h^2}^2
			&=\phi(m^h(t)) - \phi(m^h(0)) \\
			&= - \int_0^t \lb \Delta^h m^h(s), F^h(m^h)(s)  \rb_{\L_h^2} ds \\
			&\quad - \int_0^t \lb \Delta^h m^h(s), \frac{1}{2} S^h(m^h)(s)\rb_{\L_h^2} ds \\
			&\quad + \frac{1}{2} \int_0^t \sum_j | \partial^{h} (q_j f_j G^h(m^h)(s) ) |_{\L_h^2}^2 \ ds \\
			&\quad - \int_0^t \langle \Delta^h m^h(s), G^h(m^h)(s) \ dW(s)  \rangle_{\L_h^2} \\
			&:= \int_0^t \la T_1(s) + T_2(s) + T_3(s)\ra \,ds - M^h(t),
		\end{aligned}
	\end{equation}
	where $ M^h(t) $ is already estimated in Lemma \ref{Lemma: Gh est}. 
		
	\underline{An estimate on $T_1$:}
	\begin{equation}\label{pf: est T11}
	\begin{aligned}
		T_1(s) &= -\langle \Delta^h m^h(s), F^h(m^h)(s) \rangle_{\L_h^2} \\
		&= \alpha \lb \Delta^h m^h(s), m^h(s) \times \la m^h(s) \times \Delta^h m^h(s)\ra \rb_{\L_h^2} \\
		&\quad -\lb \Delta^h m^h(s),  m^h(s) \times \la m^h(s) \times v^h(s)\partial^{h} m^h(s) + \gamma v^h(s)\partial^{h} m^h(s)\ra \rb_{\L_h^2} \\
		&= - \alpha |m^h(s) \times \Delta^h m^h(s) |^2_{\L_h^2} \\
		&\quad + \lb m^h(s) \times \Delta^h m^h(s), m^h(s) \times v^h(s)\partial^{h} m^h(s) + \gamma v^h(s)\partial^{h} m^h(s) \rb_{\L_h^2}. 
	\end{aligned}
	\end{equation}
	The second term on the right hand side of \eqref{pf: est T11} is estimated using \eqref{Cv bound} and the fact that $ |m^h|=1 $ $ \P $-a.s., as follows
	\begin{equation}\label{pf: <D2mh, Fh>}
		\begin{aligned}
			&\lb m^h(s) \times \Delta^h m^h(s), m^h(s) \times v^h(s)\partial^{h} m^h(s) + \gamma v^h(s)\partial^{h} m^h(s) \rb_{\L_h^2} \\
			&\leq  \epsilon^2 |m^h(s) \times \Delta^h m^h(s)|_{\L_h^2}^2 + \frac{1}{2\epsilon^2} C_v^2 (1+\gamma^2) |\partial^{h} m^h(s)|_{\L_h^2}^2,
		\end{aligned}	
	\end{equation}
		for arbitrary $ \epsilon>0 $. 
		An estimate of $T_1$ is obtained from \eqref{pf: est T11} and \eqref{pf: <D2mh, Fh>}
		\begin{align}\label{pf: est T1}
		T_1 \leq (\epsilon^2 - \alpha )|m^h(t) \times \Delta^h m^h(t)|_{\L_h^2}^2 + \frac{1}{2\epsilon^2} C_v^2 (1+\gamma^2) |\partial^{h} m^h(t)|_{\L_h^2}^2.
		\end{align}
		
	\underline{An estimate on $T_2$:}
	\begin{align*}
		T_2(s) 
		&=-\frac12\langle \Delta^h m^h(s), S^h(m^h(s)) \rangle_{\L_h^2} \\
		&= \frac{1}{2}\lb m^h\times\Delta^h m^h, \kappa \kappa' \la m^h\times G^h(m^h)+ \gamma G^h(m^h)\ra \rb_{\L_h^2} \\
		&\quad - \frac{1}{2} \lb \Delta^h m^h, \mathcal{G}^h_\kappa(m^h) \rb_{\L_h^2} \\
		&= T_{21} + T_{22}.
	\end{align*}
	Using $ |m^h|=1 $, $ \P $-a.s., we estimate $T_{21}$:
	\begin{align}\label{pf: est T21}
		T_{21}
		&= \frac{1}{2} \lb m^h \times \Delta^h m^h, \kappa \kappa' \left[ (\gamma^2-1) m^h \times \partial^h m^h + 2 \gamma m^h \times (m^h \times \partial^h m^h) \right] \rb_{\L_h^2}\nonumber \\
		&\leq \frac{1}{2} \epsilon^2 |m^h(s)\times\Delta^h m^h(s)|^2_{\L_h^2} +  \frac{1}{4\epsilon^2}((\gamma^2-1)^2+ 4\gamma^2) \ |\kappa \kappa'|_{\L_h^\infty}^2 |\partial^{h} m^h(s)|^2_{\L_h^2}.
	\end{align}
	where $ |\kappa \kappa'|_{\L_h^\infty}^2 \leq C_\kappa^4 $ by \eqref{Ckappa bound}. 

	To estimate $T_{22}$, we first note that for any $u:\Z_h \to {\b S}^2 $, 
	\begin{align*}
		\mathcal{G}^h_\kappa(u) 
		&= \frac{1}{2}\la (\kappa^2)^- + \kappa^2 \ra \left[(\gamma^2 -1) u \times \la u \times \Delta^h u\ra - 2 \gamma u \times \Delta^h u \right] \\
		&\quad - \gamma^2 \kappa^2 \partial^h u \times (u \times \partial^h u) \\
		&\quad - \kappa^2 |u \times \partial^h u|^2 u + 2\gamma (\kappa^2)^- \langle u, (\partial^h u)^- \rangle u \times \partial^h u.
	\end{align*}
	By \eqref{dis udotDelu}, we deduce
	\begin{align}\label{pf: est T22 1}
		T_{22}(s) 
		&= - \frac12h \sum_x \lb \Delta^h m^h, \mathcal{G}^h_\kappa (m^h)\rb(s,x)\nonumber\\
		&= \frac{1}{4}(\gamma^2-1) h \sum_x \la (\kappa^2)^- + \kappa^2 \ra |m^h\times \Delta^h m^h|^2 (s,x)\nonumber\\
		&\quad + \frac{1}{2}\gamma^2 h \sum_x \kappa^2 \lb \Delta m^h, \partial^h m^h \times (m^h \times \partial^h m^h) \rb (s,x) \nonumber \\
		&\quad - \frac{1}{4} h \sum_x \kappa^2 |m^h \times \partial^h m^h|^2 \la|\partial^{h} m^h|^2 + |(\partial^{h} m^h)^-|^2\ra (s,x)\nonumber\\
		&\quad - \gamma h \sum_x (\kappa^2)^- \lb \Delta^h m^h, m^h \times \partial^h m^h \rb \lb m^h, (\partial^h m^h)^- \rb (s,x) \nonumber \\
		&= T_{22a}(s) + T_{22b}(s) + T_{22c}(s) + T_{22d}(s).
	\end{align}
	It is clear that 
	\begin{align}\label{pf: est T22a}
		T_{22a}(s) 
		&= \frac{1}{4}(\gamma^2-1) h \sum_x \la (\kappa^2)^- + \kappa^2 \ra |m^h\times \Delta^h m^h|^2 (s,x) \nonumber \\
		&\leq \frac{1}{2} \gamma^2 |\kappa|^2_{\L_h^\infty} |m^h \times \Delta^h m^h|^2_{\L_h^2}(s) - \frac{1}{4} h \sum_x \la (\kappa^2)^- + \kappa^2 \ra |m^h\times \Delta^h m^h|^2 (s,x),
	\end{align}
	where the second term on the right-hand side will cancel with parts of $ T_3 $. 
	
	To estimate $ T_{22b} $, we observe that for any $ u: \Z_h \to {\b S}^2 $, using \eqref{dis udotDelu},
	\begin{align}\label{pf: est T22b 1}
		\lb \Delta^h u, \partial^h u \times (u \times \partial^h u) \rb 
		&= |\partial^h u|^2 \lb u, \Delta^h u \rb - \lb u, \partial^h u \rb \lb \partial^h u, \Delta^h u \rb \nonumber \\
		&= -\frac{1}{2} |\partial^h u|^2 \la |\partial^h u|^2 + |\partial^h u^-|^2 \ra + \frac{1}{2} |\partial^h u|^2 \la |\partial^h u|^2 - \lb \partial^h u, \partial^h u^- \rb \ra \nonumber \\
		&= -\frac{1}{2} |\partial^h u|^2 |\partial^h u^-|^2 - \frac{1}{2} |\partial^h u|^2 \lb \partial^h u, \partial^h u^- \rb,
	\end{align}
	where
	\begin{equation}\label{pf: est T22b 1-1}
		\lb \partial^h u, \partial^h u^- \rb 
		= \frac{1}{2} \la |\partial^h u|^2 + |\partial^h u^-|^2 -h^2 |\Delta^h u|^2 \ra. 
	\end{equation}
	If $ |\partial^h u(x)| \leq |\partial^h u^-(x)| $ at some $ x \in \Z_h $, then 
	\begin{align*}
		-\lb \partial^h u, \partial^h u^- \rb(x) \leq |\partial^h u^-(x)|^2,
	\end{align*}
	and by \eqref{pf: est T22b 1},
	\begin{align*}
		\lb \Delta^h u, \partial^h u \times (u \times \partial^h u) \rb(x) 
		&\leq -\frac{1}{2} |\partial^h u(x)|^2 |\partial^h u^-(x)|^2 + \frac{1}{2} |\partial^h u(x)|^2 |\partial^h u^-(x)|^2 \\
		&= 0.
	\end{align*}
	If $ |\partial^h u(x)| \geq |\partial^h u^-(x)| $, then we can show that the term given by \eqref{pf: est T22b 1} is bounded by $ |u \times \Delta^h u|^2(x) $. Explicitly, by \eqref{pf: est T22b 1-1},
	\begin{align*}
		&|u \times \Delta^h u|^2(x) - \lb \Delta^h u, \partial^h u \times (u \times \partial^h u) \rb (x)\\
		&= |\Delta^h u|^2 - \lb u, \Delta^h u \rb^2 + \frac{1}{2} |\partial^h u|^2 |\partial^h u^-|^2 + \frac{1}{2} |\partial^h u|^2 \lb \partial^h u, \partial^h u^- \rb \\
		&= |\Delta^h u|^2 - \frac{1}{4} \la |\partial^h u|^2 + |\partial^h u^-|^2 \ra^2 + \frac{1}{2} |\partial^h u|^2 |\partial^h u^-|^2 + \frac{1}{4} |\partial^h u|^2 \la |\partial^h u|^2 + |\partial^h u^-|^2 -h^2 |\Delta^h u|^2 \ra \\
		&= \la 1- \frac{1}{4}h^2 |\partial^h u|^2 \ra |\Delta^h u|^2 - \frac{1}{4} |\partial^h u^-|^4 + \frac{1}{4} |\partial^h u|^2 |\partial^h u^-|^2 \\
		& \geq 0,
	\end{align*}
	where the last inequality holds by $ h^2 |\partial^h u|^2 \leq 4 $ and $ |\partial^h u(x)| \geq |\partial^h u^-(x)| $. 
	Combining the two cases and replacing $ u $ by $ m^h(s) $, we have
	\begin{align}\label{pf: est T22b}
		T_{22b}(s) 
		&= \frac{1}{2} \gamma^2 h \sum_x \kappa^2 \lb \Delta^h m^h, \partial^h m^h \times (m^h \times \partial^h m^h) \rb(s,x) \nonumber \\
		&\leq \frac{1}{2} \gamma^2 h \sum_x \kappa^2 |m^h \times \Delta^h m^h|^2(s,x) \ \mathbbm{1}_{\{ |\partial^h m^h(s,x)| \geq |\partial^h {m^h}^-(s,x)| \}} \nonumber \\
		&\leq \frac{1}{2} \gamma^2 C_\kappa^2 \ |m^h \times \Delta^h m^h|^2_{\L_h^2}(s).
	\end{align}
	We will see later in the proof that $ T_{22c} $ and $ T_{22d} $ also cancel with parts of $ T_3 $.

	\underline{An estimate on $ T_3 $:}			
	\begin{align*}
		T_3(s)
		&= \frac{1}{2} \sum_j \left| \partial^{h} (q_jf_j G^h(m^h)) \right|^2_{\L_h^2} (s) \\
		&= \frac{1}{2} \sum_j q_j^2 \left| \partial^{h} (f_j G^h(m^h)) \right|^2_{\L_h^2} (s) \\
		&= \frac{1}{2} h \sum_x \sum_j q_j^2 \left| \partial^{h} \la f_j m^h \times (m^h \times \partial^{h} m^h)  \ra \right|^2(s,x)  
				+ \frac{1}{2} \gamma^2 h \sum_x \sum_j q_j^2 \left| \partial^{h} (f_j m^h \times \partial^{h}m^h) \right|^2(s,x) \\
		&\quad + \gamma h \sum_x \sum_j q_j^2 \lb \partial^{h} \left(f_j m^h \times (m^h \times \partial^{h} m^h) \right), \partial^{h} \left(f_j m^h \times \partial^{h} m^h \right) \rb(s,x) \\
		&= T_{31}(s) + T_{32}(s) + T_{33}(s). 
	\end{align*}

	We first estimate $ T_{31}(s) $. 
	For $ u : \Z_h \to {\b S}^2 $, we have for every $ j \geq 1 $, 
	\begin{align*}
		&\partial^{h} \la f_j u \times (u \times \partial^{h} u) \ra(x) \\
		&= \frac{1}{2} \la \partial^{h} f_j \ u^+ \times (u^+ \times \partial^{h} u^+) + f_j \partial^{h} \la u \times (u \times \partial^{h} u) \ra \ra (x) \\
		&\quad + \frac{1}{2} \la \partial^{h}f_j \ u \times (u \times \partial^{h} u) + f_j^+ \partial^{h} \la u \times (u \times \partial^{h} u) \ra \ra (x) \\
		&= \frac{1}{2} \partial^{h}f_j \la u^+ \times (u^+ \times \partial^{h} u^+) + u \times (u \times \partial^{h} u) \ra (x) \\
		&\quad + \frac{1}{2}\left[f_j \partial^{h}u \times(u \times \partial^{h} u) + f_j^+ u \times \partial^{h}(u \times \partial^{h} u) \right] (x) \\
		&\quad + \frac{1}{2}\left[ f_j^+ \partial^{h}u \times (u^+ \times \partial^{h} u^+) + f_j u^+ \times \partial^{h}(u \times \partial^{h} u) \right] (x) \\
		&= \frac{1}{2}A_0(x) + \frac{1}{2} \left[ A_1(x) + B_1(x) \right] + \frac{1}{2} \left[ A_2(x) + B_2(x) \right],
	\end{align*}
	where
	\begin{align*}
		A_0(x) &= \partial^{h}f_j \la u^+ \times (u^+ \times \partial^{h} u^+) + u \times (u \times \partial^{h} u) \ra (x), \\
		A_1(x) &= f_j \partial^{h}u \times(u \times \partial^{h} u), \\
		A_2(x) &= f_j^+ \partial^{h}u \times (u^+ \times \partial^{h} u^+), \\
		B_1(x) &= f_j^+ u \times \partial^{h}(u \times \partial^{h} u), \\
		B_2(x) &= f_j u^+ \times \partial^{h}(u \times \partial^{h} u). 
	\end{align*}
	Hence,
	\begin{align*}
		&\frac{1}{2} \left| \partial^{h} \la f_j u \times (u \times \partial^{h} u) \ra \right|^2(x) \\
		&= \frac{1}{2} \la \frac{1}{4} |A_0|^2(x) + \frac{1}{4} |A_1 + B_1 + A_2 + B_2|^2(x) + \frac{1}{2} \lb A_0, A_1 + A_2 + B_1 + B_2 \rb(x) \ra \\
		&\leq \frac{1}{2} \la \frac{1}{4} |A_0|^2(x) + \frac{1}{2} |A_1 + B_1|^2(x) + \frac{1}{2} |A_2 + B_2|^2(x) + \frac{1}{2} \lb A_0, A_1 + A_2 + B_1 + B_2 \rb(x) \ra \\
		&= \frac{1}{8} |A_0|^2(x) + \frac{1}{4} \la |A_1|^2 + |B_1|^2 + |A_2|^2 + |B_2|^2 \ra(x) \\
		&\quad + \frac{1}{2}\lb A_1, B_1 \rb(x) + \frac{1}{2} \lb A_2, B_2 \rb(x) + \frac{1}{4} \lb A_0, A_1 + A_2 + B_1 +B_2 \rb(x).
	\end{align*}	
	For the square $ \frac{1}{8} |A_0|^2(x) $:
	\begin{equation}\label{pf: Ito correction |A0|^2}
		\begin{aligned}
			\frac{1}{8} h \sum_x \sum_j q_j^2 |A_0|^2(x) 
			&= \frac{1}{8} h \sum_x \sum_j q_j^2 | \partial^{h} f_j |^2  \ \left| u^+ \times (u^+ \times \partial^{h} u^+) + u \times (u \times \partial^{h} u) \right|^2(x) \\
			&\leq \frac{1}{8} C_\kappa^2 \ h \sum_x \left| u^+ \times (u^+ \times \partial^{h} u^+) + u \times (u \times \partial^{h} u) \right|^2(x) \\
			&\leq \frac{1}{2} C_\kappa^2 \ |\partial^{h} u|_{\L_h^2}^2, 
		\end{aligned}	
	\end{equation}
	where the second inequality holds by applying the Mean Value Theorem to $ f_j $ on the interval $ [x,x+h] $ for every $ j \geq 1 $, such that there exists some $ \xi_h \in (x,x+h) $ satisfying
	\begin{align*}
		\left|\frac{f_j(x+h)-f_j(x)}{h} \right| = |f_j'(\xi_h)|,
	\end{align*}
	and $ |\sum_j q_j^2 (f_j')^2|_{\L^\infty} \leq C_\kappa^2 $ by assumption \eqref{Ckappa bound}. 
		
	For the squares $ \frac{1}{4}|A_1|^2(x) $ and $ \frac{1}{4}|A_2|^2(x) $: 
	\begin{equation}\label{pf: Ito correction |A1,A2|^2 power 4}
		\begin{aligned}	
			& \frac{1}{4} h \sum_x \sum_j q_j^2 \la |A_1|^2(x) + |A_2|^2(x) \ra \\
			&= \frac{1}{4} \sum_x \sum_j q_j^2 \la \left| f_j \partial^{h}u \times (u \times \partial^{h} u) \right|^2 + \left| f_j^+ \partial^{h} u \times (u^+ \times \partial^{h} u^+) \right|^2 \ra(x) \\
			&\leq \frac{1}{4} h \sum_x \sum_j q_j^2 f_j^2 \ |\partial^{h}u|^2 \ |u \times \partial^{h}u|^2(x) + \frac{1}{4} h \sum_x \sum_j q_j^2 (f_j^+)^2 \ |\partial^{h}u|^2 \ |u^+ \times \partial^{h} u^+|^2(x) \\
			&= \frac{1}{4} h \sum_x \la \sum_j q_j^2 f_j^2 \ra |u \times \partial^{h} u|^2 \la |\partial^{h}u|^2 + |\partial^h u^-|^2 \ra(x) \\
			&= \frac{1}{4} h \sum_x \kappa^2 |u \times \partial^{h} u|^2 \la |\partial^{h}u|^2 + |\partial^h u^-|^2 \ra(x),
		\end{aligned}
	\end{equation}
	where the right-hand side cancels with $ T_{22c}(s) $ in \eqref{pf: est T22 1} when $ u $ is replaced with $ m^h(s) $. 
		
	For the squares $ \frac{1}{4}|B_1|^2(x) $ and $ \frac{1}{4}|B_2|^2(x) $, we first observe that
	\begin{equation}\label{pf: D+h(u x Du) in terms of u x D2hu}
	\begin{aligned}
		\partial^{h} \la u \times \partial^{h}u\ra(x) 
		&= u^+\times \Delta^hu^+.
	\end{aligned}
	\end{equation}
	Then, 
	\begin{align*}
		&\frac{1}{4}\la |B_1|^2(x) + |B_2|^2(x) \ra \\
		&= \frac{1}{4} (f_j^+)^2 \left|u \times \la u^+ \times \Delta^h u^+  \ra \right|^2(x) 
		+ \frac{1}{4} f_j^2 \left|u^+ \times \la u^+ \times \Delta^h u^+ \ra \right|^2(x) \\
		&\leq \frac{1}{4} \la (f_j^+)^2 + f_j^2 \ra \left| u^+ \times \Delta^h u^+  \right|^2(x).
	\end{align*}
	This implies that
	\begin{equation}\label{pf: Ito correction |B|^2}
	\begin{aligned}
		\frac{1}{4} h \sum_x \sum_j q_j^2 \la |B_1|^2(x) + |B_2|^2(x) \ra 
		&= \frac{1}{4} h \sum_x \sum_j q_j^2 \la (f_j^-)^2 + f_j^2 \ra \left| u \times \Delta^h u  \right|^2(x) \\
		&= \frac{1}{4} h \sum_x \la (\kappa^2)^- + \kappa^2 \ra \left| u \times \Delta^h u  \right|^2(x), 
	\end{aligned}
	\end{equation}
	where the right-hand side cancels with a part of the estimate for $ T_{22a} $ in \eqref{pf: est T22a} when $ u = m^h(s) $ as aforementioned. 
		
	For the cross terms $ \frac{1}{2} \lb A_1,B_1 \rb(x) $ and $ \frac{1}{2} \lb A_2,B_2 \rb(x) $:
	\begin{align*}
		\lb A_1, B_1 \rb(x) 
		&= f_j f_j^+ \lb \partial^{h} u \times (u \times \partial^{h}u), u \times \partial^{h}(u \times \partial^{h} u) \rb(x) \\
		&= f_j f_j^+ \lb |\partial^{h} u|^2 u - \langle \partial^{h}u, u \rangle \partial^hu, u \times \partial^{h}(u \times \partial^{h} u) \rb(x) \\
		&= f_j f_j^+ \langle u, \partial^{h} u \rangle \lb u \times \partial^{h}u, \partial^{h}(u \times \partial^{h}u) \rb(x),
	\end{align*}
	and similarly,
	\begin{align*}
		\lb A_2, B_2 \rb(x)
		&= f_j f_j^+ \lb \partial^{h}u \times (u^+ \times \partial^{h}u^+), u^+ \times \partial^{h}(u \times \partial^{h}u) \rb(x) \\
		&= f_j f_j^+ \langle u^+, \partial^{h}u \rangle \lb u^+ \times \partial^{h}u^+, \partial^{h}(u \times \partial^{h}u) \rb(x).
	\end{align*}
	Then,
	\begin{equation}\label{pf: <A,B> cal}
		\begin{aligned}
			&\lb A_1,B_1 \rb + \lb A_2,B_2 \rb(x) \\
			&= f_j f_j^+ \lb \langle u, \partial^{h} u \rangle u \times \partial^{h}u + \langle u^+, \partial^{h}u \rangle u^+ \times \partial^{h}u^+, \ \partial^{h}(u \times \partial^{h}u) \rb(x).
		\end{aligned}
	\end{equation}
	By \eqref{pf: D+h(u x Du) in terms of u x D2hu},
	the left term in the inner product \eqref{pf: <A,B> cal} can be simplified as
	\begin{align*}
		&\langle u, \partial^{h} u \rangle u \times \partial^{h}u + 
		\langle u^+, \partial^{h}u \rangle u^+ \times \partial^{h}u^+ \\
		&= \la \langle u,\partial^{h}u \rangle + \langle u^+, \partial^{h}u \rangle \ra u \times \partial^{h}u 
		+ \langle u^+, \partial^{h} u \rangle h \la u^+ \times \Delta^hu^+ \ra \\
		&= \langle u^+, u^+-u \rangle \la u^+ \times \Delta^hu^+ \ra,
	\end{align*}
	where the second equality holds by observing $ \langle u+u^+, \partial^{h}u \rangle(x) = 0 $ due to $ |u(x)| = 1 $ for all $ x $. 
	Recall that the right term in the inner product \eqref{pf: <A,B> cal} is $ \partial^{h}(u \times \partial^{h}u) = u^+ \times \Delta^hu^+ $ by \eqref{pf: D+h(u x Du) in terms of u x D2hu}.  
	Then, 
	\begin{align*}
		\frac{1}{2} \la \lb A_1,B_1 \rb + \lb A_2,B_2 \rb(x) \ra
		&= \frac{1}{2} f_j f_j^+ \langle u^+, u^+-u \rangle \left| u^+ \times \Delta^hu^+ \right|^2(x) \\
		&\leq \frac{1}{2} \la f_j^2 + (f_j^+)^2 \ra |u^+ \times \Delta^hu^+|^2(x)
	\end{align*}
	Taking the sum over $ x \in \Z_h $, 
	\begin{equation}\label{pf: <A,B>}
	\begin{aligned}
		\frac{1}{2} h \sum_x \sum_j q_j^2 \la \lb A_1, B_1 \rb + \lb A_2, B_2 \rb \ra(x)
		&\leq \frac{1}{2} h \sum_x \la \kappa^2 + (\kappa^2)^+ \ra  |u^+ \times \Delta^hu^+|^2(x) \\
		&\leq C_\kappa^2 \ |u \times \Delta^hu|_{\L_h^2}^2. 
	\end{aligned}
	\end{equation}
	
	For the cross term $ \frac{1}{4}\lb A_0, \ A_1 + A_2 + B_1 + B_2\rb(x) $:
	\begin{align*}
		A_0(x) &= \frac{f_j^+-f_j}{h} \la u^+ \times (u^+ \times \partial^{h}u^+) + u \times(u \times \partial^{h} u) \ra(x),
	\end{align*}
	and 
	\begin{align*}
		(A_1 + A_2 + B_1 + B_2)(x)
		&= (f_j+f_j^+) \partial^{h}\la u \times (u \times \partial^{h}u) \ra(x) \\
		&= \frac{f_j+f_j^+}{h} \la u^+ \times (u^+ \times \partial^{h}u^+) - u \times(u \times \partial^{h} u) \ra(x),
	\end{align*}
	which imply
	\begin{align*}
		&\frac{1}{4} \lb A_0, \ A_1 + A_2 + B_1 + B_2 \rb(x) \\
		&= \frac{1}{4h^2} \la (f_j^+)^2 - f_j^2 \ra \la |u^+ \times (u^+\times \partial^{h}u^+)|^2 - |u \times(u \times \partial^{h}u)|^2 \ra(x). 
	\end{align*}
	Then, using again the Mean Value Theorem for $ \Delta^h (f_j)^2 $, 
	\begin{equation}\label{pf: <A0,A1+A2+B1+B2>}
		\begin{aligned}
			&\frac{1}{4} h \sum_x \sum_j q_j^2 \lb A_0, \ A_1 + A_2 + B_1 + B_2 \rb(x) \\
			&= \frac{1}{4} h \sum_x \sum_j q_j^2 \frac{1}{h^2}\la f_j^2 - (f_j^-)^2 - (f_j^+)^2 + f_j^2 \ra |u \times (u \times \partial^{h} u)|^2(x) \\
			&= -\frac{1}{4} h \sum_x \sum_j q_j^2 \Delta^h \la f_j^2 \ra |u \times (u \times \partial^{h} u)|^2(x) \\
			&\leq \frac{1}{2}C_\kappa^2 \ |\partial^{h} u|_{\L_h^2}^2.
		\end{aligned}
	\end{equation}
	
	Therefore, by \eqref{pf: Ito correction |A0|^2}, \eqref{pf: Ito correction |A1,A2|^2 power 4},  \eqref{pf: Ito correction |B|^2}, \eqref{pf: <A,B>} and \eqref{pf: <A0,A1+A2+B1+B2>},
	\begin{equation}\label{pf: est T31}
		\begin{aligned}
			T_{31}(s) 
			&= \frac{1}{2} h \sum_x \left| \partial^{h} \la f_j m^h \times (m^h \times \partial^{h}m^h) \ra \right|^2(s,x) \\
			&\leq C_\kappa^2 |\partial^{h} m^h|^2_{\L_h^2}(s) + C_\kappa^2 |m^h \times \Delta^h m^h|^2_{\L_h^2}(s) \\
			&\quad + \frac{1}{4} h \sum_x \kappa^2 |m^h \times \partial^{h}m^h|^2 \la |\partial^{h}m^h|^2 + |(\partial^h m^h)^-|^2 \ra(s,x) \\
			&\quad + \frac{1}{4} h \sum_x \la \kappa^2 +(\kappa^2)^- \ra |m^h \times \Delta^h m^h|^2(s,x).
		\end{aligned}
	\end{equation}

	Next, we estimate $ T_{32}(s) $. Using \eqref{pf: D+h(u x Du) in terms of u x D2hu},
	\begin{equation}\label{pf: est T32}
		\begin{aligned}
			T_{32}(s) 
			&= \frac{1}{2} \gamma^2 h \sum_x \sum_j q_j^2 |\partial^{h}(f_j m^h \times \partial^{h}m^h)|^2(s,x) \\
			&= \frac{1}{2} \gamma^2 h \sum_x \sum_j q_j^2 |(\partial^{h} f_j) m^h \times \partial^{h}m^h + f_j^+ (m^h \times \Delta^h m^h)^+|^2(s,x) \\
			&\leq \gamma^2 h \sum_x \sum_j q_j^2 |\partial^{h} f_j|^2 \  |\partial^{h}m^h|^2(s,x) + \gamma^2 h \sum_x \la \sum_j q_j^2 f_j^2 \ra |m^h \times \Delta^h m^h|^2(s,x) \\
			&\leq \gamma^2 C_\kappa^2 \la |\partial^{h} m^h|^2_{\L_h^2}(s) + |m^h \times \Delta^h m^h|^2_{\L_h^2} \ra. 
		\end{aligned}
	\end{equation}

	Finally, we estimate $ T_{33}(s) $. We note that for $ u = u(x) $ with $ |u(x)|=1 $ for all $ x $ and for all $ j \geq 1 $, 
	\begin{align*}
		&\lb \partial^{h} \la f_j u \times (u \times \partial^{h}u) \ra, \partial^{h} \la f_j u \times \partial^{h}u \ra \rb \\
		&= \lb (\partial^{h} f_j) u^+ \times (u^+ \times \partial^{h}u^+) + f_j \partial^{h}(u \times (u \times \partial^{h}u)), (\partial^{h} f_j) u^+ \times \partial^{h}u^+ +f_j (u \times \Delta^h u)^+ \rb \\
		&= (\partial^{h} f_j) f_j \lb u^+ \times (u^+ \times \partial^{h}u^+), (u \times \Delta^h u)^+ \rb \\
		&\quad + \lb f_j \partial^{h}u \times (u^+ \times \partial^{h}u^+) + f_j u \times (u \times \Delta^h u)^+, (\partial^{h} f_j) u^+ \times \partial^{h}u^+ +f_j (u \times \Delta^h u)^+ \rb \\
		&= (\partial^{h} f_j) f_j \lb (u^+-u) \times (u^+ \times \partial^{h}u^+), (u \times \Delta^h u)^+ \rb \\
		&\quad + f_j^2 \lb \partial^{h}u \times (u^+ \times \partial^{h}u^+), u^+ \times \Delta^h u^+ \rb. 
	\end{align*}
	Since $ \sum_j q_j^2 |f_j| |\partial^h f_j|(x) \leq C_\kappa^2 $ for all $ x \in \Z_h $, 
	we have 
	\begin{equation}\label{pf: est T33}
		\begin{aligned}
			T_{33}(s)
			&= \gamma h \sum_x \sum_j q_j^2 \lb \partial^{h} \la f_j m^h \times m^h \times \partial^{h}m^h \ra, \partial^{h}\la f_j m^h \times \partial^{h}m^h \ra \rb(s,x) \\
			&\leq |\gamma| h \sum_x \sum_j q_j^2 |\partial^{h} f_j^-| \ |f_j^-| \la \frac{1}{\epsilon^2} |\partial^{h}m^h|^2 + \epsilon^2 |m^h \times \Delta^h m^h|^2 \ra(s,x) \\
			&\quad + \gamma h \sum_x \sum_j q_j^2 (f_j^-)^2 \lb m^h, (\partial^h m^h)^- \rb \lb m^h \times \partial^{h}m^h, \Delta^h m^h \rb(s,x) \\
			&= |\gamma| C_\kappa^2 \la \frac{1}{\epsilon^2} |\partial^{h}m^h|^2_{\L_h^2}(s) + \epsilon^2  |m^h \times \Delta^h m^h|^2_{\L_h^2}(s) \ra - T_{22d}(s),
			\end{aligned}
	\end{equation}
	where $ T_{22d}(s) $ is given in \eqref{pf: est T22 1}.

	\underline{An estimate on $ T_1 + T_2 + T_3 $:}
	\begin{align*}
		T_1 + T_2 + T_3
		&= T_1 + T_{21} + T_{22a} + T_{22b} + T_{22c} + T_{22d} + T_{31} + T_{32} + T_{33},
	\end{align*}
	where by \eqref{pf: est T22a} and \eqref{pf: est T31}, 
	\begin{align*}
		T_{22a}(s) + T_{22c}(s) + T_{31}(s) 
		&\leq \la \frac{1}{2} \gamma^2 +1 \ra C_\kappa^2 |m^h \times \Delta^h m^h|^2_{\L_h^2}(s) + C_\kappa^2 |\partial^{h} m^h|^2_{\L_h^2}(s),
	\end{align*}
	and by \eqref{pf: est T33}, 
	\begin{align*}
		T_{22d}(s) + T_{33}(s)
		&\leq  |\gamma| C_\kappa^2 \la \frac{1}{\epsilon^2} |\partial^{h}m^h|^2_{\L_h^2}(s) + \epsilon^2  |m^h \times \Delta^h m^h|^2_{\L_h^2}(s) \ra.
	\end{align*}
	Then, by \eqref{pf: est T1}, \eqref{pf: est T21}, \eqref{pf: est T22b} and \eqref{pf: est T32}, 
	\begin{equation}\label{pf: F + Stratonovich + Ito correction}
		T_1(s) + T_2(s) + T_3(s)
		\leq \frac{1}{2}C_{1,\epsilon} \ |\partial^h m^h(s)|^2_{\L_h^2} +\la \frac{1}{2} C_{2,\epsilon} - \alpha \ra \ |m^h(s) \times \Delta^h m^h(s)|^2_{\L_h^2}, 
	\end{equation}
	where
	\begin{equation}\label{pf: C1epsilon, C2epsilon}
		\begin{aligned}
			C_{1,\epsilon} &:= \frac{1}{2\epsilon^2}\left[ \la \gamma^2 + 1 \ra^2 C_\kappa^4 + 4|\gamma| C_\kappa^2 + 2C_v^2 (1+\gamma^2) \right] + 2\la 1 + \gamma^2 \ra C_\kappa^2, \\
			C_{2,\epsilon} &:= \la 4 \gamma^2 +2 \ra C_\kappa^2 + \epsilon^2 \la 3 + 2|\gamma| C_\kappa^2 \ra. 
		\end{aligned}
	\end{equation}

	\underline{Uniform estimate of $ \partial^h m^h $}: 
		
	Using \eqref{pf: phi Ito} and \eqref{pf: F + Stratonovich + Ito correction}, we have
	\begin{equation}\label{pf: Gronwall ineq: |D+mh|^2<=...+|Mh|}
		\begin{aligned}
			&|\partial^h m^h(t)|^{2}_{\L_h^2} 
			+ (2\alpha-C_{2,\epsilon}) \int_0^t | m^h(s) \times \Delta^h m^h(s)|_{\L_h^2}^2 \ ds \\
			&\leq |\partial^h m_0|^2_{\L_h^2} 
			+ C_{1,\epsilon} \int_0^t |\partial^h m^h(s)|^2_{\L_h^2} \ ds 
			+ 2\sup_{t \in [0,T]} |M^h(t)|. 
		\end{aligned}
	\end{equation}
	Taking a sufficiently small $ \epsilon $ such that $ \frac{1}{2}\epsilon^2 \la 3 + 2|\gamma| C_\kappa^2 \ra < \delta $, we have from \eqref{defn: N1p, N2p>0}:
	\begin{equation}\label{pf: C2 vs alpha neg_coef}
		2\alpha-C_{2,\epsilon} > 0.
	\end{equation}
	Then, for $ p \geq 1 $ and $ q = \frac{p}{p-1} $, 
	\begin{align*}
		&\E \left[ \sup_{t \in [0,T]} |\partial^h m^h(t)|^{2p}_{\L_h^2} + (2\alpha - C_{2,\epsilon})^p \la \int_0^T |m^h(s) \times \Delta^h m^h(s)|_{\L_h^2}^2 \ ds \ra^p \right] \\
		&\leq \E \left[ \la \sup_{t \in [0,T]} |\partial^h m^h(t)|^{2}_{\L_h^2} + (2\alpha - C_{2,\epsilon}) \int_0^T |m^h(s) \times \Delta^h m^h(s)|_{\L_h^2}^2 \ ds \ra^p \right] \\
		&\leq \E \left[ \la |\partial^h m_0|^2_{\L_h^2} 
							+ C_{1,\epsilon} \int_0^T |\partial^h m^h(s)|^2_{\L_h^2} \ ds 
							+ 2 \sup_{t \in [0,T]} |M^h(t)| \ra^p \right] \\
		&\leq (2^{p-1})^2 \E \left[ |\partial^h m_0|_{\L_h^2}^{2p} + C_{1,\epsilon}^p T^{\frac{p}{q}} \int_0^T \sup_{t \in [0,s]} |\partial^h m^h(t)|_{\L_h^2}^{2p} \ ds \right] 
							+ 2^{p-1} \E \left[ 2^p \sup_{t \in [0,T]} |M^h(t)|^p \right] \\
		&\leq 4^{p-1} \E \left[  |\partial^h m_0|_{\L_h^2}^{2p} + C_{1,\epsilon}^p T^{\frac{p}{q}} \int_0^T \sup_{t \in [0,s]} |\partial^h m^h(t)|_{\L_h^2}^{2p} \ ds  \right] \\
		&\quad + 4^{p-1} b_p (1+|\gamma|)^p C_\kappa^p \ \E \left[ \sup_{t \in [0,T]} |\partial^h m^h(t)|_{\L_h^2}^{2p} + \la \int_0^T |m^h(t) \times \Delta^h m^h(t)|_{\L_h^2}^2 \ dt \ra^p \right],
	\end{align*}
	where the second inequality holds by \eqref{pf: Gronwall ineq: |D+mh|^2<=...+|Mh|} and the last inequality holds by Lemma \ref{Lemma: Gh est}.
	Then, by the definitions of $ N_{1,p} $ and $ N_{2,p} $, 
	\begin{equation}\label{pf: |D+mh| + int |mxD2mh| <= int |D+mh|}
		\begin{aligned}
			&N_{1,p} \E \left[ \sup_{t \in [0,T]} |\partial^h m^h(t)|_{\L_h^2}^{2p} \right] + N_{2,p} \E \left[ \la \int_0^T |m^h(s) \times \Delta^h m^h(s) |_{\L_h^2}^2 \ ds \ra^p \right] \\
			&\leq 4^{p-1} \E \left[  |\partial^h m_0|_{\L_h^2}^{2p} + C_{1,\epsilon}^p T^\frac{p}{q} \int_0^T \sup_{t \in [0,s]} |\partial^h m^h(t)|_{\L_h^2}^{2p} \ ds  \right].
		\end{aligned}
	\end{equation}
	Hence, by \eqref{defn: N1p, N2p>0} and \eqref{pf: |D+mh| + int |mxD2mh| <= int |D+mh|}, 
	\begin{equation}\label{Gronwall ineq: |D+m|^2p}
		\E \left[ \sup_{t \in [0,T]} |\partial^h m^h(t)|_{\L_h^2}^{2p} \right] 
		\leq N_{1,p}^{-1} 4^{p-1} \la \E \left[  |\partial^h m_0|_{\L_h^2}^{2p} \right] + C_{1,\epsilon}^p T^\frac{p}{q} \E \left[ \int_0^T \sup_{t \in [0,s]} |\partial^h m^h(t)|_{\L_h^2}^{2p} \ ds  \right] \ra. 
	\end{equation}
	By Fubini's theorem and Gr{\"o}nwall's inequality, 
	\begin{equation}\label{pf: E[D+h mh_Lh2] < K1,p}
		\E\left[ \sup_{t \in [0,T]} |\partial^{+h} m^h(t) |_{\L_h^2}^{2p} \right] 
		\leq N_{1,p}^{-1} 4^{p-1}  \E \left[ |\partial^h m_0|^{2p}_{\L_h^2} \right] \exp \la \int_0^T N_{1,p}^{-1} 4^{p-1} C_{1,\epsilon}^p T^\frac{p}{q} \ dt \ra = K_{1,p},
	\end{equation}
	where $ K_{1,p} $ depends on $ p, C_v, C_\kappa, \epsilon, T $ and $ K_0 $, but not on $ h $, proving \eqref{estimate: E|D+mh|2p}. 
	
	Finally, by \eqref{defn: N1p, N2p>0}, \eqref{pf: |D+mh| + int |mxD2mh| <= int |D+mh|} and \eqref{pf: E[D+h mh_Lh2] < K1,p},
	\begin{align*}
		\E \left[ \la \int_0^T  |m^h(s) \times \Delta^h m^h(s)|_{\L_h^2}^2 \ ds \ra^p  \right]
		&\leq N_{2,p}^{-1} 4^{p-1}\E \left[ |\partial^h m_0|^{2p}_{\L_h^2} + C_{1,\epsilon}^p T^{\frac{p}{q}+1} K_{1,p} \right] = K_{2,p},
	\end{align*}
	where $ K_{2,p} $ depends on $ p, C_v, C_\kappa, \epsilon, T $ and $ K_0 $, but not on $ h $,
	proving \eqref{estimate: E(int|mh x D2mh|)^p}.
	\end{proof}

	\begin{remark}\label{Remark: example of C_kappa}
		Fix $ p \in [1,\infty) $, if
		\begin{align*}
			C_\kappa \leq \frac{\alpha-\delta}{1+2\gamma^2 + 2^{1-\frac{2}{p}} b_p^{\frac{1}{p}} (1+|\gamma|)} \wedge \frac{1}{4 b_p^{\frac{1}{p}} (1+|\gamma|)} \wedge 1-\delta,
		\end{align*}
		then the assumption \eqref{defn: N1p, N2p>0} of Lemma \ref{Lemma: D+, mh x D+ est} is satisfied.
	\end{remark}

	\begin{lemma}\label{Lemma: Deltah estimate}
		For any $ p \in [1, \infty) $, under the conditions of Lemma \ref{Lemma: D+, mh x D+ est}, there exists a constant $ K_{3,p} $ independent of $ h $ such that 
		\begin{equation}\label{estimate: E(int |D2mh|)^p}
			\E \left[ \la \int_0^T |\Delta^h m^h(s)|_{\L_h^2}^2 \ ds \ra^p \right] \leq K_{3,p}.
		\end{equation}		
	\end{lemma}
	\begin{proof}
		Since $ |m^h(s,x)|=1 $ for all $ (s,x) \in [0,T] \times \Z_h $, we have
		\begin{align*}
			|\Delta^h m^h(s,x)|^2 = |m^h(s,x) \times \Delta^h m^h(s,x)|^2 + \langle m^h(s,x), \Delta^h m^h(s,x) \rangle^2.
		\end{align*}
		Then,
		\begin{align*}
			|\Delta^h m^h(s)|_{\L_h^2}^2 
			&= h \sum_{x \in \Z_h} \la |m^h(s,x) \times \Delta^h m^h(s,x)|^2 + \langle m^h(s,x), \Delta^h m^h(s,x) \rangle^2 \ra \\
			&= |m^h(s) \times \Delta^h m^h(s)|^2_{\L_h^2} + h \sum_{x \in \Z_h} \la \frac{1}{2} (|\partial^h m^h(s,x)|^2 + |(\partial^h m^h)^-(s,x)|^2) \ra^2 \\
			&\leq |m^h(s) \times \Delta^h m^h(s)|^2_{\L_h^2} + |\partial^{h} m^h(s)|^4_{\L_h^4},
		\end{align*}
		where 
		\begin{align*}
			|\partial^{h} m^h(s)|^4_{\L_h^4} 
			&\leq |\partial^{h} m^h(s)|^2_{\L_h^\infty} \ |\partial^{h} m^h(s)|_{\L_h^2}^2.
		\end{align*}
		Applying Lemma \ref{Lemma: discrete interpolation ineq} on $ \partial^{h} m^h(s) $, 
		\begin{equation}\label{|D+m|_Linfty est}
			|\partial^{h} m^h(s)|_{\L_h^\infty} \leq C |\partial^{h} m^h(s)|_{\L_h^2}^{\frac{1}{2}} |\partial^{h}(\partial^{h} m^h(s)) |_{\L_h^2}^{\frac{1}{2}},
		\end{equation}
		where $ |\partial^{h}(\partial^{h} m^h(s)) |_{\L_h^2} = |\Delta^h m^h(s)|_{\L_h^2} $. 
		Thus, 
		\begin{equation}\label{|D+m|_L4 est}
			\begin{aligned}
				|\partial^{h} m^h(s)|^4_{\L_h^4} 
				&\leq C^2 |\partial^{h} m^h(s)|_{\L_h^2}^3 \ |\Delta^h m^h(s)|_{\L_h^2} \\
				&\leq \frac{1}{2} C^4 |\partial^{h} m^h(s)|_{\L_h^2}^6 + \frac{1}{2}|\Delta^h m^h(s)|_{\L_h^2}^2.
			\end{aligned}
		\end{equation}
		We have
		\begin{align*}
			|\Delta^h m^h(s)|_{\L_h^2}^2 
			&\leq 2|m^h(s) \times \Delta^h m^h(s)|^2_{\L_h^2} + C^4 |\partial^{h} m^h(s)|_{\L_h^2}^6.
		\end{align*}
		Then, by Lemma \ref{Lemma: D+, mh x D+ est},
		\begin{align*}
			\E \left[ \la \int_0^T |\Delta^h m^h(t) |^2_{\L_h^2} \ dt\ra^p \right]
			&\leq 2^{p(p-1)}K_{2,p} + 2^{p-1} C^{4p} K_{1, 3p}T^p
			= K_{3,p},
		\end{align*}
		proving \eqref{estimate: E(int |D2mh|)^p}.
	\end{proof}

\section{Quadratic interpolation for the solution $m^h$ of \eqref{SDE: discrete}}\label{Section: quadratic interpolation}\label{sec: int}
	
\subsection{Interpolations}
	For any fixed $h>0$, let $ x_k = kh \in \Z_h $, for $ k \in \Z $. 
	We introduce interpolations of discrete functions defined on $\Z_h$ to functions defined on $ \R $. 
	
	Given $ u: \Z_h \to \R^3 $, let $ \ol{u}: \R \to \R^3 $ denote a quadratic interpolation of $ u $, given by 
	\begin{equation}\label{interpolation 2}
		\ol{u}(x) = \frac{1}{2} \la u(x_k) + u(x_{k-1}) \ra + \partial^{h} u^-(x_k) (x-x_k) + \frac{1}{2} \Delta^h u(x_k) (x-x_k)^2,
	\end{equation}
	for $ x \in [x_k,x_{k+1}) $, $ k \in \Z $, where $ \ol{u} $ is continuously differentiable with 
	\begin{equation}\label{Dx interpolation 2}
		\begin{aligned}
			D \ol{u}(x) &= \partial^{h} u^-(x_k) + \Delta^h u(x_k) (x-x_k), &&\quad x \in [x_k, x_{k+1}), \\
			D^2 \ol{u}(x) &= \Delta^h u(x_k), &&\quad x \in (x_k,x_{k+1}). 
		\end{aligned}
	\end{equation}
	Let $ \wh{u}: \R \to \R^3 $ denote the piecewise constant interpolation of $u$, given by
	\begin{equation}\label{interpolation 1} 
		\wh{u}(x) = u(x_k), \quad x \in [x_k, x_{k+1}), 
	\end{equation}
	for any $ k \in \Z $. 
	In terms of $ \wh{u} $, we can express $ \ol{u} $ as 
	\begin{equation}\label{uhat and ubar relation}
		\begin{aligned}
			\ol{u}(x) 
			&= \frac{1}{2} \la \wh{u}(x) + \wh{u}^-(x) \ra + \partial^{h} \wh{u}^-(x) (x-x_k) + \frac{1}{2} \Delta^h \wh{u}(x) (x-x_k)^2 \\
			&= \wh{u}(x) + \la D\ol{u}(x) - D^2\ol{u}(x) (x-x_k) \ra \la x-x_k - \frac{h}{2} \ra + \frac{1}{2} D^2\ol{u}(x) (x-x_k)^2,
		\end{aligned}
	\end{equation}
	for $ x \in [x_k,x_{k+1}) $, $ k \in \Z $. 
	
	We collect estimates of $\ol{u}$ and $\wh{u}$ in terms of $u$ in the following remark.
	\begin{remark}\label{Remark: ubar est}
	Let $ u: \Z_h \to \R^3 $. Then
	\begin{align*}
		&|\wh{u}|_{\L^\infty} = |u|_{\L_h^\infty}, \\
		&|\wh{u}|_{\L^2_w} \leq |\wh{u}|_{\L^2} = |u|_{\L_h^2}, \quad w > 0,
	\end{align*}
	and 
	\begin{equation}\label{interpolation 2 u ests}
		\begin{aligned}
			&|\ol{u}|_{\L^\infty} \leq 5|u|_{\L_h^\infty}, \\
			&|D\ol{u}|_{\L^2} 
			\leq 3|\partial^{h} u|_{\L_h^2}, \\
			&|D^2\ol{u}|_{\L^2} = |\Delta^h u|_{\L_h^2}, \\
			&|D\ol{u}|_{\L^4}^4 \leq \frac{1}{2} C^4 |D\ol{u}|_{\L^2}^6 + \frac{1}{2} |D^2 \ol{u}|_{\L^2}^2. 
		\end{aligned}
	\end{equation}
	\end{remark}
%
	
\subsection{An equation and estimates for $ \ol{m}^h $}
\subsubsection{Equation for $ \ol{m}^h $}
	Since $ m^h $ is the solution of the semi-discrete scheme \eqref{SDE: discrete}, the piecewise constant interpolation $ \wh{m}^h $ satisfies 
	\begin{equation}\label{SDE: interpolated 1 int form}
		\begin{aligned}
			\wh{m}^h(t) 
			&= m_0 + \int_0^t \la F^h_{\wh{v}}(\wh{m}^h(s)) + \frac{1}{2} S^h_{\wh{\kappa}}(\wh{m}^h(s)) \ra ds 
			+ \int_0^t G^h(\wh{m}^h(s)) \ d\wh{W}^h(s),
		\end{aligned}
	\end{equation}
	where $ F^h_{\wh{v}} $ and $ S^h_{\wh{\kappa}} $ are defined as in \eqref{discrete coef} but with $ \wh{v} $, $ \wh{\kappa} $ and $ \wh{\kappa \kappa'} $ in place of $ v $, $ \kappa $ and $ \kappa \kappa' $ respectively, and
	\begin{align*}
		\wh{W}^h(t) := \sum_{j=1}^\infty q_j W_j(t) \wh{f}_j, \quad t \in [0,T]. 
	\end{align*}
	In particular, for every fixed $ h>0 $, $ m^h \in \mathcal{C}([0,T]; \E_h) $ and $ \wh{m}^h \in \mathcal{C}([0,T]; \L^2_w \cap \mathring{\H}^1) $ for $ w \geq 1 $. 
	
	In order to obtain an equation for $ \ol{m}^h $, by using \eqref{uhat and ubar relation} we note that
	\begin{flalign*}
		&\wh{m}^h 
		= \ol{m}^h - R_0\ol{m}^h, \quad
		\partial^{h} \wh{m}^h
		= D\ol{m}^h - R_1\ol{m}^h, \\
		&\wh{m}^h \times \partial^{h} \wh{m}^h
		= \ol{m}^h \times D\ol{m}^h - P_1\ol{m}^h, \quad
		\wh{m}^h \times \la \wh{m}^h \times \partial^{h} \wh{m}^h \ra 
		= \ol{m}^h \times \la \ol{m}^h \times D\ol{m}^h \ra - P_2\ol{m}^h, \\
		&|\wh{m}^h \times \partial^{h} \wh{m}^h|^2 \wh{m}^h 
		= \left|\ol{m}^h \times D\ol{m}^h \right|^2 \ol{m}^h 
		- P_3\ol{m}^h, \\
		&\partial^h \wh{m}^h \times \la \wh{m}^h \times \partial^h \wh{m}^h \ra 
		= D\ol{m}^h \times  \la \ol{m}^h \times D\ol{m}^h \ra - P_4\ol{m}^h, \\
		&\langle \wh{m}^h, \partial^h \wh{m}^{h-} \rangle \ \wh{m}^h \times \partial^h \wh{m}^h 
		= \langle \ol{m}^h, D\ol{m}^h \rangle \ \ol{m}^h \times D\ol{m}^h - P_5\ol{m}^h, \\
		&\wh{m}^h \times \Delta^h \wh{m}^h
		= \ol{m}^h \times D^2\ol{m}^h - Q_1\ol{m}^h, \quad
		\wh{m}^h \times \la \wh{m}^h \times \Delta^h \wh{m}^h \ra 
		= \ol{m}^h \times \la \ol{m}^h \times D^2\ol{m}^h \ra - Q_2\ol{m}^h,
	\end{flalign*}
	where for $ u: \R \to \R^3 $ with well-defined weak derivatives, 
	\begin{equation}\label{defn: R0 to P1}
		\begin{aligned}
			R_0u(x) &:= \la Du(x) - D^2u(x) (x-x_k) \ra \la x-x_k - \frac{h}{2} \ra + \frac{1}{2} D^2u(x) (x-x_k)^2, \\
			R_1u(x) &:= D^2u(x) (x-x_k-h), \\
			P_1u(x) &:= R_0u(x) \times Du(x) + (u-R_0u)(x) \times R_1 u(x), 
		\end{aligned}
	\end{equation}
	and
	\begin{equation}\label{defn: P2 to Q2}
		\begin{aligned}	
			P_2u(x) &:= \left[ u \times P_1u + R_0u \times \la u \times Du - P_1u \ra \right](x), \\
			P_3u(x) &:= \left[ \lb 2 u \times Du - P_1u, P_1u \rb u + \left| u \times Du - P_1u \right|^2 R_0u \right](x), \\
			P_4u(x) &:= \left[ R_1 u \times (u \times Du) + (Du - R_1u) \times P_1u \right](x), \\
			P_5u(x) &:= \la \langle R_0 u(x), Du(x) \rangle + \langle (u -R_0u)(x), D^2u(x) (x-x_k)  \rangle \ra u(x) \times Du(x) \\ 
							&\quad + \langle (u-R_0u)(x), Du(x) - D^2u(x)(x-x_k) \rangle P_1 u(x), \\
			Q_1u(x) &:= R_0u(x) \times D^2u(x), \\
			Q_2u(x) &:= \left[ u \times Q_1u + R_0u \times \la (u - R_0u) \times D^2u \ra \right](x),
		\end{aligned}
	\end{equation}
	for $ x \in [x_k,x_{k+1}) $, $ k \in \Z $. 
	
	Moreover, for $ u: [0,T] \times \R \to \R^3 $, define
	\begin{equation}\label{defn: F*,G*}
		\begin{aligned}
			F_{\wh{v}}(u) &:= - u \times \la D^2u + \alpha u \times D^2u \ra + \wh{v} \la u \times \la u \times Du \ra + \gamma u \times Du \ra, \\
			S_{\wh{\kappa}}(u) 
			&:= \frac{1}{2}\la (\wh{\kappa}^-)^2 + \wh{\kappa}^2 \ra \la (\gamma^2 - 1) u \times \la u \times D^2 u\ra - 2 \gamma u \times D^2 u \ra \\
			&\quad - \wh{\kappa}^2 \la \gamma^2 Du \times (u \times Du) + |u \times Du|^2 u \ra + 2 \gamma (\wh{\kappa}^2)^- \langle u, Du \rangle u \times Du \\
			&\quad + \wh{\kappa \kappa'} \left[ (\gamma^2 -1) u \times (u \times Du) - 2 \gamma u \times Du \right].
		\end{aligned}
	\end{equation}
	
	By \eqref{SDE: interpolated 1 int form}, we arrive at the equation of $ \ol{m}^h $. For $ t \in [0,T] $, 
	\begin{equation}\label{SDE: m-bar with remainders}
		\begin{aligned}
			\ol{m}^h(t) 
			&= m_0 + \int_0^t F_{\wh{v}}(\ol{m}^h(s)) \ ds 
			+ \frac{1}{2} \int_0^t S_{\wh{\kappa}}(\ol{m}^h(s)) \ ds 
			+ \int_0^t G(\ol{m}^h(s)) \ d\wh{W}^h(s) \\
			&\quad + R_0 \ol{m}^h(t) + \int_0^t \la -\wh{v}(\gamma P_1 + P_2) \ol{m}^h + Q_1\ol{m}^h +\alpha Q_2\ol{m}^h \ra (s) \ ds \\
			&\quad + \frac{1}{4} \int_0^t \la (\wh{\kappa}^2)^- + \wh{\kappa}^2 \ra \left[ 2 \gamma Q_1\ol{m}^h - (\gamma^2-1) Q_2\ol{m}^h \right] (s) \ ds \\
			&\quad + \frac{1}{2} \int_0^t \la \wh{\kappa}^2 (P_3\ol{m}^h + \gamma^2 P_4\ol{m}^h) - 2 \gamma (\wh{\kappa}^2)^- P_5\ol{m}^h \ra(s) \ ds \\
			&\quad + \frac{1}{2}\int_0^t \wh{\kappa \kappa'} \left[ 2\gamma P_1 \ol{m}^h - (\gamma^2-1) P_2\ol{m}^h \right](s) \ ds \\
			&\quad - \int_0^t (\gamma P_1 + P_2) \ol{m}^h(s) \ d\wh{W}^h(s).
		\end{aligned}
	\end{equation}

\subsubsection{Estimates for $ \ol{m}^h $.}
	For $ p \in [1, \infty) $ and $ w \geq 1 $, we deduce from Lemmata \ref{Lemma: D+, mh x D+ est}, \ref{Lemma: Deltah estimate} and Remark \ref{Remark: ubar est}:
	\begin{equation}\label{interpolation 2 ests L-infty}
		\sup_{t \in [0,T]} |\ol{m}^h(t)|_{\L^\infty}^p \leq 5^p, \quad \P\text{-a.s.}
	\end{equation}
	and
	\begin{equation}\label{interpolation 2 ests L}
		\E \left[ \sup_{t \in [0,T]} |D\ol{m}^h(t)|_{\L^2}^{2p} + \la \int_0^T \la |\ol{m}^h(t)|_{\L^p_w}^p + |D\ol{m}^h(t)|_{\L^4}^4 + |D^2 \ol{m}^h(t)|_{\L^2}^2 \ra dt \ra^p \right] \leq C(p,T,w),
	\end{equation}	
	for some constant $ C(p,T,w) $. 
	
	Results of convergence of $R_0 \ol{m}^h$, $R_1 \ol{m}^h$, $P_1 \ol{m}^h,\ldots, P_5 \ol{m}^h$ and $Q_1 \ol{m}^h$, $Q_2 \ol{m}^h$ are proved in the following lemma.
	\begin{lemma}\label{Lemma: interpolation 2 remainders}
		For $ f = R_1 $, $ P_1 $ or $ P_2 $,
		\begin{align*}
			\lim_{h \to 0} \E \left[\sup_{t \in [0,T]} |R_0 \ol{m}^h(t)|_{\L^2}^2 + \int_0^T |f \ol{m}^h(t)|_{\L^2}^2 \ dt \right] = 0. 
		\end{align*}
		Moreover, for $ f=P_3 $, $ P_4 $, $ P_5 $, $ Q_1 $ or $ Q_2 $, for any measurable process $ \varphi \in L^4(\Omega; L^4(0,T; \L^4)) $, 
		\begin{align*}
			\lim_{h \to 0} \E \left[ \int_0^T \lb f\ol{m}^h(t), \varphi(t) \rb_{\L^2} \right] = 0. 
		\end{align*}
	\end{lemma}
	\begin{proof}
	By construction, $ |x-x_k| < h $ for $ x \in [x_k, x_{k+1}) $, $ k \in \Z $, and $ \sup_{t \in [0,T]} |\ol{m}^h(t)|_{\L^\infty} \leq 5 $, $ \P $-a.s.
	Thus, for $ p \in [1, \infty) $, there is a constant $ C_{R_0} $ independent of $ h $ such that
	\begin{align*}
		\sup_{t \in [0,T]} |R_0\ol{m}^h(t)|_{\L^\infty}^p  \leq C_{R_0}^p, \quad \P\text{-a.s.}
	\end{align*}
	Using \eqref{Dx interpolation 2}, we can often re-write $ R_0, \ldots, Q_2 $ in terms of $ \wh{m}^h $ to simplify the estimates. 
	
	\underline{An estimate on $ R_0\ol{m}^h $:} 	
	\begin{align*}
		R_0 \ol{m}^h(t,x) = \partial^h \wh{m}^h(t,x-h) \la x-x_k- \frac{h}{2}\ra + \frac{1}{2} \Delta^h \wh{m}(t,x) (x-x_k)^2, \quad x \in [x_k, x_{k+1}],
	\end{align*}
	which implies
	\begin{equation}\label{pf: R0mh strong conv L2}
		\begin{aligned}
			\E \left[ \sup_{t \in [0,T]} |R_0 \ol{m}^h(t)|_{\L^2}^2 \right] 
			&\leq \frac{h^2}{2}\E \left[ \sup_{t \in [0,T]} \la |\partial^h \wh{m}^h(t)|_{\L^2}^2 +  |\partial^h \wh{m}^h(t) - \partial^h \wh{m}^{h-}(t)|_{\L^2}^2 \ra \right] .
		\end{aligned}
	\end{equation}
	The expectation on the right-hand side of \eqref{pf: R0mh strong conv L2} is bounded by Lemma \ref{Lemma: D+, mh x D+ est}, thus the left-hand side converges to $ 0 $ as $ h \to 0 $. 
	As a result,
	\begin{equation}\label{pf: R0mh strong conv L4}
		\E \left[ \int_0^T |R_0\ol{m}^h(t)|_{\L^4}^4 \ dt \right] 
		\leq \E \left[ C_{R_0}^2 T \sup_{t \in [0,T]} |R_0 \ol{m}^h(t)|_{\L^2}^2 \right] 
		\stackrel{h \to 0}{\to} 0.
	\end{equation}

	\underline{An estimate on $ R_1\ol{m}^h $:} 
	\begin{align*}
		\E \left[ \int_0^T |R_1 \ol{m}^h(t)|_{\L^2}^2 \ dt \right] 
		&\leq \frac{h^2}{2}\E \left[ \int_0^T | D^2 \ol{m}^h(t) |_{\L^2}^2 \ dx \ dt \right],
	\end{align*}
	implying $ R_1 \ol{m}^h \to 0 $ in $ L^2(\Omega; L^2(0,T; \L^2)) $ by \eqref{interpolation 2 ests L}.

	\underline{An estimate on $ P_1\ol{m}^h $:} 
	\begin{align*}
		\E \left[ \int_0^T |P_1 \ol{m}^h(t)|_{\L^2}^2 \ dt \right] 
		&= \E \left[ \int_0^T |R_0 \ol{m}^h(t) \times D \ol{m}^h(t) + \wh{m}^h(t) \times R_1\ol{m}^h(t)|_{\L^2}^2 \ dt \right] \\
		&\leq 2 \la \E \left[ \int_0^T |R_0 \ol{m}^h(t)|_{\L^4}^4 \ dt \right] \ra^\frac{1}{2} \la \E \left[ \int_0^T |D \ol{m}^h(t)|_{\L^4}^4 \ dt \right] \ra^\frac{1}{2} \\
		&\quad + 2 \E \left[ \int_0^T |R_1\ol{m}^h(t)|_{\L^2}^2 \ dt \right],
	\end{align*}
	where $ D\ol{m}^h \in L^4(\Omega; L^4(0,T; \L^4)) $ for all $ h > 0 $ by \eqref{interpolation 2 ests L}. 
	Then, by the $ L^4 $-convergence of $ R_0\ol{m}^h $ in \eqref{pf: R0mh strong conv L4} and the $ L^2 $-convergence of $ R_1\ol{m}^h $, we have $ P_1 \ol{m}^h \to 0 $ in $ L^2(\Omega; L^2(0,T; \L^2)) $.

	\underline{An estimate on $ P_2\ol{m}^h $:}
	\begin{align*}
		\E \left[ \int_0^T |P_2 \ol{m}^h(t)|_{\L^2}^2 \ dt \right] 
		&= \E \left[ \int_0^T |\wh{m}^h(t)\times P_1 \ol{m}^h(t) + R_0 \ol{m}^h(t) \times (\ol{m}^h(t) \times D \ol{m}^h(t))|_{\L^2}^2 \ dt \right] \\
		&\leq 2 \E \left[ \int_0^T |\wh{m}^h(t) \times P_1\ol{m}^h(t)|_{\L^2}^2 \ dt \right] \\
		&\quad + 2 \la \E \left[ \int_0^T |R_0 \ol{m}^h(t)|_{\L^4}^4 \ dt \right] \ra^\frac{1}{2} \la \E \left[ \int_0^T |\ol{m}^h(t) \times D \ol{m}^h(t)|_{\L^4}^4 \ dt \right] \ra^\frac{1}{2},
	\end{align*}
	which implies that $ P_2 \ol{m}^h \to 0 $ in $ L^2(\Omega; L^2(0,T; \L^2)) $ by \eqref{interpolation 2 ests L-infty} and \eqref{interpolation 2 ests L} together with the convergences of $ P_1\ol{m}^h $ and $ R_0\ol{m}^h $.

	\underline{An estimate on $ P_3 \ol{m}^h $:} 
	\begin{align*}
		P_3 \ol{m}^h 
		&= |\wh{m}^h \times \partial^h \wh{m}^h|^2 R_0\ol{m}^h 
		+ \lb \ol{m}^h \times D\ol{m}^h + \wh{m}^h \times \partial^h \wh{m}^h, P_1\ol{m}^h \rb \ol{m}^h \\
		&=: P_{31} \ol{m}^h + P_{32} \ol{m}^h.
	\end{align*}
	Then for $ \varphi \in L^4(\Omega; L^4(0,T; \L^4)) $,
	\begin{align*}
		\E \left[ \int_0^T \lb P_{31} \ol{m}^h(t), \varphi(t) \rb_{\L^2} \ dt \right] 
		&\leq 
		\la \E\left[ \int_0^T |\wh{m}^h(t) \times \partial^h \wh{m}^h(t)|^4_{\L^4} \ dt \right] \ra^\frac{1}{2} \\
		&\quad \times\la \E\left[ \int_0^T |\varphi(t)|_{\L^4}^4 \ dt\right] \ra^\frac{1}{4} 
		\la \E\left[ \int_0^T |R_0\ol{m}^h(t)|^4_{\L^4} \ dt \right] \ra^\frac{1}{4},
	\end{align*}
	and
	\begin{align*}
		\E\left[ \int_0^T \lb P_{32}\ol{m}^h(t), \varphi(t) \rb_{\L^2} dt \right] 
		&\leq  
		5\la \E\left[ \int_0^T\left| \ol{m}^h(t) \times D\ol{m}^h(t) + \wh{m}^h(t) \times \partial^{h} \wh{m}^h(t) \right|^4_{\L^4} \ dt \right] \ra^\frac{1}{4} \\
		&\quad \times \la \E\left[ \int_0^T |\varphi(t)|_{\L^4}^4 \ dt\right] \ra^\frac{1}{4} 
		\la \E\left[ \int_0^T |P_1\ol{m}^h(t)|^2_{\L^2} \ dt \right] \ra^\frac{1}{2}.
	\end{align*}
	By Lemmata \ref{Lemma: D+, mh x D+ est} and \ref{Lemma: Deltah estimate}, \eqref{interpolation 2 ests L-infty}, \eqref{interpolation 2 ests L} and the property of $ \varphi $, the expectations on the right-hand side of the two inequalities above are finite. 
	Then by the convergences of $ R_0\ol{m}^h $ and $ P_1\ol{m}^h $, we obtain the weak convergence of $ P_3\ol{m}^h $ as desired.

	\underline{An estimate on $ P_4\ol{m}^h $:}
	\begin{align*}
		P_4 \ol{m}^h 
		&= R_1\ol{m}^h \times (\ol{m}^h \times D\ol{m}^h) + \partial^h\wh{m}^h \times P_1\ol{m}^h \\
		&=: P_{41}\ol{m}^h + P_{42}\ol{m}^h.   
	\end{align*}
	We have
	\begin{align*}
		\E \left[ \int_0^T \lb P_{41} \ol{m}^h(t), \varphi(t) \rb_{\L^2} \ dt \right] 
		&\leq 
		5 \la \E\left[ \int_0^T |R_1\ol{m}^h(t)|^2_{\L^2} \ dt \right] \ra^\frac{1}{2} 
		\la \E\left[ \int_0^T |D\ol{m}^h(t)|^4_{\L^4} \ dt \right] \ra^\frac{1}{4} \\
		&\quad \times \la \E\left[ \int_0^T |\varphi(t)|_{\L^4}^4 \ dt\right] \ra^\frac{1}{4},
	\end{align*}
	and
	\begin{align*}
		\E \left[ \int_0^T \lb P_{42} \ol{m}^h(t), \varphi(t) \rb_{\L^2} \ dt \right] 
		&\leq 
		\la \E\left[ \int_0^T |P_1\ol{m}^h(t)|^2_{\L^2} \ dt \right] \ra^\frac{1}{2} 
		\la \E\left[ \int_0^T |\partial^h\wh{m}^h(t)|^4_{\L^4} \ dt \right] \ra^\frac{1}{4} \\
		&\times \la \E\left[ \int_0^T |\varphi(t)|_{\L^4}^4 \ dt\right] \ra^\frac{1}{4}.
	\end{align*}
	Using \eqref{interpolation 2 ests L}, \eqref{|D+m|_L4 est} and the convergences of $ R_1 \ol{m}^h $ and $ P_1 \ol{m}^h $, the right-hand side of each of the two inequalities above converges to $ 0 $ as $ h \to 0 $.

	\underline{An estimate on $ P_5\ol{m}^h $:}
	\begin{align*}
		P_5\ol{m}^h(x) 
		&= \langle R_0 \ol{m}^h(x), D\ol{m}^h(x) \rangle \ol{m}^h(x) \times D\ol{m}^h(x) 
		+  (x-x_k) \langle \wh{m}^h(x), D^2\ol{m}^h(x) \rangle \ol{m}^h(x) \times D\ol{m}^h(x)  \\ 
		&\quad + \langle \wh{m}^h(x), \partial^h \wh{m}^{h-} (x) \rangle P_1 u(x) \\
		&=: P_{51}\ol{m}^h(x) + P_{52}\ol{m}^h(x) + P_{53}\ol{m}^h(x), \quad x \in [x_k, x_{k+1}).
	\end{align*}
	By Lemma \ref{Lemma: solution of disSDE} and \eqref{interpolation 2 ests L-infty}, we have
	\begin{align*}
		\E \left[ \int_0^T \lb P_{51}\ol{m}^h(t), \varphi(t) \rb_{\L^2} \ dt \right] 
		&\leq 
		5 \la \E\left[ \int_0^T |R_0\ol{m}^h(t)|^4_{\L^4} \ dt \right] \ra^\frac{1}{4} 
		\la \E\left[ \int_0^T |D\ol{m}^h(t)|^4_{\L^4} \ dt \right] \ra^\frac{1}{2} \\
		&\quad \times \la \E\left[ \int_0^T |\varphi(t)|_{\L^4}^4 \ dt\right] \ra^\frac{1}{4}, \\
		\E \left[ \int_0^T \lb P_{52}\ol{m}^h(t), \varphi(t) \rb_{\L^2} \ dt \right] 
		&\leq 
		5h \la \E\left[ \int_0^T |D^2\ol{m}^h(t)|^2_{\L^2} \ dt \right] \ra^\frac{1}{2} 
		\la \E\left[ \int_0^T |D\ol{m}^h(t)|^4_{\L^4} \ dt \right] \ra^\frac{1}{4} \\
		&\quad \times \la \E\left[ \int_0^T |\varphi(t)|_{\L^4}^4 \ dt\right] \ra^\frac{1}{4},
	\end{align*}
	and 
	\begin{align*}
		\E \left[ \int_0^T \lb P_{53}\ol{m}^h(t), \varphi(t) \rb_{\L^2} \ dt \right] 
		&\leq 
		\la \E\left[ \int_0^T |P_1\ol{m}^h(t)|^2_{\L^2} \ dt \right] \ra^\frac{1}{2} 
		\la \E\left[ \int_0^T |\partial^h \wh{m}^h(t)|^4_{\L^4} \ dt \right] \ra^\frac{1}{4} \\
		&\quad \times \la \E\left[ \int_0^T |\varphi(t)|_{\L^4}^4 \ dt\right] \ra^\frac{1}{4}.
	\end{align*}
	Similarly, by Lemma \ref{Lemma: D+, mh x D+ est}, \eqref{interpolation 2 ests L} and the convergences of $ R_0\ol{m}^h $ and $ P_1\ol{m}^h $, the right-hand side of each of the inequalities above converges to $ 0 $ as $ h \to 0 $.

	\underline{An estimate on $ Q_1\ol{m}^h $:}
	\begin{align*}
		Q_1\ol{m}^h(t,x)
		&= R_0 \ol{m}^h(t,x) \times D^2 \ol{m}^h(t,x) \\
		&= \partial^h \wh{m}^{h-}(t,x) \times \Delta^h \wh{m}^h(t,x) \la x-x_k - \frac{h}{2} \ra, \quad x \in [x_k,x_{k+1}).
	\end{align*}
	Thus, 
	\begin{align*}
		&\E\left[ \int_0^T \langle Q_1\ol{m}^h(t), \varphi(t) \rangle_{\L^2} \ dt \right] \\
		&\leq \frac{h}{2} \E\left[ \int_0^T \int_\R |\partial^h \wh{m}^{h-}(t,x)| |\Delta^h \wh{m}^h(t,x)| |\varphi(t,x)| \ dx \ dt \right] \\
		&\leq \frac{h}{2} \la \E\left[ \int_0^T |\partial^{h} \wh{m}^h(t)|_{\L^4}^4 \ dt \right] \ra^\frac{1}{4}
		\la \E \left[ \int_0^T |\Delta^h \wh{m}^h(t)|_{\L^2}^2 \ dt \right] \ra^\frac{1}{2} 
		\la \E \left[ \int_0^T |\varphi(t)|_{\L^4}^4 \ dt \right] \ra^\frac{1}{4},
	\end{align*}
	where the three expectation terms on the right-hand side are finite, proving that the right-hand side converges to $ 0 $ as $ h \to 0 $.

	\underline{An estimate on $ Q_2\ol{m}^h $:}
	\begin{align*}
		&\E \left[ \int_0^T \langle Q_2\ol{m}^h(t), \varphi(t) \rangle_{\L^2} \ dt \right] \\
		&= \E \left[ \int_0^T \langle Q_1\ol{m}^h(t), \varphi(t) \times \ol{m}^h(t) \rangle_{\L^2} \ dt + \int_0^T \langle R_0\ol{m}^h(t), (\wh{m}^h(t) \times \Delta^h \wh{m}^h(t)) \times \varphi(t) \rangle_{\L^2} \ dt \right] \\
		&\leq \E \left[ \int_0^T \langle Q_1\ol{m}^h(t), \varphi(t) \times \ol{m}^h(t) \rangle_{\L^2} \ dt \right] \\
		&\quad + \la \E \left[ \int_0^T |R_0 \ol{m}^h(t)|_{\L^4}^4 \ dt \right] \ra^{\frac{1}{4}} \la \int_0^T |\wh{m}^h(t) \times \Delta^h \wh{m}^h(t) |_{\L^2}^2 \ dt \ra^\frac{1}{2} \la \int_0^T |\varphi(t)|_{\L^4}^4 \ dt \ra^\frac{1}{4},
	\end{align*}
	where on the right-hand side, the first term converges to $ 0 $ as $ h \to 0 $ by the argument for $ Q_1 \ol{m}^h $ with $ \varphi \times \ol{m}^h \in L^4(\Omega; L^4(0,T; \L^4)) $, and the second term converges to $ 0 $ by \eqref{pf: R0mh strong conv L4}.
	\end{proof}
	
	We also obtain uniform bounds for $\ol{m}^h$ in weighted spaces.
	\begin{lemma}\label{Lemma: mbar in B, W and |mbar| to 1}
		For any $ w \geq 1 $, the quadratic interpolation $ \ol{m}^h $ satisfies
		\begin{enumerate}[label=(\roman*)]
			\item 
			$ \sup_h \E [ |\ol{m}^h|_{B_w}^2 ] < \infty $ for $ B_w := L^2(0,T; \L^2_w \cap \mathring{\H}^2) \cap \mathcal{C}([0,T]; \L^2_w \cap \mathring{\H}^1) $,
			
			\item
			$ \sup_h \E[ |\ol{m}^h|^2_{W^{\alpha,p}(0,T; \L^2_w)} ] < \infty $, for $ p \in [2,\infty) $ and $ \alpha \in (0, \frac{1}{2}) $ such that $ \alpha - \frac{1}{p} < \frac{1}{2} $, 
			
			\item
			$ |\ol{m}^h| \to 1 $ in $ L^2(\Omega; L^2(0,T; \L^2)) $.
		\end{enumerate}
	\end{lemma}
	\begin{proof}
		
	\underline{Part (i).} 
	For every fixed $ h > 0 $, $ \wh{m}^h $ is in $ \mathcal{C}([0,T]; \L^2_w \cap \mathring{\H}^1) $ and so does $ \ol{m}^h $. Then part (i) follows directly from the estimates in \eqref{interpolation 2 ests L-infty} and \eqref{interpolation 2 ests L}. 

	\underline{Part (ii).} Recall \eqref{SDE: interpolated 1 int form}, we have from the definition \eqref{interpolation 2} that 
	\begin{align*}
		\ol{m}^h(t,x)
		&= I_0(x) + I_1(t,x) + I_2(t,x) + I_3(t,x) + I_4(t,x), 
	\end{align*}
	where for $ x \in [x_k,x_{k+1}) $, $ k \in \Z $,
	\begin{align*}
		I_0(x) 
		&= \frac{1}{2} \la m_0(x_k) + m_0(x_{k-1}) \ra + (x-x_k) \partial^{h} m_0(x_{k-1}) + \frac{1}{2}(x-x_k)^2\Delta^h m_0(x_k), \\
		I_1(t,x) 
		&= \frac{1}{2} \int_0^t \la F^h(\wh{m}^h(s,x_k)) + \frac{1}{2} S^h(\wh{m}^h(s,x_k))  \ra ds \\ 
		&\quad + \frac{1}{2} \int_0^t \la F^h(\wh{m}^h(s,x_{k-1})) + \frac{1}{2} S^h(\wh{m}^h(s,x_{k-1})) \ra ds, \\
		I_2(t,x)
		&= (x-x_k) \int_0^t \partial^{h} \la F^h(\wh{m}^h) + \frac{1}{2} S^h(\wh{m}^h) \ra (s,x_{k-1}) \ ds, \\
		I_3(t,x) 
		&= \frac{1}{2}(x-x_k)^2 \int_0^t \Delta^h \la F^h(\wh{m}^h) + \frac{1}{2} S^h(\wh{m}^h) \ra (s,x_k) \ ds, \\
		I_4(t,x)
		&= \frac{1}{2} \int_0^t \la G^h(\wh{m}^h(s,x_k)) + G^h(\wh{m}^h(s,x_{k-1})) \ra d\wh{W}^h(s) \\ 
		&\quad + (x-x_k) \int_0^t \partial^{h} G^h(\wh{m}^h(s,x_{k-1})) \ d\wh{W}^h(s) \\
		&\quad + \frac{1}{2}(x-x_k)^2 \int_0^t \Delta^h G^h(\wh{m}^h(s,x_k)) \ d\wh{W}^h(s).
	\end{align*}
	By the $ \L^\infty $-estimate of $ \ol{m}^h $ in \eqref{interpolation 2 ests L-infty}, 
	\begin{align*}
		\E \left[ \int_0^T |\ol{m}^h(t)|_{\L^2_w}^2 \ dt \right] \leq 5^2 \pi T. 
	\end{align*}
	For $ I_1 $, by Lemma \ref{Lemma: D+, mh x D+ est}, there exists a constant $ a_1 $ that may depend on $ C_v, C_\kappa, \alpha, \gamma, T, K_{1,1}, K_{1,3} $ and $ K_{2,1} $ such that
	\begin{align*}
		\E \left[ \int_0^T \left| F^h(\wh{m}^h(t)) + \frac{1}{2} S^h(\wh{m}^h(t)) \right|_{\L^2}^2 \ dt \right] 
		\leq a_1, 
	\end{align*}
	For $ I_2 $ and $ I_3 $, since $ |x-x_k| \leq h $ and 
	\begin{align*}
		h| \partial^{h} u(x_{k-1}) | &\leq |u(x_k)| + |u(x_{k-1})|, \\
		h^2 |\Delta^h u(x_k) | &\leq |u(x_{k+1})| + |u(x_{k-1})| + 2|u(x_k)|,
	\end{align*}
	there also exist constants $ a_2, a_3 $ such that 
	\begin{align*}
		\E \left[ \int_0^T \left| h \partial^{h} \la F^h(\wh{m}^h) + \frac{1}{2} S^h(\wh{m}^h) \ra (t) \right|_{\L^2}^2 dt \right] 
		&\leq a_2, \\
		\E \left[ \int_0^T \left| h^2 \Delta^h \la F^h(\wh{m}^h) + \frac{1}{2} S^h(\wh{m}^h) \ra (t) \right|_{\L^2}^2 dt \right] 
		&\leq a_3.
	\end{align*}

	Similarly, for the stochastic integrals in $ I_4 $, we only need to verify that $ \int_0^T G^h(\wh{m}^h(s)) \ d\wh{W}^h(s) $ is bounded in $ L^p(\Omega; W^{\alpha, p}(0,T; \L^2)) $. 
	By \cite[Lemma 2.1]{FlandoliGatarek}, there exist a constant $ C $ depending on $ \alpha, p, \gamma, T $ and a constant $ a_4 $ depending on $ C, C_\kappa $ and $ K_{1,p} $ such that for $ p \in [2,\infty) $ and $ \alpha \in (0, \frac{1}{2}) $, 
	\begin{align*}
		\E \left[ \left| \int_0^T G^h(\wh{m}^h(s)) \ d\wh{W}^h(s) \right|^p_{W^{\alpha,p}(0,T; \L^2)} \right] 
		&= \E \left[ \left| \int_0^T G^h(m^h(s)) \ dW(s) \right|^p_{W^{\alpha,p}(0,T; \L_h^2)} \right] \\
		&\leq C \ \E \left[ \int_0^T \la \sum_j q_j^2 \left| f_j G^h(m^h(s)) \right|^2_{\L_h^2} \ra^\frac{p}{2} ds \right] \\
		&\leq C |\kappa|_{\L^\infty}^p  \E \left[\int_0^T \left|G^h(\wh{m}^h(s)) \right|^p_{\L^2} \ ds \right] 
		\leq a_4.
	\end{align*}	
	Since $ \L^2 \hookrightarrow \L^2_w $ for $ w \geq 1 $, the estimates above hold for the $ \L^2_w $-norm. 
	By Lemma \ref{Lemma: cont embedding between W}, the embedding $ W^{1,2}(0,T; \L^2_w) \hookrightarrow W^{\alpha,p}(0,T; \L^2_w) $ is continuous for $ \alpha - \frac{1}{p} < \frac{1}{2} $. 
	Thus, 
	\begin{equation*}
		\sup_h \E \left[ |\ol{m}^h|_{W^{\alpha,p}(0,T; \L^2_w)}^2 \right] < \infty.
	\end{equation*}

	\underline{Part (iii).} Since $ |m^h(t,x)|=1 $, $ \P $-a.s. for all $ (t,x) \in [0,T] \times \Z_h $, we observe that 
	\begin{align*}
		|\ol{m}^h(t,x)| \leq 1 + \left| \partial^h m^h(t,x_{k-1})(x-x_k) + \frac{1}{2} \Delta^h m^h(t,x_k)(x-x_k)^2 \right|, 
	\end{align*}
	for $ (t,x) \in [0,T] \times [x_k,x_{k+1}) $, $ k \in \Z $. This implies
	\begin{align*}
		&\E \left[ \int_0^T \int_\R \left| |\ol{m}^h(t,x)| -1 \right|^2 \ dx \ dt \right] \\
		&\leq \E \left[ \int_0^T \sum_k \int_{x_k}^{x_{k+1}} \left| \partial^{h} m^h(t,x_{k-1})(x-x_k) + \frac{1}{2} \Delta^h m^h(t,x_k)(x-x_k)^2 \right|^2  \ dx \ dt \right] \\
		&\leq \frac{2}{3} h^2 K_{1,1} T + \frac{1}{10} h^4 K_{3,1},
	\end{align*}
	where the last inequality holds by Lemmata \ref{Lemma: D+, mh x D+ est} and \ref{Lemma: Deltah estimate}, and we obtain the convergence after taking $ h \to 0 $. 	
	\end{proof}

\section{Existence of solution}
In this section, we first show that the sequence $\{(\ol{m}^h, W)\}_h$ is tight and then by using the Skorohod theorem we obtain its almost sure convergence, up to a change of probability space. Finally, we prove that the limit is a solution of the stochastic LLS equation \eqref{sLLG} in the sense of Definition \ref{def: sol}.

\subsection{Tightness and construction of new probability space and processes}\label{Section: tightness}		
	Fix $ w_1, w_2 $ such that $ w_2 > w_1 \geq 1 $. Define
	\begin{align*}
		E_0 &:= L^2(0,T; \L^2_{w_1} \cap \mathring{\H}^2) \cap W^{\alpha, 4}(0,T; \L^2_{w_2}), \\
		E &:= L^2(0,T; \H_{w_2}^1) \cap \mathcal{C}([0,T]; \H_{w_1}^{-1}).
	\end{align*}
	Recall $ \L^2_{w_1} \cap \mathring{\H}^2 \stackrel{\text{compact}}{\hookrightarrow} \H^1_{w_2} \hookrightarrow \L^2_{w_2} $. 
	By \eqref{embed: Walpha,4 in Wbeta,2} and Lemma \ref{Lemma: comp embedding in Lp}, 
	\begin{align*}
		E_0 \hookrightarrow L^2(0,T; \L^2_{w_1} \cap \mathring{\H}^2) \cap W^{\beta,2}(0,T; \L^2_{w_2}) \stackrel{\text{compact}}{\hookrightarrow} L^2(0,T; \H^1_{w_2}),
	\end{align*}
	where $ \ol{m}^h \in E_0 $, $ \P $-a.s. by Lemma \ref{Lemma: mbar in B, W and |mbar| to 1}.
	Also, since the embeddings $ \H^1_{w_1} \hookrightarrow \L^2_{w_2} \hookrightarrow \H_{w_1}^{-1} $ are compact and $ 4 \alpha >1 $, it holds by Lemma \ref{Lemma: comp embedding in C} that
	\begin{align*}
		W^{\alpha, 4}(0,T; \L^2_{w_2}) \stackrel{\text{compact}}{\hookrightarrow} \mathcal{C}([0,T]; \H^{-1}_{w_1}).
	\end{align*}
	In summary, $ E_0 $ is compactly embedded in $ E $. 
	For any $ r>0 $, 
	\begin{align*}
		\P \la |\ol{m}^h|_{E_0} > r \ra 
		\leq \frac{1}{r^2} \E \left[ |\ol{m}^h|^2_{E_0} \right],
	\end{align*}
	where $ \{ |\ol{m}^h|_{E_0} \leq r \} $ is compact in $ E $, and the right-hand-side converges to $ 0 $ as $ r $ tends to infinity. 	
	Therefore, the set of laws $ \{ \mathcal{L}(\ol{m}^h) \} $ on the Banach space $ E $ is tight, which implies the following convergence result. 
	\begin{lemma}\label{Lemma: Skorohod}
		There exists a probability space $ (\Omega^*, \mathcal{F}^*, \P^*) $ and there exists a sequence $ (m_h^*, W_h^*) $ of $ E \times \mathcal{C}([0,T]; H^2(\R)) $-valued random variables defined on $ (\Omega^*, \mathcal{F}^*, \P^*) $, 
		such that the laws of $ (\ol{m}^h, W) $ and $ (m_h^*, W_h^*) $ on $ E \times \mathcal{C}([0,T]; H^2(\R)) $ are equal for every $ h $, 
		and there exists an $ E \times \mathcal{C}([0,T]; H^2(\R)) $-valued random variable $ (m^*, W^*) $ defined on $ (\Omega^*, \mathcal{F}^*, \P^*) $ such that 
		\begin{equation}\label{ptw conv: mh* L2(H1) + C(Xw-1)}
			m_h^* \to m^* \text{ in } E, \quad \P^*\text{-a.s.} 
		\end{equation}
		and 
		\begin{equation}\label{ptw conv: Wh}
			W_h^* \to W^* \text{ in } \mathcal{C}([0,T]; H^2(\R)), \quad \P^*\text{-a.s.}
		\end{equation}
	\end{lemma}
	\begin{proof}
		Since $ E \times \mathcal{C}([0,T]; H^2(\R)) $ is a separable metric space, the result holds by the Skorohod theorem. 
	\end{proof}
	
	Since the laws of $ (\ol{m}^h, W) $ and $ (m_h^*, W_h^*) $ on $ E \times \mathcal{C}([0,T]; H^2(\R)) $ are equal, due to the following remark we obtain the same estimates for $m_h^*$.
	\begin{remark}\label{Remark: borel sets B vs E}
		By Kuratowski's theorem, the Borel sets of 
		\begin{align*}
			B := B_{w_1}= L^2(0,T; \L^2_{w_1} \cap \mathring{\H}^2) \cap \mathcal{C}([0,T]; \L^2_{w_1} \cap \mathring{\H}^1)
		\end{align*}
		are Borel sets of $ E = L^2(0,T; \H^1_{w_2}) \cap \mathcal{C}([0,T]; \H^{-1}_{w_1}) $ for $ w_1 < w_2 $, where 
		\begin{align*}
			\P \la \ol{m}^h \in B \ra = 1. 
		\end{align*} 
		We can assume that $ m_h^* $ takes values in $ B $ and the laws on $ B $ of $ \ol{m}^h $ and $ m_h^* $ are equal. 
	\end{remark}
	
	By Remark \ref{Remark: borel sets B vs E}, the sequence $ (m_h^*)_h $ satisfies the same estimates as $ (\ol{m}^h)_h$ on $ B $. 
	By \eqref{interpolation 2 ests L}, for any $ p \in [1,\infty) $, 
	\begin{align}
		\sup_{h} \E^* \left[\sup_{t \in [0,T]} |m_h^*(t)|_{\L^2_{w_1}}^{2p}  \right] &< \infty, \label{mh* est C([0,T],Lw1)} \\
		\sup_{h} \E^* \left[ \sup_{t \in [0,T]} |Dm_h^*(t)|_{\L^2}^{2p} \right] &< \infty, \label{mh* est Dmh* C([0,T],L2)} \\
		\sup_{h} \E^* \left[ \la \int_0^T |D^2 m_h^*(t)|_{\L^2}^2 \ dt \ra^p \right] &< \infty. \label{mh* est D2mh* L2([0,T],L2)} 
	\end{align}
	Since $ |\rho_w'| \leq w \rho $ for $ w >0 $, 
	by Gagliardo-Nirenberg inequality,
	\begin{align*}
		|m_h^*(t) \rho_{w_1}^\frac{1}{2}|_{\L^\infty} 
		&\leq C\ |D(m_h^*(t) \rho_\frac{w_1}{2})|_{\L^2}^\frac{1}{2} \ |m_h^*(t) \rho_\frac{w_1}{2}|_{\L^2}^\frac{1}{2} \\
		&\leq C \la |Dm_h^*(t) \rho_\frac{w_1}{2} |_{\L^2} + |m_h^*(t) \rho_\frac{w_1}{2}^\prime|_{\L^2} \ra^\frac{1}{2} \ |m_h^*(t)|_{\L^2_{w_1}}^\frac{1}{2} \\
		&\leq C \la |Dm_h^*(t)|_{\L^2} + \frac{w_1}{2} |m_h^*(t)|_{\L^2_{w_1}} \ra^\frac{1}{2} \ |m_h^*(t)|_{\L^2_{w_1}}^\frac{1}{2} \\
		&\leq \frac{C}{2} \la |Dm_h^*(t)|_{\L^2} + \frac{w_1}{2} |m_h^*(t)|_{\L^2_{w_1}} + |m_h^*(t)|_{\L^2_{w_1}} \ra,
	\end{align*}
	which implies
	\begin{equation}\label{mh* est mrho Linfty}
		\sup_{h} \E^* \left[ \sup_{t \in [0,T]} |m_h^*(t) \rho_{w_1}^\frac{1}{2}|_{\L^\infty}^{2p} \right] < \infty, \quad p \in [1, \infty). \\
	\end{equation}
	Thus, by \eqref{mh* est C([0,T],Lw1)} -- \eqref{mh* est mrho Linfty}, for $ p \in [1,\infty) $, 
	\begin{equation}\label{mh* ests cross products}
		\begin{aligned}
			\sup_{h} \E^* \left[ \la \int_0^T |m_h^*(t) \times Dm_h^*(t)|_{\L^2_{w_1}}^2 \ dt \ra^p \right] &<\infty, \\
			\sup_{h} \E^* \left[ \la \int_0^T |m_h^*(t) \times D^2 m_h^*(t) |_{\L^2_{w_1}}^2 \ dt \ra^p \right] &< \infty, \\
			\sup_{h} \E^* \left[ \la \int_0^T |m_h^*(t) \times (m_h^*(t) \times D^2 m_h^*(t)) |_{\L_{2w_1}^2}^2 \ dt \ra^p \right] &< \infty, 
		\end{aligned}
	\end{equation}	
	As in \eqref{|D+m|_L4 est},  
	\begin{align*}
		|Dm_h^*(t)|_{\L^4}^4
		&\leq \frac{C^2}{2} \ |D^2 m_h^*(t)|_{\L^2}^2 + \frac{1}{2} |Dm_h^*(t)|_{\L^2}^6,
	\end{align*}
	which implies
	\begin{equation}\label{mh* est Dm L4}
		\sup_{h} \E^* \left[ \la \int_0^T |D m_h^*(t)|_{\L^4}^4 \ dt \ra^p \right] < \infty, \quad p \in [1,\infty).
	\end{equation}

\subsection{Identification of the limit $ (m^*, W^*) $ and pathwise uniqueness}\label{Section: convergence results}

\subsubsection{Convergence of functions of $ m_h^* $}
	For $ p \in [1,\infty) $, by the pointwise convergence of $ m_h^* $ in \eqref{ptw conv: mh* L2(H1) + C(Xw-1)} and the uniform integrability of $ m_h^* $ and $ D m_h^* $ in \eqref{mh* est C([0,T],Lw1)} -- \eqref{mh* est Dmh* C([0,T],L2)}, we have
	\begin{align}
		m_h^* \to m^* &\quad \text{in $ L^{2p}(\Omega^*; L^2(0,T; \L^2_{w_2})) $}, \label{strong conv mh* L2w} \\
		Dm_h^* \to Dm^* &\quad \text{in $ L^{2p}(\Omega^*; L^2(0,T; \L^2_{w_2})) $}. \label{strong conv Dmh* L2w}
	\end{align} 
	By \eqref{mh* est Dmh* C([0,T],L2)}, $ Dm_h^* $ also converges weakly to a measurable process $ X $ in $ L^{2p}(\Omega^*; L^2(0,T; \L^2)) $, which implies that $ X = Dm^* \in L^{2p}(\Omega^*; L^2(0,T; \L^2)) $ by the uniqueness of the limit of weak convergence in $ L^{2p}(\Omega^*; L^2(0,T; \L^2_{w_2})) $.
	By \eqref{strong conv Dmh* L2w} and integration-by-parts, 
	\begin{equation}\label{weak conv D2mh* L2w}
		D^2m_h^* \rightharpoonup D^2 m^* \quad \text{in $ L^{2p}(\Omega^*; L^2(0,T; \L^2_{w_2})) $}.
	\end{equation}
	Similarly, by \eqref{mh* est D2mh* L2([0,T],L2)}, $ D^2m_h^* $ converges weakly to a measurable process $ Y $ in $ L^{2p}(\Omega^*; \L^2(0,T; \L^2)) $, thus $ Y = D^2 m^* \in L^{2p}(\Omega^*; L^2(0,T; \L^2)) $ and 
	\begin{equation}\label{weak conv D2mh* L2}
		D^2m_h^* \rightharpoonup D^2 m^* \quad \text{in $ L^{2p}(\Omega^*; L^2(0,T; \L^2)) $} .
	\end{equation}
	
	\begin{lemma}\label{Lemma: |m*|=1 a.s., mh* to m* in Lp, Dm* in L4 cap C}
		We have
		\begin{enumerate}[label=(\roman*)]
			\item 
			$ |m^*(t,x)| = 1 $, $ (t,x) $-a.e. $ \P^* $-a.s.
			
			\item
			$ m_h^* \rho_{w_2}^\frac{1}{2} \to m^* \rho_{w_2}^\frac{1}{2} $ in $ L^p(\Omega^*; L^p(0,T; \L^p)) $, for $ p \in [2,\infty) $,
			
			\item
			$ Dm^* \in L^4(\Omega^*; L^4(0,T; \L^4)) \cap L^p(\Omega^*; L^\infty(0,T;\L^2)) $, for $ p \in [2,\infty) $.  
		\end{enumerate}
	\end{lemma}
	\begin{proof}
		\underline{Part (i).} 
		Recall Lemma \ref{Lemma: mbar in B, W and |mbar| to 1}(iii), a similar argument holds for $ \L^2_{w_2} $ (in place of $ \L^2 $). Then,
		\begin{align*}
			&\E^* \left[ \int_0^T \int_\R \left| |m^*(t,x)| - 1 \right|^2 \rho_{w_2}(x) \ dx \ dt \right] \\
			&\leq 2\E^* \left[ \int_0^T \int_\R \left| |m_h^*(t,x)|-1 \right|^2 \rho_{w_2}(x) \ dx \ dt \right] 
			+ 2\E^* \left[ \int_0^T |m_h^*(t) - m^*(t)|^2_{\L^2_{w_2}} dt \right],
		\end{align*}
		where the first expectation on the right-hand side converges to $ 0 $ since the laws of $ \ol{m}^h $ and $ m_h^* $ are the same on $ L^2(0,T;\L^2_{w_2}) $, and the second expectation converges to $ 0 $ by \eqref{strong conv mh* L2w}. 
		Thus, 
		\begin{align*}
			\E^* \left[ \int_0^T \int_\R \left| |m^*(t,x)| - 1 \right|^2 \rho_{w_2}(x) \ dx \ dt \right] = 0,
		\end{align*}
		which implies $ |m^*(t,x)| = 1 $ a.e. on $ [0,T] \times \R $, $ \P^* $-a.s.
		This also means 
		\begin{align*}
			m^* \in L^p(\Omega^*; L^p(0,T; \L^p_w)), \quad \forall p \in [1,\infty), \ w \geq 1. 
		\end{align*}
	
		\underline{Part (ii).}		
		For $ w_2 > w_1 \geq 1 $ and $ p \in [2,\infty) $, 
		\begin{align*}
			&\E^* \left[ \int_0^T |(m_h^*(t) - m^*(t)) \rho_{w_2}^\frac{1}{2} |_{\L^p}^p \ dt \right] \\
			&\leq \E^* \left[ \sup_{t \in [0,T]} \left| (m_h^*(t) - m^*(t)) \rho_{w_2}^\frac{1}{2} \right|_{\L^\infty}^{p-1} \int_0^T \int_\R |(m_h^*(t,x) - m^*(t,x)) \rho_{w_2}^\frac{1}{2}(x)| \ dx \ dt \right] \\
			&\leq C \la \E^* \left[ \sup_{t \in [0,T]} \la |m_h^*(t) \rho_{w_1}^\frac{1}{2}|_{\L^\infty}^{p-1} + |m^*(t) \rho_{w_1}^\frac{1}{2}|_{\L^\infty}^{p-1} \ra^2  \right] \ra^\frac{1}{2} \\
			&\quad \times \la \E^* \left[ \int_0^T \int_\R |m_h^*(t,x) - m^*(t,x)|^2 \rho_{w_2}(x) \ dx \ dt  \right] \ra^\frac{1}{2},
		\end{align*}
		for some constant $ C $ that depends on $ p $ and $ T $. 
		Then, by the $ \L^\infty $-estimate \eqref{mh* est mrho Linfty}, Lemma \ref{Lemma: |m*|=1 a.s., mh* to m* in Lp, Dm* in L4 cap C}(i), \eqref{rhow properties} and the strong convergence \eqref{strong conv mh* L2w}, 
		\begin{align*}
			\lim_{h \to 0} \E^*\left[ \int_0^T |(m_h^*(t) - m^*(t)) \rho_{w_2}^\frac{1}{2} |_{\L^p}^p \ dt \right] = 0, \quad p \in [2, \infty). 
		\end{align*}	
		
		\underline{Part (iii).} 
		From part (i), we have $ \frac{1}{2} D|m^*(t,x)|^2 = \langle m^*, Dm^* \rangle(t,x) = 0 $ and thus 
		\begin{align*}
			\lb m^*(t,x), D^2m^*(t,x) \rb = - |Dm^*(t,x)|^2, 
		\end{align*}
		for $ (t,x) $-a.e. $ \P^* $-a.s.
		Since $ D^2m^* \in L^2(\Omega^*; L^2(0,T; \L^2)) $, we deduce 
		\begin{align*}
			\E^*\left[ \int_0^T |Dm^*(t)|_{\L^4}^4 \ dt \right]
			&= \E^*\left[ \int_0^T \int_\R \lb m^*(t,x), D^2m^*(t,x) \rb^2 \ dx \ dt \right] \\
			&\leq \E^*\left[ \int_0^T |D^2 m^*(t)|_{\L^2}^2 \ dt \right] 
			< \infty.
		\end{align*}
		As in \cite{Brzezniak2013_sLLG}, we extend the definition of the $ \L^2_{w_2} \cap \mathring{\H}^1 $-norm to $ \H^{-1}_{w_1} $ such that \smash{$ |u|_{\L^2_{w_2} \cap \mathring{\H}^1} = \infty $} if the function $ u $ is in $ \H^{-1}_{w_1} $ but not $ \L^2_{w_2} \cap \mathring{\H}^1 $, where the extended map 
		\begin{align*}
			u \mapsto \sup_{t \in [0,T]} |u(t)|_{\L^2_{w_2} \cap \mathring{\H}^1}, \quad u \in \mathcal{C}([0,T]; \H^{-1}_{w_1}),
		\end{align*}
		is lower semicontinuous. Then, by the pointwise convergence \eqref{ptw conv: mh* L2(H1) + C(Xw-1)}, Fatou's lemma and \eqref{mh* est Dmh* C([0,T],L2)}, 
		\begin{align*}
			\E^*\left[ \sup_{t \in [0,T]} |Dm^*(t)|_{\L^2}^p \right]
			&\leq \liminf_{h \to 0} \E^*\left[\sup_{t \in [0,T]} |m_h^*(t)|_{\L^2_{w_2} \cap \mathring{\H}^1}^p  \right] 
			<\infty,
		\end{align*}
		for $ p \in [2,\infty) $.
	\end{proof}

	\begin{lemma}\label{Lemma: strong conv m*xDm* Lw2}
		We have the following strong convergences:
		\begin{enumerate}[label=(\roman*)]		
			\item 
			$ m_h^* \times Dm_h^* \to m^* \times Dm^* $ and $ \langle m_h^*, Dm_h^* \rangle \to 0 $ in $ L^2(\Omega^*; L^2(0,T; \L^2_{w_2})) $, 
			
			\item
			$ m_h^* \times \la m_h^* \times D m_h^* \ra \to m^* \times (m^* \times Dm^*) $ in $ L^2(\Omega^*; L^2(0,T; \L_{w_1+w_2}^2)) $.
		\end{enumerate}
	\end{lemma}
	\begin{proof}		
		\underline{Part (i).}
		Note that
		\begin{align*}
			m_h^* \times Dm_h^* - m^* \times Dm^* 
			= \la m_h^* - m^* \ra \times Dm_h^* + m^* \times \la Dm_h^* - Dm^* \ra. 
		\end{align*}
		Then by H{\"o}lder's inequality,
		\begin{align*}
			&\E^* \left[ \int_0^T |(m_h^*(t) - m^*(t)) \times Dm_h^*(t)|_{\L^2_{w_2}}^2 \ dt  \right] \\
			&\leq \la \E^*\left[ \int_0^T |(m_h^*(t) - m^*(t)) \rho_{w_2}^\frac{1}{2} |_{\L^4}^4 \ dt \right] \ra^\frac{1}{2} 
			\la \E^*\left[ \int_0^T |Dm_h^*(t)|_{\L^4}^4 \ dt \right] \ra^\frac{1}{2},
		\end{align*}
		where the last line converges to $ 0 $ by Lemma \ref{Lemma: |m*|=1 a.s., mh* to m* in Lp, Dm* in L4 cap C}(ii) and \eqref{mh* est Dm L4}. 	
		Similarly, by Lemma \ref{Lemma: |m*|=1 a.s., mh* to m* in Lp, Dm* in L4 cap C}(i),
		\begin{align*}
			\E^* \left[ \int_0^T |m^*(t) \times (Dm_h^*(t) - Dm^*(t)) |_{\L^2_{w_2}}^2 \ dt \right] 
			&\leq \E^*\left[ \int_0^T |Dm_h^*(t) - Dm^*(t)|_{\L^2_{w_2}}^2 \ dt \right],
		\end{align*}
		where the right-hand side converges to $ 0 $ by \eqref{strong conv Dmh* L2w}. 	
		Therefore, 
		\begin{equation}\label{pf: strong conv m*xDm*}
			\lim_{h \to 0} \E^* \left[ \int_0^T |m_h^*(t) \times Dm_h^*(t) - m^*(t) \times Dm^*(t) |_{\L^2_w}^2 \ dt \right] = 0. 
		\end{equation}
		Since $ |m^*(t,x)|=1 $, we have $ \langle m^*, Dm^* \rangle (t,x)=0 $. By the same argument as above (replacing cross product with scalar product), $ \langle m_h^*, Dm_h^* \rangle \to \langle m^*, Dm^* \rangle = 0 $ in $ L^2(\Omega^*; L^2(0,T; \L^2_{w_2})) $.
		
		\underline{Part (ii).} 
		Note that
		\begin{align*}
			&m_h^* \times \la m_h^* \times Dm_h^* \ra - m^* \times \la m^* \times Dm^* \ra \\
			&= \la m_h^* - m^* \ra \times \la m_h^* \times Dm_h^* \ra 
			+ m^* \times \la m_h^* \times Dm_h^* - m^* \times Dm^* \ra.
		\end{align*}
		Then, with $ \rho_{w_1 + w_2} = \rho_{w_1} \rho_{w_2} $,
		\begin{align*}
			&\E^* \left[ \int_0^T \left| (m_h^*(t) - m^*(t)) \times \la m_h^*(t) \times  Dm_h^*(t) \ra \right|_{\L^2_{w_1+w_2}}^2 \ dt \right] \\
			&\leq \la \E^* \left[ \int_0^T |(m_h^*(t) - m^*(t)) \rho_{w_2}^\frac{1}{2} |_{\L^4}^4 \ dt \right] \ra^{\frac{1}{2}} 
			\la \E^* \left[ \int_0^T | m_h^*(t) \rho_{w_1}^\frac{1}{2} \times Dm_h^*(t) |_{\L^4}^4 \right] \ra^\frac{1}{2} \\
			&\leq \la \E^* \left[ \int_0^T |(m_h^*(t) - m^*(t)) \rho_{w_2}^\frac{1}{2} |_{\L^4}^4 \ dt \right] \ra^{\frac{1}{2}} \\
			&\quad \times \la \E^*\left[ \sup_{t \in [0,T]} | m_h^*(t) \rho_{w_1}^\frac{1}{2} |_{\L^\infty}^8 \right] \ra^{\frac{1}{4}} 
			\la \E^*\left[ \la \int_0^T \left| Dm_h^*(t) \right|_{\L^4}^4 dt \ra^2 \right] \ra^{\frac{1}{4}},
		\end{align*}
		where the first expectation in the last inequality converges to $ 0 $ by Lemma \ref{Lemma: |m*|=1 a.s., mh* to m* in Lp, Dm* in L4 cap C}(ii) and the second and third expectations are finite by \eqref{mh* est mrho Linfty} and \eqref{mh* est Dm L4}.
		Also, 
		\begin{align*}
			&\E^* \left[  \int_0^T \left| m^*(t) \times \la m_h^*(t) \times Dm_h^*(t) - m^*(t) \times Dm^*(t) \ra \right|_{\L^2_{w_1+w_2}}^2 \ dt \right]
		\end{align*}
		converges to $ 0 $ Lemma \ref{Lemma: |m*|=1 a.s., mh* to m* in Lp, Dm* in L4 cap C}(i) and part (i).
		Then, the strong convergence of $ m_h^* \times (m_h^* \times Dm_h^*) $ follows as desired. 
	\end{proof}

	\begin{lemma}\label{Lemma: weak conv m*xD2m*, m*x(m*D2m*) Lw2}
		Assume that $ w_2 \geq 4w_1 $. For any measurable process $ \varphi \in L^4(\Omega^*; \L^4(0,T; \L^4_{w_2})) $, we have the following weak convergences (with test function $ \varphi $):
		\begin{enumerate}[label=(\roman*)]
			\item 
			$ m_h^* \times (m_h^* \times D m_h^*) \rightharpoonup m^* \times (m^* \times Dm^*) $ in $ L^2(\Omega^*; L^2(0,T; \L^2_{w_2})) $,
		
			\item
			$ |m_h^* \times Dm_h^*|^2 m_h^* \rightharpoonup |m^* \times Dm^*|^2m^* $ in $ L^2(\Omega^*; L^2(0,T; \L^2_{w_2})) $,
			
			\item
			$ Dm_h^* \times (m_h^* \times Dm_h^*) \rightharpoonup Dm^* \times (m^* \times Dm^*) $ in $ L^2(\Omega^*; L^2(0,T; \L^2_{w_2})) $,
			
			\item
			$ \langle m_h^*, Dm_h^* \rangle m_h^* \times Dm_h^* \rightharpoonup 0 $ in $ L^2(\Omega^*; L^2(0,T; \L^2_{w_2})) $,
			
			\item 
			$ m_h^* \times D^2m_h^* \rightharpoonup m^* \times D^2m^* $ in $ L^2(\Omega^*; L^2(0,T; \L^2_{w_2})) $,
			
			\item
			$ m_h^* \times (m_h^* \times D^2m_h^*) \rightharpoonup m^* \times (m^* \times D^2m^*) $ in $ L^2(\Omega^*; L^2(0,T; \L^2_{w_2})) $.
		\end{enumerate}
	\end{lemma}
	\begin{proof}		
		\underline{Part (i).}
		As in lemma \ref{Lemma: strong conv m*xDm* Lw2}(ii), we first observe that
		\begin{equation}\label{pf: G(mh) weak conv part 1}
		\begin{aligned}
			&\E^*\left[ \int_0^T \lb (m_h^* - m^*) \times (m_h^* \times Dm_h^*), \varphi \rb_{\L^2_{w_2}}(t) \ dt \right] \\
			&\leq \la \E^*\left[ \int_0^T | \la m_h^*(t) - m^*(t) \ra \rho_{w_2}^\frac{1}{2} |_{\L^4}^4 \ dt \right] \ra^\frac{1}{4} \\
			&\quad \times \la \E^*\left[ \int_0^T | m_h^*(t) \times Dm_h^*(t) \rho_{w_2}^\frac{1}{4} |_{\L^2}^2 \ dt \right] \ra^\frac{1}{2}  
			\la \E^*\left[ \int_0^T |\varphi(t)|_{\L^4_{w_2}}^4 \right]  \ra^\frac{1}{4},
		\end{aligned}
		\end{equation}
		where the first expectation in the last line converges to $ 0 $ by Lemma \ref{Lemma: |m*|=1 a.s., mh* to m* in Lp, Dm* in L4 cap C}(ii), the second and the third expectations are finite by \eqref{mh* ests cross products} with $ w_2 \geq 2w_1 $ (equivalently, $ \rho_{w_2} \leq \rho_{w_1}^2 $) and $ \varphi \in L^4(\Omega^*; L^4(0,T; \L^4_{w_2})) $.
		Since $ |m^*(t,x)| = 1 $ from Lemma \ref{Lemma: |m*|=1 a.s., mh* to m* in Lp, Dm* in L4 cap C}(i), we have $ m^* \times \varphi \in L^2(\Omega^*; L^2(0,T; \L^2_{w_2})) $. Then by Lemma \ref{Lemma: strong conv m*xDm* Lw2}(i),
		\begin{equation}\label{pf: G(mh) weak conv part 2}
			\lim_{h \to 0} \E^*\left[ \int_0^T \lb m^* \times (m_h^* \times Dm_h^* - m^* \times Dm^*), \varphi \rb_{\L^2_{w_2}}(t) \ dt \right] = 0. 
		\end{equation}
		Combining \eqref{pf: G(mh) weak conv part 1} and \eqref{pf: G(mh) weak conv part 2}, we have the desired weak convergence for part (i). 
		
		\underline{Part (ii).}
		\begin{align*}
			&\left|m_h^* \times Dm_h^* \right|^2m_h^* - |m^* \times Dm^*|^2m^* \\
			&= \la \left|m_h^* \times Dm_h^* \right|^2- |m^* \times Dm^*|^2 \ra m_h^* + |m^* \times Dm^*|^2 \la m_h^* - m^* \ra \\
			&\leq \left| m_h^* \times Dm_h^* - m^* \times Dm^* \right| \left| m_h^* \times Dm_h^* + m^* \times Dm^* \right| \ |m_h^*| 
			+ |m^* \times Dm^*|^2 \ |m_h^* - m^*|.
		\end{align*}
		Then, for the first term in the line above,
		\begin{equation}\label{pf: G1(mh) conv part 1}
			\begin{aligned}
				&\E^*\left[ \int_0^T \lb \left| m_h^* \times Dm_h^* - m^* \times Dm^* \right| \left| m_h^* \times Dm_h^* + m^* \times Dm^* \right| m_h^*, \varphi \rb_{\L^2_{w_2}}(t) \ dt \right] \\
				&\leq \la \E^* \left[ \int_0^T \int_\R \la \left|m_h^* \times Dm_h^* \right| + |m^* \times Dm^*| \ra^2 |m_h^*|^2 |\varphi|^2 (t,x) \ \rho_{w_2}(x) \ dx \ dt \right]\ra^\frac{1}{2}  \\
				&\quad \times \la \E^*\left[ \int_0^T \left|m_h^* \times Dm_h^* - m^* \times Dm^*\right|_{\L^2_{w_2}}^2 (t) \ dt \right]\ra^\frac{1}{2}.
			\end{aligned}	
		\end{equation}
		We show that the first expectation on the right-hand side of \eqref{pf: G1(mh) conv part 1} is finite. 
		Since $ w_2 \geq 4w_1 $, it holds that $ \rho_{w_2} \leq \rho_{w_1}^2 \rho_{w_2}^\frac{1}{2} $ and 
		\begin{align*}
			&\E^* \left[ \int_0^T \int_\R \left|m_h^* \times Dm_h^* \right|^2 |m_h^*|^2 |\varphi|^2 (t,x) \ \rho_{w_2}(x) \ dx \ dt \right] \\
			&\leq \E^*\left[ \sup_{t \in [0,T]} |m_h^*(t) \rho_{w_1}^\frac{1}{2} |_{\L^\infty}^4 \int_0^T \int_\R |Dm_h^*|^2 |\varphi \rho_{w_2}^\frac{1}{4}|^2 (t,x) \ dx \ dt \right] \\
			&\leq \la \E^*\left[ \sup_{t \in [0,T]} |m_h^*(t) \rho_{w_1}^\frac{1}{2}|_{\L^\infty}^{16} \right] \ra^\frac{1}{4} 
			\times \la \E^*\left[ \la \int_0^T |Dm_h^*(t)|_{\L^4}^4 \ dt \ra^2 \right] \ra^\frac{1}{4} \\
			&\quad \times \la \E^*\left[ \int_0^T |\varphi(t)|_{\L^4_{w_2}}^4 \ dt \right] \ra^\frac{1}{2},
		\end{align*}
		where three expectations in the last inequality are finite by \eqref{mh* est mrho Linfty}, \eqref{mh* est Dm L4} and $ \varphi \in L^4(\Omega^*; L^4(0,T; \L^4_{w_2})) $. 
		Similarly, by Lemma \ref{Lemma: |m*|=1 a.s., mh* to m* in Lp, Dm* in L4 cap C}(i) and (iii),
		\begin{align*}
			&\E^* \left[ \int_0^T \int_\R \left|m^* \times Dm^* \right|^2 |\varphi|^2(t,x) \ \rho_{w_2}(x) \ dx \ dt \right] \\
			&\leq \la \E^*\left[ \int_0^T |Dm^*(t)|_{\L^4_{w_2}}^4 \ dt \right] \ra^\frac{1}{4} \la \E^*\left[ \int_0^T |\varphi(t)|_{\L^4_{w_2}}^4 \ dt \right] \ra^\frac{1}{4} < \infty.
		\end{align*}
		Hence, the left-hand side of \eqref{pf: G1(mh) conv part 1} converges to $ 0 $ as $ h \to 0 $ by Lemma \ref{Lemma: strong conv m*xDm* Lw2}(i).		
		Similarly, with $ |m^*(t,x)|=1 $, $ \P^* $-a.s. we have
		\begin{equation}\label{pf: G1(mh) conv part 2}
			\begin{aligned}
				&\E^*\left[ \int_0^T \int_\R \lb |m^* \times Dm^*|^2 \la m_h^*- m^* \ra (t,x), \varphi(t,x) \rb \rho_{w_2}(x) \ dx \ dt \right] \\
				&\leq \la \E^*\left[ \int_0^T |Dm^*(t)|^4_{\L^4} \ dt \right] \ra^\frac{1}{2} \times \la \E^*\left[ \int_0^T |\varphi(t)|_{\L^4_{w_2}}^4 \ dt \right] \ra^\frac{1}{4}\\
				&\quad \times \la \E^*\left[ \int_0^T |(m_h^*(t) - m^*(t)) \rho_{w_2}^\frac{1}{2}|_{\L^4}^4 \ dt  \right] \ra^\frac{1}{4}
			\end{aligned}
		\end{equation}
		where the last line converges to $ 0 $ by Lemma \ref{Lemma: |m*|=1 a.s., mh* to m* in Lp, Dm* in L4 cap C}(ii) -- (iii). 		
		Combining \eqref{pf: G1(mh) conv part 1} and \eqref{pf: G1(mh) conv part 2}, we have the desired weak convergence for part (ii).

		\underline{Part (iii).} 
		Note that
		\begin{align*}
			&Dm_h^* \times (m_h^* \times Dm_h^*) - Dm^* \times (m^* \times Dm^*) \\
			&= (Dm_h^* - Dm^*) \times (m^* \times Dm^*) + Dm_h^* \times (m_h^* \times Dm_h^* - m^* \times Dm^*).
		\end{align*}
		Then, 
		\begin{align*}
			&\E^* \left[ \int_0^T \lb (Dm_h^*(t) - Dm^*(t)) \times (m^*(t) \times Dm^*(t)), \varphi(t) \rho_{w_2} \rb_{\L^2} \ dt \right] \\
			&\leq \la \E^*\left[ \int_0^T |Dm^*(t) \rho_{w_2}^\frac{1}{4}|^4_{\L^4} \ dt \right] \ra^\frac{1}{4} \times \la \E^*\left[ \int_0^T |\varphi(t)|_{\L^4_{w_2}}^4 \ dt \right] \ra^\frac{1}{4}\\
			&\quad \times \la \E^*\left[ \int_0^T |Dm_h^*(t) - Dm^*(t)|_{\L^2_{w_2}}^2 \ dt  \right] \ra^\frac{1}{2},
		\end{align*}
		where the last line converges to $ 0 $ by Lemma \ref{Lemma: |m*|=1 a.s., mh* to m* in Lp, Dm* in L4 cap C}(iii) and \eqref{strong conv Dmh* L2w}. 
		Similarly, 
		\begin{align*}
			&\E^* \left[ \int_0^T \lb Dm_h^*(t) \times (m_h^*(t) \times Dm_h^*(t) - m^*(t) \times Dm^*(t)), \varphi(t) \rho_{w_2} \rb_{\L^2} \ dt \right] \\
			&\leq \la \E^*\left[ \int_0^T |Dm_h^*(t) \rho_{w_2}^\frac{1}{4}|^4_{\L^4} \ dt \right] \ra^\frac{1}{4} \times \la \E^*\left[ \int_0^T |\varphi(t)|_{\L^4_{w_2}}^4 \ dt \right] \ra^\frac{1}{4}\\
			&\quad \times \la \E^*\left[ \int_0^T |m_h^*(t) \times Dm_h^*(t) - m^*(t) \times Dm^*(t)|_{\L^2_{w_2}}^2 \ dt  \right] \ra^\frac{1}{2},
		\end{align*}
		which converges to $ 0 $ by \eqref{mh* est Dm L4} and Lemma \ref{Lemma: strong conv m*xDm* Lw2}(i).
		Together, we have 
		\begin{align*}
			\lim_{h \to 0} \E^* \left[ \int_0^T \left| \langle Dm_h^* \times (m_h^* \times Dm_h^*) - Dm^* \times (m^* \times Dm^*), \varphi \rangle_{\L^2_{w_2}}(t) \right| dt \right] = 0.
		\end{align*}

		\underline{Part (iv).}
		Again, since $ w_2 \geq 4 w_1 $, we have $ \rho_{w_2}^\frac{1}{4} \leq \rho_{w_1} $. Then,
		\begin{align*}
			&\E^* \left[ \int_0^T \lb \ \langle m_h^*(t) , Dm_h^*(t) \rangle m_h^*(t) \times Dm_h^*(t), \varphi(t) \rho_{w_2} \rb_{\L^2} \ dt \right] \\
			&\leq \la \E^*\left[ \int_0^T |m_h^*(t) \times Dm_h^*(t) \rho_{w_2}^\frac{1}{4}|^4_{\L^4} \ dt \right] \ra^\frac{1}{4} \\
			&\quad \times \la \E^*\left[ \int_0^T |\varphi(t)|_{\L^4_{w_2}}^4 \ dt \right] \ra^\frac{1}{4} \la \E^*\left[ \int_0^T |\langle m_h^*(t), Dm_h^*(t) \rangle|_{\L^2_{w_2}}^2 \ dt  \right] \ra^\frac{1}{2} \\
			&\leq \la \E^*\left[ \sup_{t \in [0,T]} |m_h^*(t) \rho_{w_1}|_{\L^\infty}^8 \right] \ra^{\frac{1}{8}} \la \E^* \left[ \la \int_0^T |Dm_h^*(t)|^4_{\L^4} \ dt \ra^2 \right] \ra^\frac{1}{8} \\
			&\quad \times \la \E^*\left[ \int_0^T |\varphi(t)|_{\L^4_{w_2}}^4 \ dt \right] \ra^\frac{1}{4} \la \E^*\left[ \int_0^T |\langle m_h^*(t), Dm_h^*(t) \rangle|_{\L^2_{w_2}}^2 \ dt  \right] \ra^\frac{1}{2}
		\end{align*}
		where the right-hand side converges to $ 0 $ by \eqref{mh* est mrho Linfty} (with $ \rho_{w_1}^\frac{1}{2} \leq 1 $), \eqref{mh* est Dm L4} and the convergence of the scalar product in Lemma \ref{Lemma: strong conv m*xDm* Lw2}(i).

		\underline{Part (v).} 
		\begin{align*}
			&\langle m_h^* \times D^2m_h^* - m^* \times D^2m^*, \varphi \rangle_{\L^2_{w_2}} \\
			&= \langle (m_h^* - m^*) \times D^2 m_h^*, \varphi \rho_{w_2} \rangle_{\L^2}
			+ \langle m^* \times (D^2m_h^* - D^2m^*), \varphi \rho_{w_2} \rangle_{\L^2}.
		\end{align*}
		Then, for the first term on the right-hand side, 
		\begin{align*}
			&\E^* \left[ \int_0^T \lb (m_h^*(t) - m^*(t)) \times D^2 m_h^*(t), \varphi(t) \rho_{w_2} \rb_{\L^2} \ dt \right] \\
			&\leq \E^* \left[ \int_0^T |(m_h^*(t) - m^*(t)) \rho_{w_2} \times \varphi(t)|_{\L^2} \ |D^2m_h^*(t)|_{\L^2} \ dt \right] \\
			&\leq \E^* \left[ \int_0^T |(m_h^*(t) -m^*(t)) \rho_{w_2}^\frac{1}{2}|_{\L^4} \ |\varphi(t) \rho_{w_2}^\frac{1}{2}|_{\L^4}\ |D^2m_h^*(t)|_{\L^2} \ dt \right] \\
			&\leq \la \E^* \left[ \int_0^T |(m_h^*(t) - m^*(t)) \rho_{w_2}^\frac{1}{2}|_{\L^4}^4 \ dt \right] \ra^\frac{1}{4} \times \la \E^* \left[ \int_0^T |\varphi(t) \rho_{w_2}^\frac{1}{2}|_{\L^4}^4 \ dt \right] \ra^\frac{1}{4} \\
			&\quad \times \la \E^* \left[ \int_0^T |D^2 m_h^*(t)|_{\L^2}^2 \ dt \right] \ra^\frac{1}{2},
		\end{align*}
		where the first expectation in the last inequality converges to $ 0 $ by Lemma \ref{Lemma: |m*|=1 a.s., mh* to m* in Lp, Dm* in L4 cap C}(ii), the second expectation is finite as $ \varphi \in L^4(\Omega^*; \L^4(0,T; \L^4_{w_2})) $ and the final expectation is finite by \eqref{mh* est D2mh* L2([0,T],L2)}. 
		Thus, the left-hand side converges to $ 0 $ as $ h \to 0 $. 		
		Also, $ m^* \times \varphi \in L^2(\Omega^*; L^2(0,T; \L^2_{w_2})) $ and then by the weak convergence \eqref{weak conv D2mh* L2w}, 
		\begin{align*}
			\lim_{h \to 0} \E^* \left[ \int_0^T  \langle m^*(t) \times (D^2m_h^*(t) - D^2m^*(t)), \varphi(t) \rangle_{\L^2_{w_2}} \ dt \right] = 0.
		\end{align*}
		Therefore, 
		\begin{align*}
			\lim_{h \to 0} \E^* \left[ \int_0^T \left| \langle m_h^*(t) \times D^2 m_h^*(t) - m^*(t) \times D^2 m^*(t), \varphi(t) \rangle_{\L^2_{w_2}} \right| dt \right] = 0.
		\end{align*}

		\underline{Part (vi).} 
		Similarly, 
		\begin{align*}
			&\lb m_h^* \times \la m_h^* \times D^2m_h^* \ra - m^* \times \la m^* \times D^2m^* \ra, \varphi \rb_{\L^2_{w_2}} \\
			&= \langle (m_h^* - m^*) \times \la m_h^* \times D^2 m_h^* \ra, \varphi \rangle_{\L^2_{w_2}} 
			+ \langle m^* \times (m_h^* \times D^2m_h^* - m^* \times D^2m^*), \varphi \rangle_{\L^2_{w_2}}.
		\end{align*}
		Then, 
		\begin{align*}
			&\E^* \left[ \int_0^T \langle (m_h^*(t) - m^*(t)) \times \la m_h^*(t) \times D^2 m_h^*(t) \ra, \varphi(t) \rangle_{\L^2_{w_2}} \ dt \right] \\
			&\leq \la \E^* \left[ \int_0^T |(m_h^*(t) - m^*(t)) \rho_{w_2}^\frac{1}{2}|_{\L^4}^4 \ dt \right] \ra^\frac{1}{4} \times \la \E^* \left[ \int_0^T |\varphi(t) \rho_{w_2}^\frac{1}{4}|_{\L^4}^4 \ dt \right] \ra^\frac{1}{4} \\
			&\quad \times \la \E^* \left[ \int_0^T |m_h^*(t) \times D^2 m_h^*(t) \rho_{w_2}^\frac{1}{4}|_{\L^2}^2 \ dt \right] \ra^\frac{1}{2},
		\end{align*}
		where the last line converges to $ 0 $ by Lemma \ref{Lemma: |m*|=1 a.s., mh* to m* in Lp, Dm* in L4 cap C}(ii), $ \varphi \in L^4(\Omega^*; \L^4(0,T; \L^4_{w_2})) $ and \eqref{mh* ests cross products} with \smash{$ \rho_{w_1} \geq \rho_{w_2}^\frac{1}{4} $}. 
		Also, we have
		\begin{align*}
			&\E^* \left[ \int_0^T \langle m^*(t) \times (m_h^*(t) \times D^2m_h^*(t) - m^*(t) \times D^2m^*(t)), \varphi(t) \rho_{w_2} \rangle_{\L^2} \ dt \right] \\ 
			&\leq \E^* \left[ \int_0^T  \langle m_h^*(t) \times D^2m_h^*(t) - m^*(t) \times D^2m^*(t), m^*(t) \times \varphi(t)  \rangle_{\L^2_{w_2}} \ dt \right],
		\end{align*}
		where $ m^* \times \varphi \in L^4(\Omega^*; L^4(0,T; \L^4_{w_2})) $ and thus by part (iii), the expectation in the last line above converges to $ 0 $. 		
		Therefore, 
		\begin{align*}
			\lim_{h \to 0} \E^* \left[ \int_0^T \langle m_h^* \times \la m_h^* \times D^2 m_h^* \ra - m^* \times \la m^* \times D^2 m^* \ra, \varphi \rangle_{\L^2_{w_2}}(t) \ dt \right] = 0.
		\end{align*}
	\end{proof}

	\begin{lemma}\label{Lemma: conv F, G1, G2, G}
		Assume that $ w_2 \geq 4w_1 $. 
		Recall the definitions \eqref{defn: F*,G*}, we have 
		\begin{enumerate}[label=(\roman*)]
			\item 
			$ F_{\wh{v}}(m_h^*) \rightharpoonup F(m^*) $ in $ L^2(\Omega^*; L^2(0,T; \L^2_{w_2})) $, 
			
			\item
			$ S_{\wh{\kappa}}(m_h^*) \rightharpoonup S(m^*) $ in $ L^2(\Omega^*; L^2(0,T; \L^2_{w_2})) $, 
			
			\item
			$ \kappa G(m_h^*) \to \kappa G(m^*) $ (strongly) in $ L^2(\Omega^*; L^2(0,T; \L^2_{w_1+w_2})) $.
		\end{enumerate}
	\end{lemma}
	\begin{proof}
		As in Lemma \ref{Lemma: weak conv m*xD2m*, m*x(m*D2m*) Lw2}, let $ \varphi $ be an arbitrary measurable process in $ L^4(\Omega^*; L^4(0,T; \L^4_{w_2})) $. 
		By \eqref{Ckappa bound} and \eqref{Cv bound}, $ \kappa^2, \kappa \kappa' \in \L^\infty \cap \H^1 $ and $ v \in \mathcal{C}([0,T]; \L^\infty \cap \H^1) $. Then for $ y = \kappa^2, (\kappa^2)^-, \kappa \kappa', v $, any piecewise constant approximation (in the $ x $-variable) $ z $ of $ y $ satisfies
		\begin{equation}\label{pf: yhat to y in Lw4}
			z \to y \quad \text{ in } L^4(0,T; \L^4_{w_2}). 
		\end{equation}
		For example, the approximation $ z $ can be taken to be $ \wh{y}^- $ or $ \wh{y} $. 
		Let $ u $ be a function such that $ u(m_h^*) \in L^2(\Omega^*; L^2(0,T; \L^2_{w_2})) $. Then,
		\begin{equation}\label{pf: yhatu to yu in Lw4}
		\begin{aligned}
			&\E^* \left[ \int_0^T \lb z(t) u(m_h^*(t)) - y(t) u(m^*(t)), \varphi(t) \rb_{\L^2_{w_2}} \ dt \right] \\
			&\leq \E^*\left[ \int_0^T \lb (z - y) \times u(m_h^*), \varphi \rb_{\L^2_{w_2}}(t) \ dt \right] + \E^*\left[ \int_0^T \lb y\la u(m_h^*) - u(m^*) \ra, \varphi \rb_{L^2_{w_2}}(t) \ dt \right] \\
			&\leq \la \E^* \left[ \int_0^T |z(t) - y(t)|_{\L^4_{w_2}}^4 \ dt \right] \ra^\frac{1}{4} 
			\la \E^*\left[ \int_0^T |u(m_h^*(t))|_{\L^2_{w_2}}^2 \ dt \right] \ra^\frac{1}{2} 
			\la \E^*\left[ \int_0^T |\varphi(t)|_{\L^4_{w_2}}^4 \ dt \right] \ra^\frac{1}{4} \\
			&\quad + \E^*\left[ \int_0^T \langle u(m_h^*(t)) - u(m^*(t)), y(t)\varphi(t) \rangle_{\L^2_{w_2}} \ dt  \right].
		\end{aligned}
		\end{equation}
		If $ u(m_h^*) \rightharpoonup u(m^*) $ in $ L^2(\Omega^*; L^2(0,T; \L_{w_2}^2)) $, then the right-hand side of \eqref{pf: yhatu to yu in Lw4} converges to $ 0 $ by \eqref{pf: yhat to y in Lw4} and $ y \varphi \in L^4(\Omega^*; L^4(0,T; \L^4_{w_2})) $.

		\underline{Part (i).}
		Let $ u(m_h^*) = m_h^* \times (m_h^* \times Dm_h^*) $ and let $ y = v $. 
		The result follows immediately from \eqref{pf: yhatu to yu in Lw4} and Lemma \ref{Lemma: weak conv m*xD2m*, m*x(m*D2m*) Lw2}(v) and (vi).

		\underline{Part (ii).} 
		Since $ w_2 \geq 4w_1 $, we have $ \rho_{w_2} \leq \rho_{w_1}^4 $,
		We observe that from \eqref{mh* est Dm L4} and \eqref{mh* est mrho Linfty} that
		\begin{align*}
			Dm_h^* \times (m_h^* \times Dm_h^*), \quad  
			|m_h^* \times Dm_h^*|^2 m_h^*, \quad 
			\langle m_h^*, Dm_h^* \rangle m_h^* \times Dm_h^* 
		\end{align*}
		are in $ L^2(\Omega^*; L^2(0,T; \L_{w_2}^2)) $. 
			%
			%
		Taking the following choices of $ u $, $ y $ and $ z $:
		\begin{alignat*}{4}
			u(m_h^*) &= (\gamma^2-1) m_h^* \times (m_h^* \times D^2 m_h^*) - 2 \gamma m_h^* \times D^2 m_h^*, &&\quad y = \kappa^2, &&\quad z = \frac{1}{2} \la (\wh{\kappa}^-)^2 + \wh{\kappa}^2 \ra, \\
			u(m_h^*) &= \gamma^2 Dm_h^* \times (m_h^* \times Dm_h^*) + |m_h^* \times Dm_h^*|^2 m_h^*, &&\quad y = \kappa^2, &&\quad z = \wh{\kappa}^2, \\
			u(m_h^*) &= 2 \gamma \langle m_h^*, Dm_h^* \rangle \ m_h^* \times Dm_h^*, &&\quad y = \kappa^2, &&\quad z = (\wh{\kappa}^2)^-, \\
			u(m_h^*) &= \left[ (\gamma^2-1) m_h^* \times (m_h^* \times Dm_h^*) - 2\gamma m_h^* \times Dm_h^* \right], &&\quad y = \kappa \kappa', &&\quad z = \wh{\kappa \kappa'},
		\end{alignat*}
		and using Lemma \ref{Lemma: weak conv m*xD2m*, m*x(m*D2m*) Lw2}(ii) -- (vi), we follow again the argument \eqref{pf: yhatu to yu in Lw4} to obtain weak convergence of $ S_{\wh{\kappa}}(m_h^*) $ to $ S(m^*) $ in $ L^2(\Omega^*; L^2(0,T; \L^2_{w_2})) $.

		\underline{Part (iii).} 
		The result follows from \eqref{Ckappa bound} and Lemma \ref{Lemma: strong conv m*xDm* Lw2}.
	\end{proof}

\subsubsection{Wiener process}	
	Define a sequence of processes $ \{ \ol{M}_h \}_{h > 0} $ on $ (\Omega,\mathcal{F}, \mathbb{P}) $ by
	\begin{align*}
		\ol{M}_h(t) 
		&:= \int_0^t \la G(\ol{m}^h(s))  - (\gamma P_1 + P_2)\ol{m}^h(s) \ra d\wh{W}^h(s).
	\end{align*}
	Recall the equation of $ \ol{m}^h $, we have from \eqref{SDE: m-bar with remainders} that
	\begin{align*}
		\ol{M}_h(t) &= \ol{m}^h(t) - m_0 - \int_0^t \la F_{\wh{v}}(\ol{m}^h(s))+ \frac{1}{2} S_{\wh{\kappa}}(\ol{m}^h(s)) \ra \ ds - R_0\ol{m}^h(t) - \int_0^t R^h \ol{m}^h(s)  \ ds,
	\end{align*}
	where the operator $ R^h $ is given by
	\begin{align*}
		R^h u
		&:= -\wh{v}(\gamma P_1 + P_2)u + Q_1u + \alpha Q_2u 
		+\frac{1}{4} \la (\wh{\kappa}^2)^- + \wh{\kappa}^2 \ra \left[ 2 \gamma Q_1u - (\gamma^2 -1) Q_2u \right] \\
		&\quad + \frac{1}{2}\la \wh{\kappa}^2 \la P_3 + \gamma^2 P_4 \ra u - 2 \gamma (\wh{\kappa}^2)^- P_5u \ra
		+ \frac{1}{2} \wh{\kappa \kappa'} \left[ 2\gamma P_1 - (\gamma^2-1)P_2 \right] u.
	\end{align*}
	Similarly, define a sequence of processes $ \{ M_h^* \}_{h > 0} $ on $ (\Omega^*, \mathcal{F}^*, \mathbb{P}^*) $ by
	\begin{align*}
		M_h^*(t)
		&:= m_h^*(t) - m_0 - \int_0^t \la F_{\wh{v}}(m_h^*(s)) + \frac{1}{2} S_{\wh{\kappa}}(m_h^*(s)) \ra \ ds 
		- R_0 m_h^*(t) - \int_0^t R^h m_h^*(s)  \ ds. 
	\end{align*}

	\vspace{0.5cm}
	\begin{lemma}\label{Lemma: weak conv Mh* to m* drift}
		For each $ t \in (0,T] $, we have the following weak convergence in $ L^2(\Omega^*; \H^{-1}_{w_1}) $:
		\begin{align*}
			M_h^*(t) \rightharpoonup 
			M^*(t) := m^*(t) - m_0 - \int_0^t \la F(m^*(s))+ \frac{1}{2} S(m^*(s)) \ra \ ds.
		\end{align*}
	\end{lemma}
	\begin{proof}
		Recall that $ \H_{w_1}^1 $ is compactly embedded in $ \L^2_{w_2} $. 
		Let $ t \in (0,T] $ and $ \varphi \in L^2(\Omega^*; \H_{w_1}^1) $. 	
		By Remark \ref{Remark: borel sets B vs E}, the two sets of remainders 
		\begin{align*}
			&\{ R_0 \ol{m}^h, R_1 \ol{m}^h, P_1\ol{m}^h, \ldots, P_5\ol{m}^h, Q_1\ol{m}^h, Q_2\ol{m}^h \}, \\
			&\{ R_0 m_h^*, R_1 m_h^*, P_1m_h^*, \ldots, P_5m_h^*, Q_1m_h^*, Q_2m_h^* \},
		\end{align*}
		have the same laws for $ \ol{m}^h, m_h^* \in L^2(0,T; \L^2_{w_1} \cap \mathring{\H}^2) \cap \mathcal{C}([0,T]; \L^2_{w_1} \cap \mathring{\H}^1) $. 
		Then, by Lemma \ref{Lemma: interpolation 2 remainders}, 
		\begin{align*}
			&\lim_{h \to 0} \E^* \left[ _{\H^{-1}_{w_1}} \lb R_0m_h^*(t) +\int_0^t R^hm_h^*(s) \ ds, \varphi \rb_{\H^1_{w_1}} \right] \\
			&= \lim_{h \to 0} \E^*\left[ \lb R_0m_h^*(t), \varphi \rb_{\L^2_{w_2}} + \int_0^t \lb R^h m_h^*(s), \varphi \rb_{\L^2_{w_2}} \ ds \right] = 0.
		\end{align*}
		By the pointwise convergence \eqref{ptw conv: mh* L2(H1) + C(Xw-1)} of $ m_h^* $ in $ \mathcal{C}([0,T]; \H^{-1}_{w_1}) $ and Lemma \ref{Lemma: conv F, G1, G2, G}(i) -- (ii), 
		\begin{align*}
			&\lim_{h \to 0} \E^*\left[ _{\H^{-1}_{w_1}} \lb M_h^*(t), \varphi \rb_{\H^1_{w_1}} \right] \\
			&= \lim_{h \to 0} \E^*\left[ _{\H^{-1}_{w_1}} \lb m_h^*(t)-m_0, \varphi \rb_{\H^1_{w_1}} - \int_0^t \lb F_{\wh{v}}(m_h^*(s)) + \frac{1}{2} S_{\wh{\kappa}}(m_h^*(s)), \varphi \rb_{\L^2_{w_2}} ds \right] \\
			&\quad - \lim_{h \to 0} \E^*\left[ _{\H^{-1}_{w_1}} \lb R_0 m_h^*(t) + \int_0^t R^h m_h^*(s), \varphi \rb_{\H^1_{w_1}}  \right] \\
			&= \E^*\left[ _{\H^{-1}_{w_1}} \lb M^*(t), \varphi \rb_{\H^1_{w_1}} \right]. 
		\end{align*}
	\end{proof}
	
	\begin{lemma}\label{Lemma: W* is Wiener process}
		The process $ W^* $ is a $ Q $-Wiener process on $ (\Omega^*, \mathcal{F}^*, \mathbb{P}^*) $, and $ W^*(t) - W^*(s) $ is independent of the $ \sigma $-algebra generated by $ m^*(r) $ and $ W^*(r) $ for $ r \in [0,s] $.  
	\end{lemma}
	\begin{proof}
		See \cite[Lemma 5.2(i)]{Brzezniak2013_sLLG} (using Lemma \ref{Lemma: Skorohod}). 
	\end{proof}
	
	\begin{lemma}\label{Lemma: M*=stochastic int}
		For each $ t \in [0,T] $, 
		\begin{align*}
			M^*(t) = \int_0^t G(m^*(s)) \ dW^*(s).
		\end{align*}
	\end{lemma}
	\begin{proof}
		Fix $ h $ and $ t \in (0,T] $. 
		For each $ n \in {\b N} $, define the partition $ \{ s_i^n = \frac{iT}{n} : i = 0, \ldots, n \} $. Define 
		\begin{align*}
			\delta\wh{W}^h(t,s_i^n) &:= \wh{W}^h(t \wedge s_{i+1}^n) - \wh{W}^h(t \wedge s_i^n), \\
			\delta\wh{W}_h^*(t,s_i^n) &:= \wh{W}_h^*(t \wedge s_{i+1}^n) - \wh{W}_h^*(t \wedge s_i^n), \\
			\delta W^*(t,s_i^n) &:= W^*(t \wedge s_{i+1}^n) - W^*(t \wedge s_i^n),
		\end{align*}
		where $ \wh{W}_h^*(s) $ is the piecewise constant approximation of $ W_h^*(s) $ (as in \eqref{interpolation 1}) for every $ s \in [0,T] $. 
		As in \eqref{ptw conv: Wh}, we also have
		\begin{equation}\label{ptw conv: Wh*hat} 
			\wh{W}_h^* \to W^* \text{ in } \mathcal{C}([0,T]; L^2(\R)), \quad \P^*\text{-a.s.}
		\end{equation}
		Consider the following two $ \L^2_{w_2} $-valued random variables:
		\begin{align*}
			\ol{Y}_{h,n}(t) &:= \ol{M}_h(t) - \sum_{i=0}^{n-1} \la G(\ol{m}^h(s_i^n)) - \la \gamma P_1 + P_2 \ra \ol{m}^h(s_i^n) \ra \delta\wh{W}^h(t,s_i^n), \\
			Y^*_{h,n}(t) &:= M_h^*(t) - \sum_{i=0}^{n-1} \la G(m_h^*(s_i^n)) - \la \gamma P_1 + P_2 \ra m_h^*(s_i^n) \ra \delta\wh{W}_h^*(t,s_i^n).
		\end{align*}
		Following Remark \ref{Remark: borel sets B vs E}, $ \ol{Y}_{h,n} $ and $ Y^*_{h,n} $ have the same distribution. 
		As $ n \to \infty $, 
		\begin{align*}
			\ol{Y}_{h,n}(t) \to \ol{M}_h(t) - \int_0^t \la G(\ol{m}^h(s)) - \la \gamma P_1 + P_2 \ra \ol{m}^h(s) \ra \ d\wh{W}^h(s) = 0 \text{ in } L^2(\Omega; \L^2_{w_2}).
		\end{align*}
		This implies that $ Y^*_{h,n}(t) $ also converges to $ 0 $ in $ L^2(\Omega^*; \L^2_{w_2}) $ as $ n \to \infty $. Thus,
		\begin{align*}
			M_h^*(t) = \int_0^t \la G(m_h^*(s)) -  \la \gamma P_1 + P_2 \ra m_h^*(s) \ra \ d\wh{W}_h^*(s), \quad \P^*\text{-a.s.}
		\end{align*}
		
		We observe that
		\begin{align*}
			&\E^*\left[ M_h^*(t) - \int_0^t G(m^*(s)) \ dW^*(s) \right] \\
			&= \E^*\left[ \int_0^t \la G(m_h^*(s)) - (\gamma P_1 + P_2)m_h^*(s) - \sum_{i=0}^{n-1} G(m_h^*(s_i^n)) \mathbbm{1}_{(s_i^n, s_{i+1}^n]}(s) \ra \ d\wh{W}_h^*(s)  \right] \\
			&\quad + \E^*\left[ \sum_{i=0}^{n-1} G(m_h^*(s_i^n)) \ \delta\wh{W}_h^*(t,s_i^n)- \sum_{i=0}^{n-1} G(m^*(s_i^n)) \ \delta W^*(t \wedge s_i^n) \right] \\
			&\quad + \E^*\left[ \int_0^t \la \sum_{i=0}^{n-1} G(m^*(s_i^n)) \mathbbm{1}_{(s_i^n,s_{i+1}^n]}(s) - G(m^*(s)) \ra \ dW^*(s) \right] \\
			&= \E^*[J_0^h] + \E^*[J_1^h] + \E^*[J_2].
		\end{align*}
		
		For $ J_0^h $ and $ J_2 $, let $ \epsilon > 0 $ and choose $ n \in {\b N} $ such that 
		\begin{equation}\label{pf: epsilon for G(s) vs G(sn)}
			\la \E^*\left[ \int_0^t \left|G(m^*(s)) - \sum_{i=0}^{n-1}G(m^*(s_i^n)) \mathbbm{1}_{(s_i^n, s_{i+1}^n]}(s) \right|_{\H^{-1}_{w_1}}^2 \ ds \right] \ra^\frac{1}{2} < \frac{\epsilon}{2}.
		\end{equation}
		Since $ \wh{W}_h^* $ and $ \wh{W}^h $ have the same laws on $ \mathcal{C}([0,T]; H^2(\R)) $, we have
		\begin{align*}
			\la \E^*\left[ |J_0^h|_{\H^{-1}_{w_1}}^2 \right] \ra^\frac{1}{2} 
			&\leq \la \E^* \left[ \int_0^t \sum_j q_j^2 \left| \wh{f}_j G(m_h^*(s)) - \wh{f}_j G(m^*(s)) \right|_{\H^{-1}_{w_1}}^2 ds \right] \ra^\frac{1}{2} \\
			&\quad + \la \E^*\left[ \int_0^t \sum_j q_j^2 \left| \wh{f}_j \la G(m^*(s)) - \sum_{i=0}^{n-1}G(m^*(s_i^n)) \mathbbm{1}_{(s_i^n, s_{i+1}^n]}(s) \ra \right|_{\H^{-1}_{w_1}}^2 ds \right] \ra^\frac{1}{2} \\
			&\quad + \la \E^*\left[ \int_0^t \sum_j q_j^2 \left| \sum_{i=0}^{n-1} \la \wh{f}_j G(m^*(s_i^n)) - \wh{f}_j G(m_h^*(s_i^n)) \ra \mathbbm{1}_{(s_i^n, s_{i+1}^n]}(s) \right|_{\H^{-1}_{w_1}}^2 ds \right] \ra^\frac{1}{2} \\
			&\quad + \la \E^*\left[ \int_0^t \sum_j q_j^2 \left| \wh{f}_j (\gamma P_1 + P_2) m_h^*(s) \right|_{\H^{-1}_{w_1}}^2 \ ds \right] \ra^\frac{1}{2}. 
		\end{align*}
		Recall that $ \L^2_w \hookrightarrow \H^{-1}_{w_1} $ for all $ w > w_1 $. 
		Let $ w = w_1+w_2 $. 
		As $ h \to 0 $, the first and the third term on the right-hand side converges to $ 0 $ by Lemma \ref{Lemma: conv F, G1, G2, G}(iii), the fourth term converges to $ 0 $ by Lemma \ref{Lemma: interpolation 2 remainders} and \eqref{Ckappa bound}, and the second term is less than $ \frac{\epsilon}{2} $ by \eqref{pf: epsilon for G(s) vs G(sn)}. 
		Hence, for a sufficiently small $ h $, we have
		\begin{align*}
			|J_0^h|_{L^2(\Omega^*; \H_{w_1}^{-1})} < \frac{\epsilon}{2}.
		\end{align*}
		Similarly, 
 		$ |J_2|_{L^2(\Omega^*; \H_{w_1}^{-1})} < \frac{\epsilon}{2} $. 
		
		For $ J_1^h $, we have
		\begin{align*}
			\E^*\left[ |J_1^h|_{\H^{-1}_{w_1}}^2 \right] 
			&\leq \E^*\left[ \left| \sum_{i=0}^{n-1} \la G(m_h^*(s_i^n)) - G(m^*(s_i^n)) \ra \delta W^*(t,s_i^n)\right|_{\H^{-1}_{w_1}}^2 \right] \\
			&\quad + \E^*\left[ \left| \sum_{i=0}^{n-1} G(m_h^*(s_i^n)) \la \delta\wh{W}_h^*(t,s_i^n)- \delta W^*(t,s_i^n) \ra \right|_{\H^{-1}_{w_1}}^2 \right].
		\end{align*}
		Since $ W^* $ is a $ Q $-Wiener process, the first term on the right-hand side converges to $ 0 $ by Lemma \ref{Lemma: conv F, G1, G2, G}(iii), 
		Also, the second term converges to $ 0 $ by the pointwise convergence \eqref{ptw conv: Wh} (or \eqref{ptw conv: Wh*hat}) and the result $ G(m_h^*) \rho_{\frac{w}{2}} \in L^2(\Omega^*; L^2(0,T;\L^\infty)) $, which can be deduced from the estimates \eqref{mh* est Dmh* C([0,T],L2)}, \eqref{mh* est D2mh* L2([0,T],L2)} and \eqref{mh* est mrho Linfty}.
		
		Therefore, for any sufficiently small $ h $, 
		\begin{align*}
			\E^*\left[ \left| M_h^*(t) - \int_0^t G(m^*(s)) \ dW^*(s) \right|_{\H^{-1}_{w_1}}^2 \right] < \epsilon^2.
		\end{align*}
		Using Lemma \ref{Lemma: weak conv Mh* to m* drift} and the uniqueness of weak limit, the proof is concluded.  
	\end{proof}
	We are ready to prove the main theorem.

\subsubsection{Proof of Theorem~\ref{Theorem: existence}}
	By Lemmata \ref{Lemma: weak conv Mh* to m* drift} and \ref{Lemma: M*=stochastic int}, $ m^* $ satisfies the \eqref{sLLS_mainthm} in $ \H^{-1}_{w_1} $. 
	Moreover, using Lemma \ref{Lemma: |m*|=1 a.s., mh* to m* in Lp, Dm* in L4 cap C}(i), we can simplify $ F, S $ and $ G $:
	\begin{align*}
		F(m^*) &= -v \la Dm^*+ \gamma m^* \times Dm^*\ra - m^* \times D^2m^* + \alpha D^2m^* + \alpha |Dm^*|^2m^*, \\
		S(m^*) 
		&= \kappa^2 \la (1-\gamma^2) D^2 m^* - 2\gamma^2 |Dm^*|^2 m^* -2 \gamma m^* \times D^2m^* \ra \\
		&\quad +\kappa \kappa' \la (1-\gamma^2) Dm^* - 2\gamma m^* \times Dm^* \ra, \\
		G(m^*) &= -Dm^* + \gamma m^* \times Dm^*, 
	\end{align*}
	and each of them is in $ L^2(\Omega^*; L^2(0,T; \L^2)) $, 
	hence the equality \eqref{sLLS_mainthm} holds in $ \L^2 $. 
	Recall the properties of $ m^* $ shown previously in \eqref{weak conv D2mh* L2} and Lemma \ref{Lemma: |m*|=1 a.s., mh* to m* in Lp, Dm* in L4 cap C}, we have now verified that $ m^* $ is a solution of \eqref{eq_ito} in the sense of Definition \ref{def: sol}. 
	It only remains to show that $ m-m_0 \in \mathcal{C}^\alpha([0,T]; \L^2) $. 
	For $ s,t \in [0,T] $ and $ p \in [1,\infty) $, there exists a constant $ C $ that may depend on $ p,T,C_{\kappa} $ such that
	\begin{align*}
		&\E^* \left[ |m^*(t)-m^*(s)|_{\L^2}^{2p} \right] \\
		&\leq |t-s|^p \E^* \left[ \la \int_s^t \left|F(m^*(r)) + \frac{1}{2}S(m^*(r)) \right|_{\L^2}^2 \ dr \ra^p \right] \\
		&\quad + \E^* \left[ \la \int_s^t \sum_j q_j^2 |f_j G(m^*)(r)|^2_{\L^2} \ dr \ra^p \right] \\
		&\leq C |t-s|^p \E^* \left[ \sup_{r \in [0,T]} |Dm^*(r)|^{2p}_{\L^2} + \la\int_s^t \la |Dm^*(r)|^4_{\L^4} + |D^2m^*(r)|^2_{\L^2}\ra \ dr \ra^p \right],
	\end{align*}
	where the expectation on the right-hand side is finite. 
	Then by Kolmogorov's continuity criterion, $ m^*(t)-m_0 \in \mathcal{C}^\alpha([0,T]; \L^2) $, $ \P^* $-a.s. for $ \alpha \in (0,\frac{1}{2}) $. 
	
\subsubsection{Proof of Theorem \ref{Thm: pathwise uniqueness}}
	Let $ (m_1,W) $ and $ (m_2,W) $ on $ (\Omega,\mathcal F,\la\mathcal F_t\ra,\P) $ be two solutions of \eqref{eq_ito} in the sense of Definition \ref{def: sol}. 
	Let $ u = m_1 - m_2 $ and $ w>0 $. 
	Applying It{\^o}'s lemma to $ \frac{1}{2}|u(t)|_{\L^2_w}^2 $,
	\begin{equation}\label{d|m1-m2|_Lw2}
	\begin{aligned}
		\frac{1}{2} |u(t)|^2_{\L^2_w} 
		&= |u(0)|^2_{\L^2_w} + \int_0^t \lb u(s), F(m_1(s)) - F(m_2(s)) \rb_{\L^2_w} \ dt \\
		&\quad + \frac{1}{2} \int_0^t \lb u(s), S(m_1(s))-S(m_2(s)) \rb_{\L^2_w} \ dt \\
		&\quad + \frac{1}{2} \int_0^t \sum_j q_j^2 \left|f_j (G(m_1(s)) - G(m_2(s))) \right|_{\L^2_w}^2 \ dt \\
		&\quad + \int_0^t \lb u(s), \la G(m_1(s)) - G(m_2(s)) \ra \ dW(s) \rb_{\L^2_w}  \\
		&= |u(0)|^2_{\L^2_w}+ \int_0^t \left[ U_1(s) + U_2(s) + U_3(s) \right] \ ds + U_4(t), 
	\end{aligned}
	\end{equation}	

	\underline{An estimate on $ U_1 $:}
	\begin{align*}
		U_1(s) 
		&= \lb u, F(m_1) - F(m_2) \rb_{\L^2_w}(s) \\
		&= \lb u, v \la -Du + \gamma u \times Dm_1 + \gamma m_2 \times Du \ra - u \times D^2m_1 - m_2 \times D^2 u  \rb_{\L^2_w}(s) \\
		&\quad + \alpha \lb u, \lb D(m_1+m_2), Du \rb m_2 + D^2 u + |Dm_1|^2 u \rb_{\L^2_w}(s) \\
		&= \lb u, v \la -Du + \gamma m_2 \times Du \ra \rb_{\L^2_w}(s) + \lb u, -m_2 \times D^2 u \rb_{\L^2_w}(s) \\
		&\quad + \alpha \lb u, \lb D(m_1+m_2), Du \rb m_2 \rb_{\L^2_w}(s) + \alpha \lb D^2u, u \rb_{\L^2_w}(s) + \alpha \lb u, |Dm_1|^2 u \rb_{\L^2_w}(s).
	\end{align*}
	Then, for an arbitrary $ \epsilon>0 $,
	\begin{equation}\label{pf: <u, v(G1-G2)>}
	\begin{aligned}
		\lb u, v \la -Du + \gamma m_2 \times Du \ra \rb_{\L^2_w}(s)
		&= \lb u \rho_w^\frac{1}{2}, v(-Du + \gamma m_2 \times Du) \rho_w^\frac{1}{2}  \rb_{\L^2}(s) \\
		&\leq \frac{1}{2\epsilon^2} C_v^2 (1 + \gamma^2) |u(s)|_{\L^2_w}^2 + \epsilon^2 |Du(s)|_{\L^2_w}^2,
	\end{aligned}
	\end{equation}
	and
	\begin{equation}\label{pf: <u, -m2 x D2u>}
	\begin{aligned}
		\lb u, -m_2 \times D^2u \rb_{\L^2_w}(s) 
		&= - \lb D^2u, u \times m_2 \rho_w \rb_{\L^2}(s) \\
		&= \lb Du, Du \times m_2 \rho_w + u \times D(m_2 \rho_w) \rb_{\L^2}(s) \\
		&= \lb Du, u \times Dm_2 \rb_{\L^2_w} + \lb Du, u \times m_2 \rho_w' \rho_w^{-1} \rb_{\L^2_w}(s) \\
		&\leq \frac{1}{2 \epsilon^2} \la |Dm_2(s)|_{\L^\infty}^2+w^2 \ra |u(s)|_{\L^2_w}^2 + \epsilon^2 |Du(s)|_{\L^2_w}^2.
	\end{aligned}
	\end{equation}
	Similarly,
	\begin{equation}\label{pf: <u, D(m1+m2).Du m2>}
		\alpha \lb u, \lb D(m_1+m_2), Du \rb m_2 \rb_{\L^2_w}(s)
		\leq \frac{1}{2\epsilon^2}\alpha^2 |Dm_1(s) + Dm_2(s)|_{\L^\infty}^2 |u(s)|_{\L^2_w}^2 + \frac{1}{2} \epsilon^2 |Du(s)|_{\L^2_w}^2,
	\end{equation}
	and
	\begin{equation}\label{pf: <u,D2u>}
	\begin{aligned}
		\alpha \lb u, D^2u \rb_{\L^2_w}(s) 
		&= \alpha \lb u \rho_w, D^2u \rb_{\L^2}(s) \\
		&= -\alpha \lb Du, D(u \rho_w) \rb_{\L^2}(s) \\
		&= -\alpha |Du(s)|_{\L^2_w}^2 - \alpha \lb Du, u \rho_w' \rho_w^{-1} \rb_{\L^2_w}(s) \\
		&= -\alpha |Du(s)|_{\L^2_w}^2 + \frac{1}{2} \epsilon^2 |Du(s)|_{\L^2_w}^2 + \frac{1}{2 \epsilon^2} \alpha^2 w^2 |u|_{\L^2_w}^2.
	\end{aligned}
	\end{equation}
	Also,
	\begin{equation}\label{pf: <u,|Dm1|^2 u>}
		\alpha \lb u, |Dm_1|^2u \rb_{\L^2_w}(s) 
		\leq \alpha |Dm_1(s)|_{\L^\infty}^2 |u(s)|_{\L^2_w}^2.
	\end{equation} 
	Hence,
	\begin{align*}
		U_1(s) &\leq \psi_1(s) |u(s)|_{\L^2_w}^2 + (3 \epsilon^2 - \alpha) |Du(s)|_{\L^2_w}^2,
	\end{align*}
	for the process $ \psi_1 $ given by
	\begin{equation}\label{defn: psi1}
	\begin{aligned}
		\psi_1(s) 
		&= \frac{1}{2\epsilon^2} \la C_v^2(1+\gamma^2) + |Dm_2(s)|_{\L^\infty}^2 + w^2 + \alpha^2 |Dm_1(s) + Dm_2(s)|_{\L^\infty}^2 + \alpha^2 w^2 \ra \\
		&\quad + \alpha |Dm_1(s)|_{\L^\infty}^2.
	\end{aligned}
	\end{equation}
	For $ i=1,2 $, there exists a constant $ C>0 $ such that
	\begin{equation}\label{pf: Dm1, Dm2 Linf}
	\begin{aligned}
		\E^*\left[ \int_0^T |Dm_i|^2_{\L^\infty}(t) \ dt \right]
		&\leq \E^*\left[ \int_0^T |Dm_i|^2_{\H^1}(t) \ dt \right]
		< \infty,
	\end{aligned}
	\end{equation}
	which implies $ \int_0^T \psi_1(t) \ dt < \infty $, $ \P $-a.s.

	\underline{An estimate on $ U_2 $: }
	\begin{align*}
		U_2(s) 
		&= \frac{1}{2} \lb u, S(m_1) - S(m_2) \rb_{\L^2_w}(s) \\
		&= \frac{1}{2} \lb u, \kappa^2 \left[(1-\gamma^2) D^2u - 2\gamma (u \times D^2 m_1 + m_2 \times D^2 u) \right] \rb_{\L^2_w}(s) \\
		&\quad - \gamma^2 \lb u, \kappa^2 \left[ \lb D(m_1+m_2), Du \rb m_2 + |Dm_1|^2 u \right] \rb_{\L^2_w}(s) \\
		&\quad + \frac{1}{2} \lb u, \kappa \kappa' \left[ (1-\gamma^2) Du - 2 \gamma (u \times Dm_1 + m_2 \times Du) \right] \rb_{\L^2_w}(s) \\
		&= \frac{1}{2} \lb u, \kappa \kappa' \left[ (1-\gamma^2) Du - 2 \gamma m_2 \times Du \right] \rb_{\L^2_w}(s) \\
		&\quad + \frac{1}{2} (1-\gamma^2) \lb u, \kappa^2 D^2 u \rb_{\L^2_w}(s) - \gamma \lb u, \kappa^2 m_2 \times D^2 u \rb_{\L^2_w}(s) \\
		&\quad - \gamma^2 \lb u, \kappa^2 \lb D(m_1+m_2), Du \rb m_2 \rb_{\L^2_w}(s) - \gamma^2 \lb u, \kappa^2 |Dm_1|^2 u \rb_{\L^2_w}(s).
	\end{align*}
	Again, for $ \epsilon >0 $,
	\begin{align*}
		 &\frac{1}{2}\lb u, \kappa \kappa' \left[ (1-\gamma^2) Du - 2 \gamma m_2 \times Du \right] \rb_{\L^2_w}(s) \\
		 &\leq \frac{1}{4 \epsilon^2} C_\kappa^4 \la (1-\gamma^2)^2 + 4 \gamma^2 \ra |u(s)|_{\L^2_w}^2 
		 + \frac{1}{2}\epsilon^2 |Du(s)|_{\L^2_w}^2. 
	\end{align*} 
	As in \eqref{pf: <u, -m2 x D2u>} and \eqref{pf: <u, D(m1+m2).Du m2>},
	\begin{align*}
		-\gamma \lb u, \kappa^2 m_2 \times D^2 u \rb_{\L^2_w}(s) 
		&= -\gamma \lb D^2u, \kappa^2 u \times m_2 \rho_w \rb_{\L^2}(s) \\
		&= \gamma \lb Du, Du \times m_2 \kappa^2 \rho_w + u \times D(m_2 \kappa^2 \rho_w) \rb_{\L^2}(s) \\
		&= \gamma \lb Du, \kappa^2 u \times Dm_2 + \kappa^2 u \times m_2 \rho_w' \rho_w^{-1} + 2 \kappa \kappa' u \times m_2 \rb_{\L^2_w}(s) \\
		&\leq \frac{1}{2 \epsilon^2} \gamma^2 C_\kappa^4 \la |Dm_2(s)|_{\L^\infty}^2 + w^2 + 4 \ra |u(s)|_{\L^2_w}^2 + \epsilon^2 |Du(s)|_{\L^2_w}^2, 
	\end{align*}
	and
	\begin{align*}
		- \gamma^2 \lb u, \kappa^2 \lb D(m_1+m_2), Du \rb m_2 \rb_{\L^2_w} 
		&\leq \frac{1}{2\epsilon^2}\gamma^4 C_\kappa^4 |Dm_1(s) + Dm_2(s)|_{\L^\infty}^2 |u(s)|_{\L^2_w}^2 + \frac{1}{2} \epsilon^2 |Du(s)|_{\L^2_w}^2.
	\end{align*}
	Also, 
	\begin{align*}
		- \gamma^2 \lb u, \kappa^2 |Dm_1|^2 u \rb_{\L^2_w}(s) \leq 0, \quad \forall s \in [0,T]. 
	\end{align*}
	For the remaining term in $ U_2 $, we use integration-by-parts as in \eqref{pf: <u,D2u>}:
	\begin{align*}
		&\frac{1}{2} (1-\gamma^2) \lb u, \kappa^2 D^2 u \rb_{\L^2_w} \\
		&= \frac{1}{2}(\gamma^2-1) \lb Du, D(\kappa^2 u \rho_w) \rb_{\L^2} \\
		&= \frac{1}{2}(\gamma^2-1) \left[ \lb Du, \kappa^2 Du \rb_{\L^2_w}(s) + \lb Du, \kappa^2 u \rho_w' \rho_w^{-1} + 2 \kappa \kappa' u \rb_{\L^2_w}(s) \right] \\
		&\leq -\frac{1}{2} \sum_j q_j^2 |f_j Du(s)|^2_{\L^2_w} + \frac{1}{2}\gamma^2 C_\kappa^2 |Du(s)|^2_{\L^2_w} + \frac{1}{4\epsilon^2} C_\kappa^4 \la 1-\gamma^2 \ra^2(w^2+4) \ |u(s)|^2_{\L^2_w} \\
		&\quad + \frac{1}{2}\epsilon^2 |Du(s)|^2_{\L^2_w}.
	\end{align*}
	Thus, 
	\begin{align*}
		U_2(s) 
		&\leq \psi_2(s) |u(s)|_{\L^2_w}^2 
		+ \frac{5}{2}\epsilon^2 |Du(s)|_{\L^2_w}^2 + \frac{1}{2}\gamma^2 C_\kappa^2 |Du(s)|^2_{\L^2_w} 
		-\frac{1}{2} \sum_j q_j^2 |f_j Du(s)|^2_{\L^2_w},
	\end{align*}
	where
	\begin{equation}\label{defn: psi2}
	\begin{aligned}
		\psi_2(s)
		&= \frac{1}{4 \epsilon^2} C_\kappa^4 (1-\gamma^2)^2(w^2+5) \\
		&\quad + \frac{1}{2\epsilon^2} \la \gamma^2(w^2+6) + \gamma^2 |Dm_2(s)|_{\L^\infty}^2 + \gamma^4 |Dm_1(s) + Dm_2(s)|_{\L^\infty}^2 \ra,
	\end{aligned}
	\end{equation}
	and by \eqref{pf: Dm1, Dm2 Linf}, $ \int_0^T \psi_2(t) \ dt < \infty $, $ \P $-a.s.

	\underline{An estimate on $ U_3 $:}
	\begin{align*}
		U_3(s) 
		&= \frac{1}{2} \sum_j q_j^2 \left| f_j \la G(m_1) - G(m_2) \ra \right|_{\L^2_w}^2(s) \\
		&= \frac{1}{2} \sum_j q_j^2 \left| f_j \la -Du(s) + \gamma u(s) \times Dm_1(s) + \gamma m_2(s) \times Du(s) \ra \right|_{\L^2_w}^2 (s),
	\end{align*}
	where for every $ j \geq 1 $, 
	\begin{align*}
		&f_j^2 \left| G(m_1) - G(m_2) \right|^2(s,x) \\
		&= f_j^2 \left| -Du + \gamma u \times Dm_1 + \gamma m_2 \times Du \right|^2(s,x) \\
		&= f_j^2 \la |Du|^2 + \gamma^2 |m_2 \times Du|^2 + \gamma^2 |u \times Dm_1|^2 + 2 \gamma \lb -Du + \gamma m_2 \times Du ,u \times Dm_1 \rb \ra (s,x) \\
		&\leq (1+\gamma^2)|f_j Du(s,x)|^2 + q_j^2\gamma^2 \ |Dm_1(s)|^2_{\L^\infty} \ |f_j u(s,x)|^2 \\
		&\quad + \frac{2}{\epsilon^2} \gamma^2 \la 1+ \gamma^2 \ra |Dm_1(s)|^2_{\L^\infty} \ |f_j^2 u(s,x)|^2 + \epsilon^2 |Du(s,x)|^2.
	\end{align*}
	Hence, 
	\begin{align*}
		U_3(s) 
		&\leq \frac{1}{2} (1+\gamma^2) \sum_j q_j^2 |f_j Du(s)|^2_{\L^2_w} + \gamma^2  C_\kappa^2 \la \frac{1}{2} + \frac{1}{\epsilon^2}(1+\gamma^2) C_\kappa^2 \ra \ |Dm_1(s)|^2_{\L^\infty} |u(s)|^2_{\L^2_w} \\
		&\quad + \frac{1}{2} \epsilon^2 |Du(s)|^2_{\L^2_w} \\
		&\leq \psi_3(s) |u(s)|^2_{\L^2_w} + \frac{1}{2} \sum_j q_j^2 |f_j Du(s)|^2_{\L^2_w} + \frac{1}{2}\gamma^2 C_\kappa^2 |Du(s)|^2_{\L^2_w} + \frac{1}{2} \epsilon^2 |Du(s)|^2_{\L^2_w},
	\end{align*}
	where the second term on the right-hand side cancels with the corresponding term in $ U_2(s) $ and 
	$ \psi_3(s) =  \gamma^2  C_\kappa^2 \la \frac{1}{2} + \frac{1}{\epsilon^2}(1+\gamma^2) C_\kappa^2 \ra  |Dm_1(s)|^2_{\L^\infty} $ is similarly integrable $ \P $-a.s. 
	
	We have
	\begin{align*}
		U_1(s) + U_2(s) + U_3(s)
		&\leq \la \psi_1(s)+ \psi_2(s) + \psi_3(s) \ra |u(s)|_{\L^2_w}^2 
		+ \la 6\epsilon^2 + \gamma^2 C_\kappa^2  - \alpha \ra |Du(s)|_{\L^2_w}^2.
	\end{align*}
	We can choose a sufficiently small $ \epsilon > 0 $ such that under the assumption \eqref{defn: N1p, N2p>0}, 
	\begin{align*}
		\la 6 \epsilon^2 + \gamma^2 C_\kappa^2  - \alpha \ra < 0,
	\end{align*}
	which implies
	\begin{align*}
		U_1(s) + U_2(s) + U_3(s) 
		&\leq \la \psi_1(s)+ \psi_2(s) + \psi_3(s) \ra |u(s)|_{\L^2_w}^2 = \psi(s) \ |u(s)|^2_{\L^2_s}.
	\end{align*}
	Therefore, by \eqref{d|m1-m2|_Lw2},
	\begin{align*}
		\frac{1}{2}d|u(t)|_{\L^2_w}^2 \leq \psi(t) |u(t)|_{\L^2_w}^2 \ dt + \lb u(t), \left[ G(m_1(t)) - G(m_2(t)) \right] dW(t)  \rb_{\L^2_w}.
	\end{align*}	

	Define the process $ Y $ by
	\begin{align*}
		Y(t) := \frac{1}{2}|u(t)|_{\L^2_w}^2 e^{-2\int_0^t \psi(s) \ ds}, \quad t \in [0,T].
	\end{align*}
	Then, 
	\begin{align*}
		dY(t)
		&= \lb \frac{1}{2}d|u(t)|_{\L^2_w}^2, e^{-2\int_0^t \psi(s) \ ds} \rb
		+ \lb \frac{1}{2}|u(t)|_{\L^2_w}^2, de^{-2\int_0^t \psi(s) \ ds} \rb \\
		&\quad + \lb \frac{1}{2}d|u(t)|_{\L^2_w}^2, de^{-2\int_0^t \psi(s) \ ds} \rb \\
		&\leq e^{-2\int_0^t \psi(s) \ ds} \lb u(t), \left[ G(m_1(t)) - G(m_2(t)) \right] dW(t)  \rb_{\L^2_w}.
	\end{align*}
	Since $ |u(t)|_{\L^\infty} \leq 2 $ $ \P $-a.s. and 
	there exists a constant $ C $ such that 
	\begin{align*}
		\E \left[ \sup_{t \in [0,T]} \la |Dm_1(t)|_{\L^2}^2 + |Dm_2(t)|_{\L^2}^2 \ra \right] \leq C, 
	\end{align*}
	the process 
	\begin{align*}
		M(t) := \int_0^t e^{-2\int_0^s \psi(r) \ dr} \lb u(s), \left[ G(m_1(s)) - G(m_2(s)) \right] dW(s) \rb_{\L^2_w}.
	\end{align*}
	is a martingale, and then
	\begin{align*}
		\E[Y(t)] \leq Y(0) + \E[M(t)] = Y(0), \quad t \in [0,T]. 
	\end{align*}
	By the definition of $ Y(t) $, if $ Y(0) = m_1(0) - m_2(0) = 0 $, then 
	\begin{align*}
		|u(t)|_{\L^2_w}^2 = 0, \quad \P\text{-a.s.}
	\end{align*}
	for $ t \in [0,T] $,
	proving pathwise uniqueness of the solution of \eqref{eq_ito}. 
	By the Yamada-Watanabe Theorem, the uniqueness in law follows. 
	
\appendix
\section{}

\subsection{Some calculations in discrete spaces}
	\begin{enumerate}[(a)]
		\item 
		discrete integration-by-parts:
		\begin{align*}
			\sum_{x \in \Z_h} \langle u(x), \partial^{h} w(x) \rangle = - \sum_{x \in \Z_h} \langle \partial^{h} u^-(x), w(x) \rangle,
		\end{align*}
		for $ u,w \in \H_h^1 = \{ v \in \L_h^2 : |\partial^h v|_{\L_h^2} < \infty \} $ with appropriate decay properties. 
		In particular, 
		\begin{align*}
			\lb \partial^h u, \partial^h w \rb_{\L_h^2} = - \lb \Delta^h u, w \rb_{\L_h^2}.
		\end{align*}

		\item
		discrete expansion of $ \langle u, \Delta^h u \rangle $ and $ \langle u, \partial^h u \rangle $: for any $u$ satisfying that $|u(x)|=1$ for all $x\in \Z_h$,
		\begin{equation}\label{dis udotDelu}
		\begin{aligned}
			\langle u(x), \Delta^h u (x) \rangle 
			&= - \frac{1}{2} \left( |\partial^{h} u(x)|^2 + |\partial^{h} u^-(x)|^2 \right) \leq 0, \\
			\langle u(x), \partial^h u(x) \rangle 
			&= -\frac{h}{2} |\partial^h u(x)|^2 \leq 0.
		\end{aligned}
		\end{equation}
		
		\item
		product rule:
		\begin{align*}
			\partial^{h} (fu)
			&= (\partial^{h} f) u(x) + f^+ \partial^{h} u 
			= (\partial^{h} f) u^+ + f \partial^{h} u(x)
		\end{align*}
		for $ f $ scalar-valued and $ u $ vector-valued; 
		similarly for $ f $ and $ u $ both scalar-valued, and for $ \langle f, u \rangle $ and $ u \times u $ when $ f,u $ are vector-valued.
		
		\item
		$ L_h^2 $-norm of $ \Delta^h u $:
		\begin{align*}
			|\Delta^h u|_{\L_h^2} = |\partial^{h} (\partial^{h} u)^-|_{\L_h^2} = |\partial^{h}(\partial^{h} u)|_{\L_h^2}.
		\end{align*}
	\end{enumerate}

	\begin{lemma}[{\cite[Chapter 1, Theorem 3]{Zhou_FDbook}}]\label{Lemma: discrete interpolation ineq}
		For $ u^h: \Z_h \to \mathbb{R} $, 
		\begin{align*}
			\left| (\partial^{h})^k u^h \right|_{\L_h^p} \leq C |u^h|_{\L_h^2}^{1- \frac{k+ \frac{1}{2} - \frac{1}{p}}{n}}  \left| (\partial^{h})^n u^h \right|_{\L_h^2}^{\frac{k+\frac{1}{2} - \frac{1}{p}}{n}},
		\end{align*}
		for $ p \in [2,\infty] $, $ k \in [0,n) $ and $ C $ is a constant independent of $ u^h $.
	\end{lemma}
	
\subsection{Some tightness results}
	\begin{lemma}[{\cite[Theorem 2.1]{FlandoliGatarek}}]\label{Lemma: comp embedding in Lp}
		Let $ B_0 \subset B \subset B_1 $ be Banach spaces, $ B_0 $ and $ B_1 $ reflexive, with compact embedding of $ B_0 $ in $ B $. 
		Let $ p \in (1,\infty) $ and $ \alpha \in (0,1) $ be given. 
		Let $ X $ be the space 
		\begin{align*}
			X = L^p(0,T; B_0) \cap W^{\alpha,p}(0,T;B_1)
		\end{align*}
		endowed with the natural norm. Then the embedding of $ X $ in $ L^p(0,T; B) $ is compact. 
	\end{lemma}
	
	\begin{lemma}[{\cite[Theorem 2.2]{FlandoliGatarek}}]\label{Lemma: comp embedding in C}
		If $ B_1 \subset \widetilde{B} $ are two Banach spaces with compact embedding, and the real numbers $ \alpha \in(0,1) $, $ p>1 $ satisfy $ \alpha p > 1 $, 
		then the space $ W^{\alpha,p}(0,T; B_1) $ is compactly embedded into $ \mathcal{C}([0,T]; \widetilde{B}) $. 
	\end{lemma}
	
	\begin{lemma}[{\cite[Corollary 19]{Simon_sobolev}}]\label{Lemma: cont embedding between W}
		Let $ I $ be an either bounded or unbounded interval of $ \R $. 
		Let $ E $ be a Banach space. 
		Suppose $ s \geq r $, $ p \leq q $ and $ s- \frac{1}{p} \geq r - \frac{1}{q} $ for $ 0< r \leq s <1 $ and $ 1 \leq p \leq q \leq \infty $. 
		Then, 
		\begin{align*}
			W^{s,p}(I; E) \hookrightarrow W^{r,q}(I; E). 
		\end{align*}
	\end{lemma}
	
	In addition, we verify the continuous embedding 
	\begin{equation}\label{embed: Walpha,4 in Wbeta,2}
		W^{\alpha, 4}(0,T; \L^2_w) \hookrightarrow W^{\beta,2}(0,T; \L^2_w), \quad w \geq 1, \ \alpha \in \la \frac{1}{4}, \frac{1}{2} \ra,\  \beta = \alpha - \frac{1}{4}.
	\end{equation}
	Indeed, for $ u \in W^{\alpha,4}(0,T; \L^2_w) $, 
	\begin{align*}
		|u|_{W^{\beta,2}(0,T;\L^2_w)} 
		&= \left[ \int_0^T |u(t)|_{\L^2_w}^2 \ dt + \int_0^T \int_0^T \frac{|u(t)-u(s)|_{\L^2_w}^2}{|t-s|^{1+2\beta}} \ dt \ ds \right]^\frac{1}{2} \\
		&\leq \left[ \la \int_0^T |u(t)|_{\L^2_w}^4 \ dt\ra^{\frac{1}{2}} T^\frac{1}{2} + \la \int_0^T \int_0^T \frac{|u(t)-u(s)|_{\L^2_w}^4}{|t-s|^{2+4\beta}} \ dt \ ds \ra^{\frac{1}{2}} T \right]^\frac{1}{2} \\
		&\leq \left[ 2 T \int_0^T |u(t)|_{\L^2_w}^4 \ dt + 2 T^2 \int_0^T \int_0^T \frac{|u(t)-u(s)|_{\L^2_w}^4}{|t-s|^{2+4\beta}} \ dt \ ds  \right]^{\frac{1}{4}} \\
		&\leq \la 2 \max\{ T, T^2 \} \ra^\frac{1}{4} \left[ \int_0^T |u(t)|_{\L^2_w}^4 \ dt + \int_0^T \int_0^T \frac{|u(t)-u(s)|_{\L^2_w}^4}{|t-s|^{1+4\alpha}} \ dt \ ds \right]^\frac{1}{4} \\
		&= \la 2 \max\{ T, T^2 \} \ra^\frac{1}{4} |u|_{W^{\alpha,4}(0,T; \L^2_w)}, 
	\end{align*}
	where the second inequality holds by $ \sqrt{a} + \sqrt{b} \leq \sqrt{2a + 2b} $ for $ a,b \geq 0 $.


\end{document}